\documentclass[11pt]{article}
	\usepackage{setspace}
\usepackage{amssymb,amsmath,amsthm, bbm, mathtools} 
\usepackage{float} 
\usepackage{graphicx, caption, subcaption} 
\usepackage{algorithm, algpseudocode} 
\usepackage[colorinlistoftodos,prependcaption,textsize=scriptsize]{todonotes}
\usepackage[final]{hyperref}
\usepackage{enumitem} 

\usepackage{tikz, pgfplots}
\usetikzlibrary{decorations.markings}

\usetikzlibrary{shapes.geometric}
\usetikzlibrary{patterns}

\usepackage[round, authoryear]{natbib}
\newcommand{\supnorm}[1]{\vert\vert #1 \vert\vert_{\infty}}
\newcommand{\E}{\mathbb{E}}
\newcommand{\dx}{\mathrm{d}}

\newtheorem{theorem}{Theorem}
\newtheorem{lemma}[theorem]{Lemma}
\newtheorem{proposition}[theorem]{Proposition}

\newtheorem{applemma}{Lemma}[section]
\newtheorem{appproposition}{Proposition}[section]

\newtheorem{example}{Example}
\newtheorem{definition}{Definition}
\newtheorem{remark}{Remark}

\newtheorem*{remark*}{Remark}

\newtheorem{corollary}{Corollary}

\usepackage{authblk}
\title{Polling on a circle with non-uniform batch arrivals}
\date{\today}

\begin{document}
\author[1,*]{Tim Engels}
\author[2]{Ivo Adan}
\author[1]{Onno Boxma}
\author[1]{Jacques Resing}

\affil[1]{Department of Mathematics and Computer Science, Eindhoven University of Technology}
\affil[2]{Department of Industrial Engineering and Innovation Sciences, Eindhoven University of Technology}
\affil[*]{Corresponding author: t.p.g.engels@tue.nl}
\maketitle

\begin{abstract}
    In this paper, we analyze a polling system on a circle with. Random batches of customers arrive at a circle, where each customer, independently, obtains a location according to a general distribution. A single server cyclically travels over the circle to serve all customers. We analyze the experienced delay of batches for two service policies: globally gated and exhaustive. The Laplace-Stieltjes transform of the experienced delay is found under the former policy. For the latter policy, we propose a mean-value analysis, resulting in an algorithmic approach for the evaluation of the mean experienced delay. Light- and heavy-traffic limits are derived exactly for the system performance.
\end{abstract}

\section{Introduction}
Polling systems are queuing models in which the servers attend to multiple queues. In standard polling models, a single server cyclically visits the queues and serves the customers in these queues. The server typically incurs a so-called switch-over time to go from one queue to another. Ample research is devoted to the analysis and optimisation of polling systems, highlighted by the extensive literature reviews of \citet{takagi2001} and 
\citet{Borst2018}. The popularity of polling systems lies within their wide range of applications that come from instances where several customers or customer types compete for the service provided by one or more servers.
Examples include computer communication systems \citet{Grillo1990, Altman2012},  traffic systems \citet{Webster1958} and warehouse logistics \citet{Gong2008}. In this paper we analyse the average performance of continuous polling system on a circle with batch arrivals and non-uniform arrival locations.

The analysis is motivated by the following application of warehouse logistics, but can also be applied to other examples, like in ferry-assisted wireless networks \citep{Altman2012}. Milkrun systems are order picking systems where one or more pickers continuously walk through the entire warehouse, picking certain items they encounter \citet{Gong2008, Gaast2019}. Unlike traditional picking systems, the picker can respond to real time information and does not change her route based on the requested products. By doing so, milkrun systems achieve higher responsiveness, which improves efficiency when the warehouse receives many small orders. \\
Since the picker in a milkrun system walks a fixed route, she visits the product locations cyclically, similar to polling models. More specifically, one can see the picker as a server and each item location as a queue. A requested item can thus be seen as a customer and an order as a batch of customers. The arrival location distribution represents the storage policy: where high- and low-demand items are located. Hence, a milkrun system can be modelled as a single-server polling system with batch arrivals and general arrival locations. Many warehouses store a large variety of products, and thus it is natural to approximate these systems by a continuous counterpart, where item locations are no longer discrete, but continuous. In case of uniform arrival locations, \citet{Engels} show that this approximation yields accurate results, even for a relatively small number of different products.

Continuous polling models are an extension of the standard (discrete) polling model. Instead of visiting a fixed number of queues, the server attends to the needs of customers located on a circle. First introduced by \citet{Fuhrmann1985}, these models have gained considerable attention. Notable mentions include the analyses in \cite{Coffman1986} and \cite{Kroese1992}. The analysis of continuous polling systems can be quite involved, yet often also results in explicit expressions for the mean performance metrics, see for instance \citep{Engels}. Going beyond the mean, however, tends to considerably complicate the analysis. One can think of these continuous polling models as limiting cases of their discrete counterpart, where one takes the number of queues to infinity, as proven by \citet{Eliazar2003, Eliazar2005}. Consequently, the continuous model is proven to be an appropriate approximation of its discrete counterpart, especially when the application concerns many queues.

Batch arrivals in polling systems (or related models) have been studied in several papers. \citet{Boxma1986} derive a pseudo-conservation law for a class of queueing models with batch arrivals and vacations, including polling models. \citet{LevySidi} introduce a framework in which batches of customers arrive, with, possibly, correlated arrival locations. The authors provide an algorithmic approach to find the mean waiting time of a customer. \Citet{Mei} extends the research further and finds an exact expression for the waiting time. \citet{Gaast2017}, instead, focus on the sojourn time of a batch, rather than that of a customer. Using a mean-value approach, the authors derive the mean batch sojourn time in the polling system. In a recent paper, \citet{Engels} combine the analysis of continuous polling models with batch arrivals, under the assumption that customers arrive uniformly on the circle, and derive analytical expressions for the mean performance of the system. 

In the current paper, we relax the assumption about the arrival locations of customers. Instead of assuming uniform arrival locations, we allow customers to be located on the circle according to a predetermined distribution, reflecting the storage policy in the warehousing application. Unlike \cite{Kroese1993}, we do not assume that these locations are relative to the server. The current paper is devoted to the analysis of (i) the long-run behaviour of the sojourn time of a batch, that is the time between the arrival of the batch and the service completion of the last customer in that batch and (ii) the time to delivery of a batch of customers, that is the time between the arrival of the batch and the delivery of the batch at a central point.  We consider two picking policies: globally gated and exhaustive. The main contributions of this paper are three-fold. (i) We derive exact expressions for the distribution of the batch sojourn time and time to delivery under the globally gated picking policy. (ii) Using a mean-value approach, we derive an iterative algorithm for evaluating the mean performance of the system under the exhaustive service policy and (iii) we find explicit expressions under light- and heavy-traffic limits for both service policies. Using these results, we investigate the system performance under several storage policies. Remarkably, we show that the storage policy only barely affects the average performance of milkrun systems.

This paper is organised as follows. We start with a formal definition of the polling model in Section \ref{sec:model}. Afterwards, we analyse the system under the globally gated policy (Section \ref{sec:GG}), resulting in exact expressions for the Laplace-Stieltjes transform and mean of both batch sojourn time and time to delivery. In Section \ref{sec:EX} we propose a mean-value analysis of the continuous polling model under exhaustive service. The obtained results are used to find the expected batch sojourn time and time to delivery in Sections \ref{sec:sojourn} and \ref{sec:del} respectively. Section \ref{sec:limit} is devoted to the analysis of light- and heavy-traffic behaviour of the polling model. Section \ref{sec:numericalresults} contains several numerical results, allowing a numerical comparison of globally gated and exhaustive service and of several storage policies. In that section we also discuss the application of our results to the milkrun system. Here, we will see that the storage policy plays no real role in the performance of the system. Appendices \ref{app:proofs} and \ref{app:Deliver} concern the proofs of Sections \ref{sec:sojourn} and \ref{sec:del}. In appendix \ref{app:numerical}, we describe and analyse an algorithm that are used for the numerical results.

\section{Model description}
\label{sec:model}
Consider a single-server polling system in which batches of customers arrive at a circle (with circumference 1) according to a Poisson process with intensity $\lambda$. The sizes of subsequent customer batches, $K_1, K_2,...$, are assumed to be independent and identically distributed with probabilities $p_k= \mathbb{P}(K_1=k), \, k\geq 1$. We write $\tilde{K}(\cdot)$ for the probability generating function of the generic batch size $K$.
The locations on the circle are parametrised to $[0,1)$ and denote the clockwise distance from an arbitrary starting point, referred to as the depot, to said location. Throughout this paper, we denote $S$ as the server's location.\\
Each customer in a batch is independently assigned a location according to a known continuous distribution with density $\pi(x)$, $0 \leq x \leq 1$. As customer locations are continuous, no queues will form on the circle. Throughout this paper, we further assume that $\pi(x) < \infty$ for all $0\leq x \leq 1$.\\
A single server is located on the circle and travels at a fixed speed in a clockwise direction, taking a time $\alpha$ to traverse the entire circle. The server stops moving whenever she serves a customer. The service times of consecutive customers are denoted as $B_1, B_2,...$ and are assumed to be independent and identically distributed. We write $B$ for the generic service time, with Laplace-Stieltjes transform (LST) $\phi_B(\cdot)$. A cycle (of the server) is defined to be the period between two subsequent depot crossings, and its length is denoted by $C$. We let $S^B$ denote the batch-sojourn time and $D$ denote the time to delivery of a batch. \\
Which customers the server serves in a cycle depends on the discipline we consider. Under the globally gated discipline, the server only serves those customers that were already present in the system at the start of the cycle. Under the exhaustive discipline, the server serves all customers she encounters. 

The stability of related systems is studied in several papers. In exhaustive and globally gated polling systems, with a fixed number of queues, the systems are stable under $\rho := \lambda \E[K] \E[B] < 1$, cf. \citet{Boxma1986, Gaast2017}. The continuous variant, discussed in \citet{Kroese1992, Kroese1993} also is known to be stable for $\rho < 1$. The proofs of those results can be extended to the current model.\\
For an alternative proof that $\rho<1$ is a sufficient stability condition in the systems under consideration in the present paper, one could use a limiting argument from \citet{Eliazar2005}. By generalising this work to allow for batch arrivals, one can prove that this continuous polling model is the limiting case of a discrete polling model, where one takes the number of queues to infinity. This also immediately shows that the long-run average cycle-length satisfies: $\E[C] = \alpha/(1-\rho)$ and that the distribution of the server's location satisfies:
\begin{align}
    \label{eq:serverlocation}
    \lim_{\dx y \downarrow 0}\frac{\mathbb{P}(S\in [y, y+ \dx y])}{\dx y} = \big[\rho \pi(y) + 1-\rho \big].
\end{align}

\section{Globally Gated}
\label{sec:GG}
In this section, we focus on the globally gated service discipline and derive the LST of both the time to delivery and the batch sojourn time. For this, we start with an initial analysis of the cycle time, and derive expressions for its LST, as well as expressions for the LST of the residual cycle time and age of the cycle time. In turn, these are used for the derivation of the LSTs of the performance measures $S^B$ and $D$.

Remark that the length of the $n$-th cycle, $C_n$, is given by the travel time $\alpha$ plus the service time of all customers that arrived during the previous cycle, $C_{n-1}$.
Denote the density of $C_n$ by $f_{C_n}(\cdot)$ and its LST by $\phi_{C_n}(\cdot)$. Let $A(t)$ denote the number of arrival instances that occur in a time interval of length $t$ and let $B_{i,j}$ denote the service time of the $j$-th customer in the $i$-th batch that arrived during this time. Then the LST of $C_n$ satisfies:
\begin{align}
    \phi_{C_n}(\omega) := \mathbb{E}\big[\exp(-\omega C_n)\big] &= \int_{t=0}^\infty \mathbb{E}\big[\exp(-\omega C_n)\vert C_{n-1} = t\big]f_{C_{n-1}}(t) \dx t\nonumber\\
    &= \int_{t=0}^\infty \exp(-\omega\alpha)\mathbb{E}\left[\exp\left(-\omega\cdot\sum_{i=1}^{A(t)}\sum_{j=1}^{K_i}B_{i,j}\right)\right]f_{C_{n-1}}(t) \dx t\nonumber .
\intertext{Since the batches arrive according to a Poisson process, we have:}
\label{eq:phiCEQ1}
    \phi_{C_n}(\omega) &=  \int_{t=0}^\infty \exp(-\omega\alpha)\exp\left(-\lambda t\left[1-\tilde{K}\big(\phi_{B}(\omega)\big)\right]\right)f_{C_{n-1}}(t) \dx t\nonumber\\
    &= \exp(-\omega \alpha) \phi_{C_{n-1}}\left(\lambda\left[1-\tilde{K}\big(\phi_{B}(\omega)\big)\right]\right).
\end{align}
Taking the limit for $n\to \infty$ on both sides now gives a functional equation for the LST $\phi_C(\cdot)$ of the cycle time. Similar results, in the case of discrete polling models with the globally gated service discipline, can be found in \citet{Boxma1992}. It now follows, by repeated application of \eqref{eq:phiCEQ1}, that the LST of the cycle time can be written as an infinite product. Let $\delta_0(\omega) = \omega$, $\delta_1(\omega) = \lambda\big[1-\tilde{K}\big(\phi_B(\omega)\big)\big]$ and $\delta_{i+1}=\delta_1\big(\delta_{i}(\omega)\big), 
 i\geq 1$. Then:
\begin{align}
    \phi_{C}(\omega) = \prod_{i=0}^\infty \exp\big(-\alpha \delta_{i}(\omega)\big) = \exp\left(- \alpha \sum_{i=0}^{\infty} \delta_i(\omega)\right).
\end{align}
Observe that $\sum_{i=0}^{\infty} \delta_i(\omega) < \infty$ because $\rho < 1$. Using that $\tilde{K}(x) > 1 + \E[K](x-1)$ and $\phi_B(x)\geq 1-\E[B] x$, we indeed have that $|\delta_i(\omega) | \leq \lambda \E[K] |1-\phi_B(\delta_{i-1}(\omega))| \leq \rho |\delta_{i-1}(\omega)|$.\\
The LST of the joint distribution of the age of the cycle time, $C_P$, and residual cycle time $C_R$, that an arbitrary arriving batch encounters, satisfies the following (cf.\ Section 3.1 of \citet{Fralix2009}): 
\begin{align}
    \mathbb{E}\left[\exp(-\omega_P C_P - \omega_R C_R)\right] = \frac{1}{\mathbb{E}[C]}\frac{1}{\omega_P - \omega_R}\big[\phi_C(\omega_R) - \phi_C(\omega_P)\big].
    \label{ref3}
\end{align}
This will form an integral part for the derivation of the batch sojourn time and time to delivery. Call $C^{*} := C_P + C_R$, the length of the cycle during which an arbitrary batch of customers arrives. Note that the above relation gives:
\begin{align}
    \mathbb{E}\left[\exp(-\omega C^*)\right] &= \mathbb{E}\left[\exp(-\omega C_P - \omega C_R)\right]\\
    &= \lim_{\omega_P\to \omega}\frac{1}{\mathbb{E}[C]}\frac{1}{\omega_P - \omega}\big[\phi_C(\omega) - \phi_C(\omega_P)\big] = -\frac{1}{\mathbb{E}[C]}\phi_C'(\omega).
\end{align}
It now follows that $\E[C^*] = \E[C^2]/\E[C]$ and $\E[C_P] = \E[C_R] = \E[C^2]/(2\E[C])$, where $\E[C^2]$ can be found by differentiating \eqref{eq:phiCEQ1} twice and solving for $\E[C^2]$:
\[\E[C^2] = \frac{1}{1-\rho^2}\cdot \Big\{\alpha^2 + 2\rho\alpha\E[C] + \lambda\E[K]\E[B^2]\E[C] + \lambda\E[B]^2\E[K(K-1)]\E[C]\Big\}.\]\\
Remark that $C, C^*, C_R$ and $C_P$ are all independent of the arrival location density $\pi$. \\
We can now use these expressions for the LSTs of $C^*, C_R, C_P$ to derive the distribution of both the time to delivery and batch sojourn time. In both derivations, we use that the customers have to wait for the residual cycle time, plus extra work that has arrived/will arrive during the current cycle.

\subsubsection*{Time to delivery}
The time to delivery of a batch consists of the residual cycle time, the travel time of a cycle and the service times of \emph{all} customers who arrive(d) during the same cycle as the tagged batch. Using this observation, the derivation of the LST of the time to delivery immediately follows:

\begin{theorem}
    The time to delivery, $D$, in the polling model under the globally gated service discipline has the following Laplace-Stieltjes Transform:
    \begin{equation}
    \label{eq:GG_LSTD}
        \begin{aligned}
            \phi_D(\omega)&=\tilde{K}\big(\phi_B(\omega)\big)\exp(-\omega\alpha)\frac{1}{\omega\mathbb{E}[C]}\\
             &\;\;\;\cdot\Big[ \phi_C\Big(\lambda - \lambda \tilde{K}\big(\phi_{B}(\omega)\big)\Big) - \phi_C\Big(\omega+\lambda - \lambda \tilde{K}\big(\phi_{B}(\omega)\big)\Big)\Big].
        \end{aligned}
    \end{equation}
\end{theorem}
\begin{proof}
     The time to delivery of a tagged batch is composed of the following elements: (i) the residual cycle time, $C_R$, (ii) the travel time of an entire cycle, $\alpha$, (iii) the service time of the customers in the tagged batch, say $K_0$ customers arrive in this batch, and (iv) the service time of all other customers who arrived during $C^*$. A total of $A(C^*)$ arrivals occur in $C^*$, and each of these arrivals contains a batch of customers, say of size $K_j$ for the $j$-th arriving batch. Let $B_{0,j}$ denote the service time of the $j$-th customer in the tagged batch and $B_{i,j}$ the service time of the $j$-th customer in the $i$-th batch that also arrived during $C^*$. Then, with $\overset{d}{=}$
     denoting equality in distribution,
    \begin{equation}
    \label{eq:GG_distD}
        D \overset{d}{=} C_R+\sum_{j=1}^{K_0}B_{0,j} + \alpha + \sum_{i=1}^{A(C^*)}\sum_{j=1}^{K_i}B_{i,j}.
    \end{equation}
    Let  $f_{C_P,C_R}(\cdot,\cdot)$ denote the joint density of $C_P$ and $C_R$ and condition over the size of the past and residual cycle time. Applying \eqref{ref3} then gives:
    \begin{align*}
         \mathbb{E}\left[\exp(-\omega D)\right] &= \tilde{K}\big(\phi_B(\omega)\big)\exp(-\omega\alpha)\cdot\nonumber \\
         &\quad\quad \bigg\{\int_{t=0}^\infty \int_{s=0}^\infty \exp\left(- \omega t - (s+t)\Big[\lambda - \lambda \tilde{K}\big(\phi_{B}(\omega)\big)\Big]\right)f_{C_P,C_R}(s,t)\dx s \dx t\bigg\}\nonumber\\
         &=\tilde{K}\big(\phi_B(\omega)\big)\exp(-\omega\alpha)\frac{1}{\omega\mathbb{E}[C]}\nonumber\\
         &\;\;\;\cdot\Big[ \phi_C\Big(\lambda - \lambda \tilde{K}\big(\phi_{B}(\omega)\big)\Big) - \phi_C\Big(\omega+\lambda - \lambda \tilde{K}\big(\phi_{B}(\omega)\big)\Big)\Big].\qedhere
    \end{align*}
\end{proof}
\begin{corollary}
\label{cor:GGED}
    The mean time to delivery under the globally gated discipline is given by:
    \begin{equation}
        \begin{aligned}
            \E[D] &= \E[B]\E[K] + \alpha + \E[C_R] + \lambda\E[B]\E[K]\E[C^*]\\
            &=\E[B]\E[K] + \alpha + (1+2\rho)\E[C_R].
        \end{aligned}
    \end{equation}
\end{corollary}
\begin{proof}
    This is a direct consequence of \eqref{eq:GG_distD}. Alternatively, one can derive this by evaluating the derivative of \eqref{eq:GG_LSTD} at $\omega = 0$.
\end{proof}

\begin{remark}
\label{rem:GG_ED}
    The arrival location distribution, $\pi$ does not affect the distribution of the time to delivery in the polling system under the globally gated policy. This is because the arrival locations do not affect the cycle times, and therefore does not affect the time to delivery.
\end{remark}

\subsubsection*{Batch sojourn time}
The approach to the batch sojourn time is similar, where we now only have to consider {customers} that arrive ahead (in location) of the last customer in the batch. Therefore, we condition over the arrival location of the furthest customer (measured in the clockwise distance from the depot) in the tagged batch. Let $\Pi(x)$ denote the cumulative distribution function of the arrival locations and $X^B$ the arrival location of the furthest customer in the tagged batch, then:
\begin{align*}
    \mathbb{E}\left[\exp(-\omega S^B)\right] = \int_{x=0}^1\sum_{k}p_k k\pi(x) \Pi(x)^{k-1} \mathbb{E}\left[\exp(-\omega S^B)\vert K = k, X^B = x\right]\dx x.
\end{align*}
The conditional distribution of the batch sojourn time is now quite similar to the distribution of the time to delivery. Let $X_{i,j}$ denote the arrival location of the $j$-th customer in the $i$-th batch arriving in the same cycle as the tagged customer batch. Further, let $K_i$ again denote the size of the $i$-th batch arriving in the same cycle as the tagged batch, of which the $j$-th customer has a service time $B_{i,j}$. Then we can write for the conditional batch sojourn time:
\begin{equation}
\label{eq:GG_distS}
        S^B\Big\vert \{K_0, X^B\} \overset{d}{=} C_R+\sum_{j=1}^{K_0}B_{0,j} + \alpha X^B + \sum_{i=1}^{A(C^*)}\sum_{j=1}^{K_i}B_{i,j}\mathbbm{1}\big\{X_{i,j}\leq X^B\big\}.
\end{equation}
Using a similar approach to the time to delivery now shows:
\begin{theorem}
    The batch sojourn time, $S^B$, in the polling model under the globally gated service discipline has the following Laplace-Stieltjes Transform:
    \begin{equation}
    \label{eq:GG_LSTS}
        \begin{aligned}
            \phi_{S^B}(\omega)= \frac{\phi_B(\omega)}{\omega\mathbb{E}[C]}\int_{x=0}^1 \bigg\{&\pi(x)\tilde{K}'\Big(\Pi(x)\phi_B(\omega)\Big)\exp(-\omega\alpha x) \\
     &\cdot\Big[ \phi_C\Big(\lambda - \lambda\tilde{K}\big(1-\Pi(x) + \Pi(x)\phi_{B}(\omega)\big)\Big) \\
     &\quad- \phi_C\Big(\omega+\lambda - \lambda\tilde{K}\big(1-\Pi(x) + \Pi(x)\phi_{B}(\omega)\big)\Big)\Big]\bigg\}\dx x.
        \end{aligned}
    \end{equation}
\end{theorem}
\begin{proof}
    First remark that the distribution of the number of customers in a batch, arriving in front of $X^B=x$, has the following probability generating function:
\begin{align*}
    \mathbb{E}\left[z^{\sum_{j=1}^{K_i}\mathbbm{1}\big\{X_{i,j}\leq x\big\}}\right] &=  \sum_{k=1}^\infty p_k\mathbb{E}\left[z^{\mathbbm{1}\big\{X\leq x\big\}}\right]^k \\
    &= \tilde{K}\Big(1-\Pi(x) + \Pi(x)z\Big).
\end{align*}
Using this, and following the steps of the derivation for the time to delivery, we conclude:
\begin{align}
     \mathbb{E}\Big[\exp(-\omega S^B)\vert K = k, X^B = x\Big]&= \phi_B(\omega)^k\exp(-\omega\alpha x)\frac{1}{\omega\mathbb{E}[C]} \nonumber\\
     &\quad\cdot \Big[ \phi_C\Big(\lambda - \lambda\tilde{K}\big(1-\Pi(x) + \Pi(x)\phi_{B}(\omega)\big)\Big) \\
     &\quad\quad- \phi_C\Big(\omega+\lambda - \lambda\tilde{K}\big(1-\Pi(x) + \Pi(x)\phi_{B}(\omega)\big)\Big)\Big] \nonumber.
\end{align}
Deconditioning with respect to $X^B$ and $K$, using $\mathbb{P}(X^B \leq x) = \Pi(x)^{k}$, now results in the final expression.
\end{proof}
\begin{corollary}
\label{cor:GGSB}
    The mean batch sojourn time under the globally gated discipline is given by:
    \begin{equation}
            \E[S^B] =\E[B]\E[K] + \E[C_R] + \alpha - \alpha\int_{x=0}^1 \tilde{K}(\Pi(x))\dx x  + \rho\E[C^*]\E\left[\frac{K}{K+1}\right]. 
    \end{equation}
\end{corollary}
\begin{proof}
    This result follows from \eqref{eq:GG_distS}, where we use that $A(C^*), K_i$ and $B_{i,j}$ are independent from $X_{i,j}$ and $X^B$. The expected distance to the furthest customer in a batch, from the depot, is given by:
    \begin{align*}
        \E[X^B] = \int_{x=0}^1 x\cdot \sum_{k=1}^\infty p_k\cdot k \cdot \pi(x)\Pi(x)^{k-1}\dx x = \int_{x=0}^1 x\pi(x)\tilde{K}'\left(\Pi(x)\right)\dx x.
    \end{align*}
    By partial integration, we find that the expected travel time to the furthest customer equals:
    \begin{align*}
        \alpha\E[X^B] = \alpha - \alpha\int_{x=0}^1 \tilde{K}(\Pi(x))\dx x.
    \end{align*}
    We further note that:
    \begin{align*}
        \mathbb{P}(X_{i,j}\leq X^B) &= \int_{x=0}^1 \sum_k k p_k \pi(x)\Pi(x)^{k} \dx x = \int_{x=0}^1 \pi(x)\Pi(x)\tilde{K}'\big(\Pi(x)\big) \dx x.
    \end{align*}
    The proof now follows from the substitution of $u = \Pi(x)$ in the integral, combined with the observation that:
    \begin{align*}
        \int_{u=0}^1 u\tilde{K}'(u) \dx u = 
        \sum_{k=1}^{\infty} p_k \int_{u=0}^1 k u^k \dx u = \E\left[\frac{K}{K+1}\right].
    \end{align*}
    One can also prove the corollary by evaluating the derivative of the LST in \eqref{eq:GG_LSTS} at $\omega = 0$.
\end{proof}

\begin{remark}
    The expected batch sojourn time is solely affected by the arrival location distribution through the expected travel time to the furthest customer in a batch. All other terms remain unaffected. One can minimize the expected batch sojourn time by locating all customers just behind the depot.
\end{remark}

\begin{remark}
\label{rem:GG_SB}
    The difference $\E[D] - \E[S^B]$ consists of the extra travel distance from the furthest customer in a batch to the depot, and the service of all extra customers that the server encounters during this travelling. Corollaries \ref{cor:GGED} and \ref{cor:GGSB} show that this is in expectation given by:
    \begin{align*}
        \E[D]- \E[S^B]  
        &= \alpha \int_{x=0}^1 \tilde{K}(\Pi(x)) \dx x + \rho \E[C^*] \E\left[\frac{1}{K+1}\right].
    \end{align*}
    This difference is thus (relatively) large when either $\pi$ has most of its mass at the beginning of the circle, or the sizes of customer batches are small. These cases coincide with instances where the batch sojourn time is small.
\end{remark}

\section{Exhaustive}
\label{sec:EX}
The analysis of the system under the exhaustive service discipline is more complex on both a conceptual and a technical level. Since the server now serves any customers she encounters, we cannot apply similar methods as in the previous section. Instead of deriving the LST of the performance measures, we now consider their average behaviour. We extend the work in \citet{Engels}, where only uniform arrival locations are considered, and derive a mean-value analysis of the system with the exhaustive service discipline. 

In the current paper, the mean-value approach focuses on the average ``spread'' of customers on the circle. We do this by analysing $L(A)$, the number of waiting customers with locations within a set $A$. In particular, we focus on this variable conditional on the server's location. We define the following:
\begin{align}
\label{eq:deff}
    f(x,y) &:= \lim_{\delta \to 0} \frac{\mathbb{E}\big[L\big([x,x + \delta]\big)\big\vert S = y\big]}{\delta}.
\end{align}
Intuitively, $f(x,y)\dx x$ represents the average number of waiting customers in the interval $[x,x+\dx x]$ when the server is at location $y$ (where $x,y$ denote distances from the depot). The function, $f(x,y)$, can also be seen as a measure of the \emph{average density} of customers at a certain location. 

In the derivations in this section, it will prove to be useful to extend the definition of integration over the circle in the following way for $a,b \in [0,1]$:
\begin{align}
    {\int_{x=a}^b}^{*} f(x) \dx x := \begin{dcases}
         \int_{x=a}^b f(x) \dx x & \text{if: } a\leq b\\
         \int_{x=a}^1 f(x) \dx x + \int_{x=0}^b f(x) \dx x &\text{if: } a > b.
    \end{dcases}
\end{align}
Specifically, this definition can be used to define the clockwise distance between two locations:
\begin{align*}
    d(a,b) :=   {\int_{x=a}^b}^{*} \dx x = \begin{dcases}
         b-a & \text{if: } a\leq b\\
         1 - a + b &\text{if: } a > b.
    \end{dcases}
\end{align*}

The remainder of this section is built up as follows. In Section \ref{sec:waiting} we derive the expected total number of waiting customers in the polling system. Afterwards, we turn our attention to the average spread of customers, dictated by $f$, and prove that this function adheres to an integral equation, see Section \ref{sec:fpe}. In Section \ref{sec:partial} we find a partial solution to the integral equation, and in Section \ref{sec:successive} we prove that the entire solution can be found by means of successive substitution.

\subsection{Expected number of waiting customers}
\label{sec:waiting}
We start with an analysis of the average number of waiting customers in the system, $\E[L]$, i.e. the expected number of customers whose service has not yet started. We couple the current system to a similar system that serves the customers in a First-Come-First-Serve (FCFS) manner. This system no longer assigns a location to the customer. Instead, the server keeps track of all locations she has to visit. Whenever the server reaches one of these locations, she starts serving the longest waiting customer from the entire system.\\
Remark that both the service times and travel times between consecutive service starts are the same in the polling system and the FCFS variant, only the order of service is altered. As a direct consequence, the numbers of customers in the systems are equal. Analysing the First-Come-First-Server model arguably is easier, as the waiting time of a customer in this model does not depend on future arrivals. 

\begin{lemma}
\label{lemma:ncust}
   The average number of waiting customers in the polling model under the exhaustive service discipline is given by:
   \begin{equation}
       \E[L] = \frac{\lambda \E[K]}{2(1-\rho)}\Big(\alpha + \rho\frac{\E[B^2]}{\E[B]} + \frac{\E[B]\E[K(K-1)]}{\E[K]}\Big).
   \end{equation}
\end{lemma}
\begin{proof}
    We consider the related FCFS system, and denote the variables related to this system with a superscript $^{\rm FCFS}$. Remark that the waiting time of a tagged customer under the FCFS model is given by the sum of: (i) the travel time, $T$, to the tagged customer; (ii) the residual service time of the customer in service, if any, at the arrival instant of the tagged customer; (iii) the service time of all waiting customers present at the time of arrival and (iv) the service time of all customers arriving in the same batch, who are served before the tagged customer. Hence:
    \[\E\left[W^{\rm FCFS}\right] = \E[T] + \rho\frac{\E[B^2]}{2\E[B]} + \E\left[L^{\rm FCFS}\right]\E[B] + \frac{\E[B]\E[K(K-1)]}{2\E[K]}.\]
    One can now apply Little's law, $\E[L^{\rm FCFS}] = \lambda\E[K] \E[W^{\rm FCFS}]$ to find a first expression for the mean number of waiting customers in the system. It remains to find the average travel time $\E[T]$ to a waiting customer.\\
    The travel time under the FCFS service policy is hard to analyse, as this depends on the state of the system at the arrival instant, as well as future arrivals. Instead, we remark that the \emph{average} travel time to a tagged customer is the same for the FCFS model and polling model. This is due to the fact that the sum of the travel times to all customers is the same. We can therefore focus on the original system to derive $\E[T]$. Recall \eqref{eq:serverlocation}, stating that the server is within a small interval $[x, x + \dx x]$ with probability $[\rho \pi(x) + 1-\rho]\dx x$ . Therefore, we have:
    \[
        \E[T] = \alpha \int_{x=0}^1 \big[\rho \pi(x) + 1-\rho]\E[d(x, X_1)]\dx x = \alpha\rho \E[d(X_2, X_1)] + \alpha(1-\rho)\E[d(U,X_1)],
    \]
    where $X_1, X_2$ denote the locations of two arbitrary customers and $U$ denotes a Uniform$[0,1]$ random variable. As the average distance from a uniform point on the circle to any point equals $1/2$, also $\E[d(U,X_1)] = 1/2$ by the independence of $X_1$ and $U$. Due to the symmetry, we further have $\E[d(X_1,X_2)] = \E[d(X_2,X_1)]$. Combined with the fact that $d(X_1,X_2) + d(X_2,X_1) = 1$, we find: $\E[d(X_2,X_1)] = 1/2$. The proof is now finished by substituting $\E[T] = \alpha/2$ in the expression for the waiting time.
\end{proof}

\begin{remark}
    The mean number of waiting customers in the system is not affected by the arrival location distribution. Using Eliazar's limit argument \citep{Eliazar2005}, this is a consequence of the pseudo-conservation law for discrete polling models, cf. Equation (3.21) of \citet{Boxma1989}, stating that the weighted (by load of queue) sum of expected waiting times is indifferent to the arrival location distribution, $\lambda_i/(\sum_j \lambda_j)$.
\end{remark}
\subsection{Integral equation for spread of customers}
\label{sec:fpe}
We now turn to $f(x,y)$, see \eqref{eq:deff}, describing the average spread of customers on the circle. This characteristic is essential to the mean-value analysis of polling models on a circle and can be used to derive expressions for both the mean batch sojourn time and mean time to delivery. In this section, we prove that $f(x,y)$ satisfies is the unique solution to a given integral equation.

Consider the class of functions, $C_\pi$, consisting of all functions $f: [0,1]^2\to [0,\infty)$ such that $f(x,y) = O(\pi(x))$. In this section we prove that:

\begin{proposition}
\label{prop:FPE}
The function $f(x,y)$ is the \underline{unique} solution, in $C_{\pi}$, to the integral equation:
\begin{align}
\label{eq:integraleq}
        \big[\rho \pi(y) + 1-\rho\big]f(x,y) &= \rho{\int_{u=x}^y}^*\big[\rho \pi(u) + 1-\rho]\cdot \big(\pi(y)f(x,u) + \pi(x)f(y,u)\big)\dx u\\
        &\quad + \lambda\mathbb{E}[K]\pi(x)
         \begin{aligned}[t]
         \bigg\{&\alpha{\int_{u=x}^y}^*\big[\rho\pi(u) + (1-\rho)\big]\dx u + \frac{\rho\mathbb{E}[B^2]}{2\mathbb{E}[B]}\pi(y)\\
        &+ \mathbb{E}[B]\pi(y)\frac{\mathbb{E}[K(K-1)]}{\mathbb{E}[K]}{\int_{u=x}^y}^* \big[\rho \pi(u) + 1- \rho\big]\dx u\bigg\}.
        \end{aligned} \nonumber
\end{align}
\end{proposition}

\begin{remark}
$f(x,y)$ should be such that the expected number of waiting customers in the system is equal to $\mathbb{E}[L]$. It can indeed be proven that any solution, in $C_\pi$, to the fixed point equation immediately satisfies this condition, see Lemma \ref{lemma:unique}.
\end{remark}

The proof consists of three main elements: (i) $f(x,y)$ satisfies \eqref{eq:integraleq}, (ii) $f(x,y)$ is an element of $C_{\pi}$ and (iii) the integral equation has a unique solution in $C_{\pi}$. \\
The proof of the first part follows a mean-value based approach. We start from a variant of Little's law, where we relate the number of customers in the system (at given locations) to the waiting time of a customer (arriving at a given location). We then derive an expression for the average waiting time of such a customer, which again involves the location of the other customers in the system. Ultimately, this results in the integral equation \eqref{eq:integraleq}.

The variant of Little's law is related to a similar observation as \citet[Equation (6)]{Winands2006}. The idea is that each customer at a location $x$ needs to have arrived after the last visit to that location.

\begin{lemma}
\label{lemma:LittlesLaw_general}
    Let $W_{y}$ denote the waiting time of a customer arriving at location $y$. Then the function $f(\cdot,\cdot)$ satisfies the following for all $x,y \in [0,1]$:
    \begin{align}
    \label{eq:gen_Little}
        {\int_{u=x}^y}^* \big[\rho\pi(u) + 1-\rho]f(x, u) \dx u = \lambda\mathbb{E}[K]\pi(x){\int_{u=x}^y}^* \big[\rho\pi(u) + 1-\rho\big]\mathbb{E}[W_{y}\vert S = u]\dx u.
    \end{align}
\end{lemma}
\begin{proof}
We use an intuitive argument to prove this statement. A similar argument, for Little's law, can be found in \citet[page 51]{tijms2003}.\\
    Consider the case in which each customer at location $x$ gets a reward of $1$ dollar per time unit she has to wait while the server is in $(x,y)$. Then there would be two ways of paying this out. First, each time unit we can pay each customer at location $x$ one dollar if the server is in $(x,y)$. Given that the server is at $u \in (x,y)$, this amounts to an average of $f(x,u)$ per time unit the server is at $u$. Using that the server is in $[u,u+\dx u]$ with probability $[\rho\pi(u) + 1-\rho)]\dx u$, then shows that per time unit we pay on average:
    \[ {\int_{u=x}^y}^* \big[\rho\pi(u) + 1-\rho]f(x, u) \dx u. \]
    The second method entails paying each customer at location $x$ the moment the server leaves $(x,y)$, i.e. when the server reaches $y$. The amount each customer receives is then given by the time they waited at $x$ while the server was in $(x,y)$. Remark that we only pay customers who arrive during the period that the server is in $(x,y)$. The total reward these customers obtain is the residual time of this period at their arrival instant, which is equivalent to the waiting time $W_y$ of an auxiliary customer arriving at location $y$. With probability $[\rho\pi(u) + 1-\rho]\dx u$ a customer arriving at location $x$ sees the server in a small interval $[u, u+\dx u]$, obtaining a reward of $W_y\vert S = u$. On average, a total of $\lambda \E[K]\pi(x)$ customers arrive at $x$ per time unit, therefore on average we pay the following amount per time unit:
    \begin{equation*}
        \lambda\mathbb{E}[K]\pi(x)\bigg\{{\int_{u=x}^y}^* \big[\rho\pi(u) + 1-\rho\big]\mathbb{E}[W_y\vert S = u]\dx u\bigg\}. \qedhere
    \end{equation*}
\end{proof}

\begin{remark}
    An alternative argument for Lemma \ref{lemma:LittlesLaw_general} is the following. Each customer at $x$ needs to have arrived since the last visit of the server to location $x$. Given that the server is currently in $(x,y]$, the time since the last visit to $x$ is given by the age of the period in which the server travels from $x$ to $y$. The expected age of this period is equal to the expected residual time of this period, which in turn is equal to the waiting time of an auxiliary customer at location $y$ (given that the server is in $(x,y]$).
\end{remark}

It is now left to find an expression for $\E[W_y|S=u]$. For this, we consider the state of the system at the arrival instant of the customer, and link future arrivals to this current state of the system. Inspired by \citet{Resing}, we construct a branching process with immigration adhering to the following rules:
\begin{itemize}
\item A customer (a) is the offspring of customer (b) when (a) arrives during the service of (b).
\item A customer (a) is called an immigrant when she arrives during a travel period of the server.
\end{itemize}
This branching process may include customers who will not be served before the tagged customer at location $y$, and therefore do not add to the experienced waiting time of the customer at $y$. Because of this, we trim the branching process to only include customers who will be served before the tagged customer, see Figure \ref{fig:generatedwaiting} \citep{Engels}.

 \begin{figure}[!htbp]
    \centering
    \begin{subfigure}{0.3\textwidth}
        \begin{tikzpicture}
        \draw[postaction = {decorate, decoration = {markings, mark = at position 0.6 with {\node[draw, red, circle, fill, inner sep = 1mm] (server){};}}}]
        [postaction = {decorate, decoration = {markings, mark = at position 0.48 with {\node[regular polygon, regular polygon sides = 5,draw, black,  fill, inner sep = 0.7mm] {};}}}]
        [postaction = {decorate, decoration = {markings, mark = at position 0.35 with {\node[regular polygon, regular polygon sides = 5,draw, black,  fill, inner sep = 0.7mm]{};}}}]
        [postaction = {decorate, decoration = {markings, mark = at position 0.6 with {\node[regular polygon, regular polygon sides = 5,draw,  fill=orange!50, inner sep = 0.7mm]{};}}}]
        [postaction = {decorate, decoration = {markings,  mark = at position 0.675 with {\node[inner sep = 0mm] (xlabel) {};}}}]
        [postaction = {decorate, decoration = {markings, mark = at position 0.42 with {\node[regular polygon, regular polygon sides = 5,draw, fill = blue!50,  fill, inner sep = 0.7mm] {};}}}]
        [postaction = {decorate, decoration = {markings, mark = at position 0.55 with {\node[regular polygon, regular polygon sides = 5,draw, fill = blue!50,  fill, inner sep = 0.7mm] {};}}}]
        [postaction = {decorate, decoration = {markings, mark = at position 0.8 with {\node[regular polygon, regular polygon sides = 5,draw, fill = blue!50,  fill, inner sep = 0.7mm] {};}}}]
        [postaction = {decorate, decoration = {markings, mark = at position 0.3 with {\node[regular polygon, regular polygon sides = 5,draw, fill=green!50, inner sep = 0.7mm] {};}}}]
        [postaction = {decorate, decoration = {markings,  mark = at position 0.5 with {\arrowreversed[line width = 0.7mm]{stealth}}}}]
        [postaction = {decorate, decoration = {markings,  mark = at position 1 with {\arrowreversed[line width = 0.7mm]{stealth}}}}]
        (0,0) circle (2);
        \node [left] at (server.west) {Server};
        \end{tikzpicture}
    \end{subfigure}
        \begin{subfigure}{0.3\textwidth}
            \begin{tikzpicture}
            \draw[postaction = {decorate, decoration = {markings, mark = at position 0.55 with {\node[draw, red, circle, fill, inner sep = 1mm] (server){};}}}]
            [postaction = {decorate, decoration = {markings, mark = at position 0.48 with {\node[regular polygon, regular polygon sides = 5,draw, black,  fill, inner sep = 0.7mm] {};}}}]
            [postaction = {decorate, decoration = {markings, mark = at position 0.35 with {\node[regular polygon, regular polygon sides = 5,draw, black,  fill, inner sep = 0.7mm]{};}}}]
            [postaction = {decorate, decoration = {markings,  mark = at position 0.675 with {\node[inner sep = 0mm] (xlabel) {};}}}]
            [postaction = {decorate, decoration = {markings, mark = at position 0.42 with {\node[regular polygon, regular polygon sides = 5,draw, fill = blue!50,  fill, inner sep = 0.7mm] {};}}}]
            [postaction = {decorate, decoration = {markings, mark = at position 0.55 with {\node[regular polygon, regular polygon sides = 5,draw, fill = blue!50,  fill, inner sep = 0.7mm] {};}}}]
            [postaction = {decorate, decoration = {markings, mark = at position 0.8 with {\node[regular polygon, regular polygon sides = 5,draw, fill = blue!50,  fill, inner sep = 0.7mm] {};}}}]
            [postaction = {decorate, decoration = {markings, mark = at position 0.3 with {\node[regular polygon, regular polygon sides = 5,draw, fill=green!50, inner sep = 0.7mm] {};}}}]
            [postaction = {decorate, decoration = {markings,  mark = at position 0.5 with {\arrowreversed[line width = 0.7mm]{stealth}}}}]
            [postaction = {decorate, decoration = {markings,  mark = at position 1 with {\arrowreversed[line width = 0.7mm]{stealth}}}}]
            [postaction = {decorate, decoration = {markings, mark = at position 0.18 with {\node[regular polygon, regular polygon sides = 5,draw, fill=purple!50, inner sep = 0.7mm] {};}}}]
            [postaction = {decorate, decoration = {markings, mark = at position 0.96 with {\node[regular polygon, regular polygon sides = 5,draw, fill=purple!50, inner sep = 0.7mm] {};}}}]
            [postaction = {decorate, decoration = {markings, mark = at position 0.38 with {\node[regular polygon, regular polygon sides = 5,draw, fill=purple!50, inner sep = 0.7mm] {};}}}]
            [postaction = {decorate, decoration = {markings, mark = at position 0.63 with {\node[regular polygon, regular polygon sides = 5,draw, fill=purple!50, inner sep = 0.7mm] {};}}}]
            (0,0) circle (2);
            \node [left] at (server.west) {Server};
            \end{tikzpicture}
    \end{subfigure}
    \hspace{0.04\textwidth}
        \begin{subfigure}{0.3\textwidth}
            \begin{tikzpicture}
            \useasboundingbox (-3,-3.5) rectangle (3,0);
            \node[regular polygon, regular polygon sides = 5, black, draw, fill = orange!50, inner sep = 2mm] (1) at (0,0){};
            \node[regular polygon, regular polygon sides = 5, black, draw, fill = blue!50, inner sep = 2mm] (2) at (-1,-1.5){};
            \node[regular polygon, regular polygon sides = 5, black, draw, fill = blue!50, inner sep = 2mm] (3) at (1,-1.5){};
            \node[regular polygon, regular polygon sides = 5, black, draw, fill = purple!50, inner sep = 2mm] (4) at (-1,-3){};
            \path[draw] (1) -- (2);
            \path[draw] (1) -- (3);
            \path[draw] (2) -- (4);
            \end{tikzpicture}
    \end{subfigure}
    \caption{Illustration of the extra waiting time of a tagged customer (green) that is generated by a service (of the orange customer) and the corresponding branching process. During the service of the orange customer, blue customers arrive, of which only the first two are considered. During the service of the first blue customer, the red customers arrive, of which only one will be served before the tagged customer.}
    \label{fig:generatedwaiting}
\end{figure}
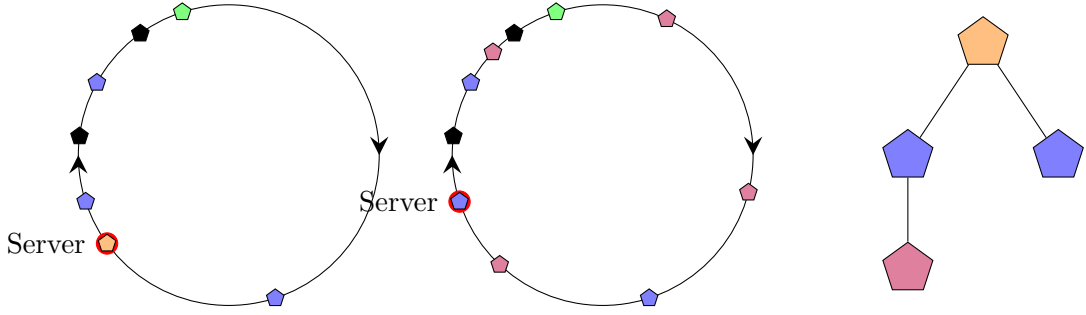

Using this branching process construction, we can define the ``effect'' that customers of these two types have on the waiting time of the tagged customer. Formally we define these as:
\begin{definition}
The total service time of customers in a branching process with an ancestor at location $x$ and tagged customer at location $y$ is denoted as $S(x,y)$. We call this the extra waiting time for a (tagged) customer at location $y$ generated by the service of a customer at location $x$.
\end{definition} 
\begin{definition}
The travel time plus total service time of customers in the branching processes, trimmed after location $y$, of all immigrants that arrive during the travelling of the server from location $x$ to $y$ is defined as $T(x,y)$. We call this the extra waiting time generated by the travelling of the server from $x$ to $y$.
\end{definition} 
Intuitively, these definitions represent the effect that the service of a customer resp. travelling of the server has on the waiting time of the tagged customer. E.g., without a customer at location $x$, the waiting time of a tagged customer at location $y$ would have been $S(x,y)$ lower.

The first moment of $S(x,y)$ can be found by conditioning on the arrivals that occur during the service of the customer. These arrivals again generate a new time $S(X,y)$, where $X$ is their arrival location. This gives a recursive formula, that can be solved analytically. Similarly, the first moment of $T(x,y)$ can be found by conditioning on the arrivals that occur during the travelling of the server. The proof heavily relies on the memoryless property of exponential inter-arrival times.

\begin{lemma}
\label{lemma:generatedwaiting}
    The extra waiting time for a tagged customer at location $y$ generated by a customer at location $x$, $S(x,y)$, has expectation:\\
    \begin{align}
    \label{eq:genwaitservice}
        \mathbb{E}[S(x,y)] = \mathbb{E}[B]\exp\Big(\rho {\int_{\xi = x}^y}^* \pi(\xi)\dx \xi\Big).
    \end{align}
    The expected extra waiting time generated by the travelling of the server from $x$ to a tagged customer at location $y$, $\E[T(x,y)]$, is given by:\\
    \begin{align}
    \label{eq:genwaittravel}
        \mathbb{E}[T(x,y)] = {\int_{u=x}^y}^* \alpha  \exp\Big(\rho {\int_{\xi = u}^y}^{*} \pi(\xi)\dx \xi\Big) \dx u.
     \end{align}
\end{lemma} 
\begin{proof}
    During the service of a customer a total of $\rho = \lambda\E[K]\E[B]$ customers are expected to arrive. Only those customers that arrive between $x$ and $y$ contribute to the waiting time of the tagged customer at location $y$. Note that a customer arriving at location $u$ again generates an extra waiting time of $S(u,y)$, hence:
    \begin{align*}
        \E[S(x,y)] = \E[B] + \rho{\int_{u=x}^y}^* \pi(u)\E[S(u,y)]\dx u.
    \end{align*}
    Taking the derivative with respect to $x$ on both sides gives a first-order differential equation:
    \begin{align*}
        \frac{\dx}{\dx x}\E[S(x,y)] = -\rho \pi(x) \E[S(x,y)].
    \end{align*}
    The proposed expression solves this differential equation with the additional condition that $\lim_{x\uparrow y}\E[S(x,y)] = \E[B]$.\\
    The extra waiting time generated by the travelling of the server can now be derived by noting that during the travelling over a distance $\dx u$ a total of $\lambda\E[K]\alpha \dx u$ customers are expected to arrive. Again, we only consider customers who arrive in between the server and the tagged customer:
    \begin{align*}
        \E[T(x,y)] &= \lambda\E[K]\alpha {\int_{u=x}^y}^*  {\int_{z=u}^y}^* \pi(z)\E[S(z,y)] \dx z\dx u\\
        &=\alpha d(x,y) + \alpha {\int_{u=x}^y}^* \bigg(\exp\Big(\rho {\int_{\xi = u}^y}^* \pi(\xi)\dx \xi\Big) - 1\bigg)\dx u. \qedhere
    \end{align*}
\end{proof}

\begin{remark}
\label{rem:SR}
    At the arrival instant of a customer, she might encounter a residual service time, $B^R$. The total expected extra waiting time generated by this residual service time, $\mathbb{E}[S^R(x,y)]$, satisfies relation \eqref{eq:genwaitservice} with $\mathbb{E}[B]$ replaced by $\mathbb{E}[B^R] = \E[B^2]/(2\E[B])$.
\end{remark} 

With this, we can find an expression for $\E[W_y|S=u]$. In combination with Lemma \ref{lemma:LittlesLaw_general}, this results in a fixed point equation for the average spread of customers $f(x,y)$. The proof of Proposition \ref{prop:FPE} is then finished by proving that this fixed point equation maintains a unique solution.

\begin{proof}[Proof of Proposition \ref{prop:FPE}] We prove this statement in three parts. First we prove that $f(x,y)$ satisfies \eqref{eq:integraleq}, then prove that $f\in C_\pi$ and afterwards prove uniqueness of a solution.\\
 \textbf{Part (i): $f(x,y)$ satisfies \eqref{eq:integraleq}.}\\
    We start from \eqref{eq:gen_Little} and use that the waiting time of a customer at location $y$, given that the server is in location $u$, consists of the following four elements: (i) the extra waiting time generated by the travelling of the server from $u$ to $y$, $T(u,y)$; (ii) the extra waiting time generated by the residual service time of the customer, possibly in service at the arrival instant. With probability $\rho\pi(u)/(\rho \pi(u) + 1-\rho)$ the server is working, in which case the extra generated waiting time equals $S^R(u,y)$; (iii) the extra waiting time generated by all customers arriving in the same batch that will be served before the tagged customer. Recall that a customer at location $z$ generates an extra waiting time of $S(z,y)$; (iv) the work generated by all customers present on the circle in front of $y$ at the time of arrival. In expectation, the waiting time of a customer at location $y$, given the server is at $u$, therefore is given by:
    \begin{equation}
    \label{eq:condwait}
        \begin{aligned}[c]
                \mathbb{E}[W_{y}\vert S = u] &= \underbrace{\mathbb{E}\big[T(u,y)\big]}_{(i)} + \underbrace{\frac{\rho\pi(u)}{\rho \pi(u) + 1-\rho} \mathbb{E}[S^R(u,y)]}_{(ii)} \\
                &\quad + \underbrace{\frac{\mathbb{E}[K(K-1)]}{\mathbb{E}[K]}{\int_{z=u}^y}^* \pi(z)\mathbb{E}[S(z,y)]\dx z}_{(iii)} + \underbrace{{\int_{z=u}^y}^* \mathbb{E}[S(z,y)]f(z,u)\dx z}_{(iv)}.
        \end{aligned}
    \end{equation}
    We now return to \eqref{eq:gen_Little}, where we take the derivative w.r.t. $y$ on both sides:
    \begin{equation}
    \label{eq:proofFPE_firstder}
    \begin{aligned}[c]
        \big[\rho\pi(y) + 1-\rho]f(x, y)= \lambda\mathbb{E}[K]\pi(x)\bigg\{&\big[\rho\pi(y) + 1-\rho\big]\mathbb{E}[W_{y}\vert S = y] \\&+ {\int_{u=x}^y}^* \big[\rho\pi(u) + 1-\rho\big]\frac{\dx}{\dx y}\mathbb{E}[W_{y}\vert S = u]\dx u \bigg\}.
    \end{aligned}        
    \end{equation}
    The expectation $\mathbb{E}[W_{y}\vert S = y]$ is equivalent to the residual service time of a customer at location $y$:
    \[
    \mathbb{E}[W_{y}\vert S = y] = \frac{\rho\pi(y)}{\rho\pi(y) + 1-\rho}\frac{\E[B^2]}{2\E[B]} + \frac{1-\rho}{\rho\pi(y) + 1-\rho}\cdot 0.\]
    For the second part of \eqref{eq:proofFPE_firstder}, we realize that the derivative of all four elements of \eqref{eq:condwait} results in a factor $\rho\pi(y)$ times the same elements, plus some other parts:
    \begin{align}
    \label{eq:condWait_der}
        \frac{\dx}{\dx y}\mathbb{E}[W_{y}\vert S = u] &= \underbrace{\rho\pi(y)\mathbb{E}[T(u,y)] + \alpha}_{(i)} + \underbrace{\rho\pi(y)\frac{\rho \pi(u)}{\rho \pi(u) + 1-\rho} \mathbb{E}[S^R(u,y)]}_{(ii)} \nonumber\\
        &\quad + \underbrace{\rho \pi(y) \frac{\mathbb{E}[K(K-1)]}{\mathbb{E}[K]}{\int_{z=u}^y}^* \pi(z)\mathbb{E}[S(z,y)]\dx z + \frac{\mathbb{E}[K(K-1)]}{\mathbb{E}[K]}\pi(y)\mathbb{E}[B]}_{(iii)}\nonumber\\
        &\quad + \underbrace{\rho \pi(y){\int_{z=u}^y}^* \mathbb{E}[S(z,y)]f(z,u)\dx z + \mathbb{E}[B]f(y,u)}_{(iv)} \nonumber\\
        &=\rho \pi(y)\mathbb{E}[W_{y}\vert S = u] + \alpha +  \frac{\mathbb{E}[K(K-1)]}{\mathbb{E}[K]}\pi(y)\mathbb{E}[B] +  \mathbb{E}[B]f(y,u).
    \end{align}
    Therefore \eqref{eq:proofFPE_firstder} simplifies to:
    \begin{align*}
        \big[\rho\pi(y) + 1-\rho]f(x, y) &=  \lambda\mathbb{E}[K]\pi(x)\bigg\{\rho\pi(y)\frac{\mathbb{E}[B^2]}{2\mathbb{E}[B]} + \rho\pi(y){\int_{u=x}^y}^* \big[\rho\pi(u) + 1-\rho\big]\mathbb{E}[W_{y}\vert S = u]\dx u\\
        &\quad\quad+ {\int_{u=x}^y}^* \big[\rho\pi(u) + 1-\rho\big]\Big\{\alpha +  \frac{\mathbb{E}[K(K-1)]}{\mathbb{E}[K]}\pi(y)\mathbb{E}[B] +  \mathbb{E}[B]f(y,u)\Big\}\dx u\bigg\}\\
        &= \rho\pi(y)\lambda\mathbb{E}[K]\pi(x){\int_{u=x}^y}^* \big[\rho\pi(u) + 1-\rho\big]\mathbb{E}[W_{y}\vert S = u]\dx u\\
        &\quad + \rho\pi(x) {\int_{u=x}^y}^*\big[\rho\pi(u) + 1-\rho\big] f(y,u)\dx u \\
        &\quad +  \lambda\mathbb{E}[K]\pi(x)\bigg\{\pi(y)\frac{\rho\mathbb{E}[B^2]}{2\mathbb{E}[B]} + \alpha {\int_{u=x}^y}^*\big[\rho\pi(u) + 1-\rho\big]\dx u \\
        &\hspace{3cm} +\mathbb{E}[B]\pi(y)\frac{\mathbb{E}[K(K-1)]}{\mathbb{E}[K]}{\int_{u=x}^y}^*\big[\rho\pi(u) + 1-\rho\big]\dx u\bigg\}.
    \end{align*}
    The proof that the function $f$ adheres to the integral equation is now concluded by applying the initial equality \eqref{eq:gen_Little} once more.

\textbf{Part (ii): $f$ is an element of $C_\pi$.}\\
We need to prove that $f$ is indeed bounded by $\pi(x)\cdot c$ for some $c$. Recall the definition of $f$, and consider the following:
\begin{align*}
    \E\big[L\big([x, x+\dx x]\big)\big] = \int_{u=0}^1 [\rho \pi(u) + 1-\rho]f(x, u)\dx u \dx x + O(\dx x^2),
\end{align*}
for $\dx x$ sufficiently small. Intuitively, the left-hand side needs to be bounded by the arrival rate to $[x,x+\dx x]$ times the expected cycle length (starting from $x$) that an arbitrary arriving customer encounters. As $\rho < 1$, it is known that this is finite, as otherwise the mean waiting times would explode, hence there exists a $c_1$ such that:
\begin{align}
\label{eq:elementcpiproof_1}
    \E\big[L\big([x, x+\dx x]\big)\big] = \int_{u=0}^1 [\rho \pi(u) + 1-\rho]f(x, u)\dx u \dx x \leq c_1 \pi(x) \dx x.
\end{align}
We now turn our attention to the fixed point equation. Remark that for any $x,y$ we have:
\begin{align*}
    [\rho \pi(y) + 1-\rho]f(x,y) &= \rho {\int_{u=x}^y}^*\big[\rho \pi(u) + 1-\rho]\big[\pi(x)f(y,u) + \pi(y)f(x,u)]\dx u \\
    &\quad + \alpha\lambda\E[K]\pi(x){\int_{u=x}^y}^* \big[\rho \pi(u) + 1-\rho]\dx u +\frac{\rho\E[B^2]}{2\E[B]}\pi(x)\pi(y)\\
    &\quad + \frac{\rho\E[K(K-1)]}{\E[K]}\pi(x)\pi(y){\int_{u=x}^y}^* \big[\rho \pi(u) + 1-\rho]\dx u.
\end{align*}
As all integrands are strictly positive we obtain the following bound:
\begin{align*}
    [\rho \pi(y) + 1-\rho]f(x,y) &\leq \rho {\int_{u=0}^1}\big[\rho \pi(u) + 1-\rho]\big[\pi(x)f(y,u) + \pi(y)f(x,u)]\dx u \\
    &\quad+ c_2\pi(x) + c_3\pi(x)\pi(y),
\end{align*}
for some $c_2,c_3 > 0$. Using \eqref{eq:elementcpiproof_1} we then find:
\begin{align*}
    [\rho \pi(y) + 1-\rho]f(x,y) \leq 2\rho c_1 \pi(x)\pi(y) + c_2\pi(x) + c_3\pi(x)\pi(y). 
\end{align*}
One can now divide both sides by $\rho\pi(y) + 1-\rho$ and use that $\rho\pi(y)/(\rho \pi(y) + 1-\rho) < 1$. Indeed, this results in the constraint that $f(x,y)$ can not grow faster than $\pi(x)$.

\textbf{Part (iii): \eqref{eq:integraleq} has a unique solution in $C_\pi$}\\
\textit{Case (i): $\pi(x) >0$ for all $x$.} Let $a(x,y)$ denote the inhomogeneous part of the integral equation in \eqref{eq:integraleq}. First consider the case where $\pi$ is strictly positive, then we have that solving \eqref{eq:integraleq} is equivalent to solving:
\begin{align}
\label{eq:FPEg2}
    g(x,y) = \rho\pi(y){\int_{u=x}^y}^*\big[g(x,u) + g(y,u)\big]\dx u + \frac{a(x,y)}{\pi(x)},
\end{align}
where $\frac{a(x,y)}{\pi(x)} < \infty$. This equivalence follows from
the transformation $g(x,y) = \frac{[\rho\pi(y) +  1-\rho]f(x,y)}{\pi(x)}$. Remark that $f\in C_\pi$ implies that $g$ is bounded. Consequently proving that \eqref{eq:FPEg2} has a unique \emph{bounded} solution implies that \eqref{eq:integraleq} has a unique solution in $C_\pi$.\\
Now remark that two bounded solutions $g_0,g_1$ to \eqref{eq:FPEg2} adhere to:
\begin{align}
\label{eq:uniek_ineq1}
    \vert g_0(x,y) - g_1(x,y) \vert &\leq \rho\pi(y){\int_{u=x}^y}^*\big[\vert g_0(x,u)-g_1(x,u)\vert + \vert g_0(y,u)-g_1(y,u)\vert \big]\dx u 
\intertext{and therefore by bounding $g_0-g_1$ by the supremum norm we find:}
    \vert g_0(x,y) - g_1(x,y) \vert&\leq 2\rho\pi(y) \supnorm{g_0-g_1}.\label{eq:uniek_ineq2}
\end{align}
We apply inequality \eqref{eq:uniek_ineq2} to \eqref{eq:uniek_ineq1} with $(x,y)$ replaced by $(x,u)$ and $(y,u)$, and see:
\begin{align}
\label{eq:uniek_ineq3}
     \vert g_0(x,y) - g_1(x,y) \vert \leq 4\rho^2\pi(y)\supnorm{g_0-g_1}{\int_{u=x}^y}^*\pi(u);
\end{align}
substituting this inequality for $(x,u)$ and $(y,u)$ into \eqref{eq:uniek_ineq1}, we obtain:
\begin{align*}
    \vert g_0(x,y) - g_1(x,y) \vert &\leq 4\rho^3\pi(y)\supnorm{g_0-g_1}{\int_{u=x}^y}^*\pi(u){\int_{v=x}^u}^*\pi(v)\dx v \dx u \\
    &\quad + 4\rho^3\pi(y)\supnorm{g_0-g_1}{\int_{u=x}^y}^*\pi(u){\int_{v=y}^u}^*\pi(v)\dx v \dx u,
    \intertext{and interchanging the order of the first integrals shows:}
     \vert g_0(x,y) - g_1(x,y) \vert &\leq  4\rho^3\pi(y)\supnorm{g_0-g_1}{\int_{v=x}^y}^*\pi(v){\int_{u=v}^y}^*\pi(u)\dx u \dx v \\
     &\quad + 4\rho^3\pi(y)\supnorm{g_0-g_1}{\int_{u=x}^y}^*\pi(u){\int_{v=y}^u}^*\pi(v)\dx v \dx u \\
     &=  4\rho^3\pi(y)\supnorm{g_0-g_1}{\int_{u=x}^y}^*\pi(u)\dx u.
\end{align*}
We can now substitute this inequality into \eqref{eq:uniek_ineq1}, and repeat this process, to eventually find for any $n>0$:
\begin{align*}
    \vert g_0(x,y) - g_1(x,y) \vert &\leq 4\rho^n\pi(y)\supnorm{g_0-g_1}{\int_{u=x}^y}^*\pi(u)\dx u.
\end{align*}
As we consider bounded solutions $g$, we know that $\supnorm{g_0-g_1}$ is bounded, and therefore by taking $n \to \infty$ the above relation implies $\vert g_0(x,y) - g_1(x,y)\vert = 0$, as $\pi(y) < \infty$. Hence, \eqref{eq:FPEg2} attains at most one solution.\\
\textit{Case (ii): $\pi(x)$ not strictly positive.} For the case where $\pi(x)$ is not strictly positive, we can not directly use the same transformation as before, as it is not defined for $x,y$ where $\pi(x)\pi(y) = 0$. Note, however, that as $f(x,y) \in C_\pi$ we have $f(x,y) \leq c\cdot \pi(x)$. Therefore $f(x,y) = 0$ for $\pi(x) = 0$. Turning to \eqref{eq:integraleq}, for $\pi(y) = 0$ we see that the first integral of \eqref{eq:integraleq}  disappears as both $\pi(y)$ and $f(y,u)$ are $0$. Therefore, taking the factor $(1-\rho)$ to the right-hand side shows:
\[ f(x,y) = \frac{a(x,y)}{1-\rho} \quad \text{if: } \pi(y) = 0. \]
With this, we can now rewrite the integral equation in \eqref{eq:integraleq} to exclude points at which $\pi(x)\pi(y) = 0$:
\begin{align*}
     \big[\rho\pi(y)+1-\rho\big]f(x,y) &=\rho{\int_{u=x}^y}^*\mathbbm{1}\{\pi(u) > 0\}\big[\rho \pi(u) + 1-\rho\big]\big[\pi(y)f(x,u) + \pi(x)f(y,u)\big]\dx u \\
     &\quad+ \rho{\int_{u=x}^y}^*\mathbbm{1}\{\pi(u) = 0\}\big[\pi(y)a(x,u) + \pi(x)a(y,u)\big]\dx u + a(x,y).
\end{align*}
To prove that this has a unique solution for $f(x,y)$ with $\pi(x)\pi(y) > 0$ follows the same steps as in the case with $\pi(x) > 0$. 
\end{proof}

\subsection{Partial solution}
\label{sec:partial}
The integral equation, \eqref{eq:integraleq}, consists of two components: a homogeneous part involving $f$ (line 1 of \eqref{eq:integraleq}), and inhomogeneous terms that do not depend on $f$ (lines 2 and 3 of \eqref{eq:integraleq}). The sum of the solutions to the integral equation with each inhomogeneous term alone, also is equal to the solution to the full integral equation. We therefore can also write $f = f_{\alpha} + f_{B^R} + f_{K}$. Where $f_{K}$, for instance, is the solution to:
\begin{equation}
\label{eq:FPEnotsolve}
\begin{aligned}[c]
    \big[\rho \pi(y) + 1-\rho\big]f_{K}(x,y) = &\rho{\int_{u=x}^y}^*\big[\rho \pi(u) + 1-\rho]\cdot \big(\pi(y)f_{K}(x,u) + \pi(x)f_{K}(y,u)\big)\dx u\\
    &\quad + \rho\frac{\E[K(K-1)]}{\E[K]}\pi(x)\pi(y){\int_{u=x}^y}^*\big[\rho\pi(u) + (1-\rho)\big]\dx u.
\end{aligned}
\end{equation}

In this section, we find analytical expressions for $f_{\alpha}$ and $f_{B^R}$. Consequently, we solve the integral equation for the system without batch arrivals.

\begin{proposition}
    \label{prop:partialsol}
    The solution to the integral equation is given by:
    \begin{equation}
    \label{eq:partialsol}
        \begin{aligned}
        f(x,y) &= \frac{\lambda\E[K] \pi(x)\alpha }{1-\rho} {\int_{u=x}^y}^*\big[\rho \pi(u) + 1 - \rho] \dx u +   \frac{\rho\pi(y)}{\rho \pi(y) + 1-\rho}\frac{\lambda\E[K]\pi(x) \E[B^2]}{2\E[B]}\\
        &\quad + \frac{\rho\pi(y)}{\rho \pi(y) + 1-\rho}\frac{\lambda\E[K] \pi(x) \rho \E[B^2] }{(1-\rho)\E[B]}{\int_{u=x}^y}^* \pi(u) \dx u + f_K(x,y),
        \end{aligned}
    \end{equation}
    where $f_K(\cdot)$ is the solution to \eqref{eq:FPEnotsolve}.
\end{proposition}
\begin{proof}
    Using the splitting of $f$ into the different functions, we propose:
    \begin{align*}
         f(x,y) &= \underbrace{\frac{\lambda\E[K] \pi(x)\alpha }{1-\rho} {\int_{u=x}^y}^*\big[\rho \pi(u) + 1 - \rho] \dx u}_{f_\alpha} \\
    &\quad + \underbrace{\frac{\rho\pi(y)}{\rho \pi(y) + 1-\rho}\frac{\lambda\E[K]\pi(x) \E[B^2]}{2\E[B]} + \frac{\rho\pi(y)}{\rho \pi(y) + 1-\rho}\frac{\lambda\E[K]\pi(x) \rho \E[B^2] }{(1-\rho)\E[B]}{\int_{u=x}^y}^* \pi(u) \dx u}_{f_{B^R}}  + f_K(x,y) \nonumber.
    \end{align*}
    We first prove that $f_\alpha$ solves the integral equation with as inhomogeneous term only the term dependent on $\alpha$. We substitute the solution into the right-hand side of the integral equation:
    \begin{align*}
        \rho {\int_{u=x}^y}^* \big[&\rho \pi(u) + 1- \rho\big]\big[\pi(y)f_\alpha(x,u) + \pi(x)f_\alpha(y,u)]\dx u + \lambda\E[K]\pi(x)\alpha{\int_{u=x}^y}^*\big[\rho\pi(u) + (1-\rho)\big]\dx u\\
        =&\lambda\E[K]\frac{\rho\alpha}{1-\rho}\pi(x)\pi(y){\int_{u=x}^y}^*\big[\rho \pi(u) + 1- \rho\big]{\int_{z=x}^u}^*\big[\rho \pi(z) + 1- \rho\big]\dx z\dx u\\
        &\quad+\lambda\E[K]\frac{\rho\alpha}{1-\rho}\pi(x)\pi(y){\int_{u=x}^y}^*\big[\rho \pi(u) + 1- \rho\big]{\int_{z=y}^u}^*\big[\rho \pi(z) + 1- \rho\big]\dx z\dx u\\
        &\quad + \lambda\E[K]\pi(x)\alpha{\int_{u=x}^y}^* \big[\rho\pi(u) + 1-\rho \big]\dx u. 
    \end{align*}
    Note that the following equivalence holds:
    \begin{align*}
        {\int_{u=x}^y}^* &\big[\rho\pi(u) + 1-\rho\big]\Big\{{\int_{z=x}^u}^*\big[\rho \pi(z) + 1- \rho\big]\dx z+{\int_{z=y}^u}^*\big[\rho \pi(z) + 1- \rho\big]\dx z\Big\} \dx u\\
        &={\int_{z=x}^y}^* \big[\rho\pi(z) + 1-\rho\big]{\int_{u=z}^y}^* \big[\rho\pi(u) + 1-\rho\big]\dx u \dx z \\
        &\quad + {\int_{u=x}^y}^* \big[\rho\pi(u) + 1-\rho\big]{\int_{z=y}^u}^* \big[\rho\pi(z) + 1-\rho\big]\dx z \dx u\\
        &={\int_{z=x}^y}^* \big[\rho\pi(z) + 1-\rho\big]\dx z,
    \end{align*}
    where the last step follows from the fact that the sum of the two inner integrals adds up to the integral over the entire circle which equals 1. Therefore we have:
    \begin{align*}
        \rho {\int_{u=x}^y}^* \big[&\rho \pi(u) + 1- \rho\big]\big[\pi(y)f_\alpha(x,u) + \pi(x)f_\alpha(y,u)]\dx u + \lambda\E[K]\pi(x)\alpha{\int_{u=x}^y}^*\big[\rho\pi(u) + (1-\rho)\big]\dx u\\
        &=\lambda\E[K]\frac{\rho\alpha}{1-\rho}\pi(x)\pi(y){\int_{u=x}^y}^* \big[\rho\pi(u) + 1-\rho\big]\dx u + \lambda\E[K]\pi(x)\alpha{\int_{u=x}^y}^* \big[\rho\pi(u) + 1-\rho \big]\dx u\\
        &=\lambda\E[K]\pi(x)\frac{\alpha[\rho \pi(y) + 1-\rho]}{1-\rho}{\int_{u=x}^y}^* \big[\rho\pi(u) + 1-\rho \big]\dx u.
    \end{align*}
    Hence, the first part of the solution solves the integral equation with the first inhomogeneous term.\\
    We now turn to the integral equation with only the second inhomogeneous term. We prove that $f_{B^R}$ solves this integral equation. Again we first focus on the right-hand side and substitute the proposed solution:
    \begin{align*}
        \rho {\int_{u=x}^y}^* \big[&\rho \pi(u) + 1- \rho\big]\big[\pi(y)f_{B^R}(x,u) + \pi(x)f_{B^R}(y,u)]\dx u + \lambda\E[K]\frac{\rho\mathbb{E}[B^2]}{2\mathbb{E}[B]} \pi(x)\pi(y)\\
        &=\frac{\lambda\E[K]\rho^3\E[B^2]}{(1-\rho)\E[B]}\pi(x)\pi(y){\int_{u=x}^y}^*\pi(u){\int_{z=x}^u}^*\pi(z)\dx z\dx u + \frac{\lambda\E[K] \rho^2 \E[B^2]}{2\E[B]}\pi(x)\pi(y){\int_{u=x}^y}^*\pi(u)\dx u\\
        &\quad + \frac{\lambda\E[K]\rho^3\E[B^2]}{(1-\rho)\E[B]}\pi(x)\pi(y){\int_{u=x}^y}^*\pi(u){\int_{z=u}^y}^*\pi(z)\dx z\dx u + \frac{\lambda\E[K] \rho^2 \E[B^2]}{2\E[B]}\pi(x)\pi(y){\int_{u=x}^y}^*\pi(u)\dx u\\
        &\quad + \lambda\E[K]\frac{\rho\mathbb{E}[B^2]}{2\mathbb{E}[B]} \pi(x)\pi(y).
    \end{align*}
    Here we can apply the following equivalence relation,. Following a similar reasoning as for the $f_{\alpha}$ case, we have:
    \begin{align*}
         {\int_{u=x}^y}^* \pi(u)\Big\{&{\int_{z=x}^u}^*\pi(z)\dx z+{\int_{z=y}^u}^*\pi(z)\dx z\Big\} \dx u \\
         &={\int_{z=x}^y}^*\pi(z){\int_{u=z}^y}^*\pi(u)\dx u \dx z + {\int_{u=x}^y}^*\pi(u){\int_{z=y}^u}^*\pi(z)\dx z \dx u ={\int_{u=x}^y}^*\pi(u)\dx u.
    \end{align*}
    Applying this to the previous step, we see the following:
    \begin{align*}
        \rho {\int_{u=x}^y}^* \big[&\rho \pi(u) + 1- \rho\big]\big[\pi(y)f_{B^R}(x,u) + \pi(x)f_{B^R}(y,u)]\dx u + \lambda\E[K]\frac{\rho\mathbb{E}[B^2]}{2\mathbb{E}[B]} \pi(x)\pi(y)\\
        &=\frac{\lambda\E[K]\rho^2\E[B^2]}{(1-\rho)\E[B]}\pi(x)\pi(y){\int_{u=x}^y}^*\pi(u)\dx u +  \lambda\E[K]\frac{\rho\mathbb{E}[B^2]}{2\mathbb{E}[B]} \pi(x)\pi(y).
    \end{align*}
    Note that this is equal to the left-hand side of the integral equation. Hence, we have solved the integral equation with only the inhomogenous term dependent on $\E[B^2]/E[B]$.\\
    The proof now follows from the fact that the sum of the solutions solves the complete integral equation. 
\end{proof}

\begin{remark}
    Under $K\equiv 1$ it is evident that $f_K(x,y) = 0$, and therefore an analytical solution to $f(x,y)$ is found.
\end{remark}

\begin{example}[$\pi(x) = 1$]
    In this case the system simplifies to the case with uniform arrivals. The resulting solution for $f(x,y)$, under $K\equiv 1$, is given by:
    \begin{align*}
        f(x,y) &= \frac{\lambda \alpha}{1-\rho}d(x,y) + \frac{\rho\lambda\E[B^2]}{2\E[B]}+\frac{\rho\lambda\E[B^2]}{(1-\rho)\E[B]}d(x,y).
    \end{align*}
\end{example}

\begin{example}[$\pi(x) = nx^{n-1}, \, n\geq 1$]
    Under this case, we find that \eqref{eq:partialsol}, with $K\equiv 1$, reduces to:
    \begin{align*}
        f(x,y) &= \frac{\lambda \alpha nx^{n-1}}{1-\rho}\Big[(1-\rho)d(x,y) + \rho\big(y^n - x^n + \mathbbm{1}\{x>y\}\big)\Big]\\
        &\quad + \frac{\rho ny^{n-1}}{\rho ny^{n-1}+1-\rho}\frac{nx^{n-1}\lambda\E[B^2]}{2\E[B]} + \frac{\rho ny^{n-1}}{\rho ny^{n-1}+1-\rho}\frac{nx^{n-1}\lambda \E[B^2]}{(1-\rho)\E[B]}\big(y^n - x^n + \mathbbm{1}\{x>y\}\big).
    \end{align*}
    For $n=1$ this agrees with Example 1 as $d (x,y) = y - x + \mathbbm{1}\{x>y\}$. Additionally, one can verify that this agrees with Proposition 4.6 in \citet{Engels}.
\end{example}
\subsection{Solving the integral equation}
\label{sec:successive}
Thus far, we have constructed an integral equation for $f(x,y)$ and have solved it partly analytically. We have not been able to find an analytical solution to the integral equation (unless $\pi(x) \equiv 1$, the uniform arrival case treated in \citet{Engels} or $K \equiv 1$). We therefore propose an algorithm to determine the solution to \eqref{eq:integraleq}, and in particular \eqref{eq:FPEnotsolve}.\\
Building on the proof of uniqueness, cf. the third part of the proof of Proposition \ref{prop:error}, we know it is sufficient to instead solve the following simpler integral equation: 
\begin{equation}
\label{eq:fpe_g_general}
    g(x,y) = \rho\pi(y) {\int_{u=x}^y}^*\big[g(x,u) + g(y,u)\big]\dx u + b(x,y),
\end{equation}
with:
\[b(x,y) = \frac{\rho\E[K(K-1)]\pi(y)}{\E[K]}{\int_{u=x}^y}^*\big[\rho \pi(u) + 1-\rho\big]\dx u.\] 
Note that we have chosen a slightly different representation for the inhomogeneous function ($b(x,y)$) than in the proof of Proposition \ref{prop:FPE}. The theory in this section will hold for any non-negative function $b(x,y)$. The function $f(x,y)$ now follows from the following transformation of $g$:
\begin{align*}
	f(x,y) = \frac{\pi(x)g(x,y)}{\rho \pi(y) + 1-\rho}.
\end{align*}

Solving the integral equation for $g$ can be done by means of successive substitution \citep[Section 8.3]{PicardRef}. In this algorithm, one starts with an initial guess, $g_0$, which is then updated according to the right-hand side of the integral equation in each step. Once the difference between two subsequent guesses is small enough, the iteration is stopped. Formally: let $\delta > 0$ be the stopping criterion, then the algorithm follows from the following steps:
\begin{itemize}
\item Find $g_1$ using:
\[g_1(x,y) = \rho\pi(y) {\int_{u=x}^y}^*\big[g_0(x,u) + g_0(y,u)\big]\dx u + b(x,y),\]
set $n = 1$.
\item While $\supnorm{g_n-g_{n-1}} > \delta$ do:
\begin{itemize}
\item Find 
\[g_{n+1} = \rho \pi(y) {\int_{u=x}^y}^*\big[g_n(x,u) + g_n(y,u)\big]\dx u + b(x,y),\]
set $n = n+1$.
\end{itemize}
\item Set final result:
\[f_n(x,y) = \frac{\pi(x)g_n(x,y)}{\rho \pi(y) + 1-\rho}.\]
\end{itemize}

We use the notation $f^*, g^*$ to denote the solutions to their respective fixed point equations. The aim of this section is to prove that the iteration converges and to find the error in the approximation of this function. It can be proven that the algorithm converges, provided that the initial choice for $g$ is bounded. Additionally, the error in the final iteration is given in the following Lemma.

\begin{lemma}
\label{lemma:PicardAnalytical}
Let $\delta > 0$ an arbitrary stopping criterion. Then the iteration converges, say at iteration $n$, and maintains the following error:
    \begin{align*}
        \big\vert g_n(x,y) - g^*(x,y)\big\vert\leq \frac{4\rho\delta \pi(y)}{1-\rho}.
    \end{align*}
\end{lemma}
\begin{proof}
 This proof is based on the same observation as the uniqueness proof in Propostion \ref{prop:FPE}. First we consider the convergence. Remark the following:
 \begin{align}
     \vert g_i(x,y) - g_{i-1}(x,y)\vert &\leq \rho\pi(y) \left\vert  {\int_{u=x}^y}^* \big[g_{i-1}(x,u)-g_{i-2}(x,u) + g_{i-1}(y,u)-g_{i-2}(y,u)\big]\dx u \right\vert \nonumber\\
     \label{eq:convergence1}
     &\leq \rho\pi(y) {\int_{u=x}^y}^* \big[\big \vert g_{i-1}(x,u)-g_{i-2}(x,u)\big \vert + \big\vert g_{i-1}(y,u)-g_{i-2}(y,u)\big \vert\big]\dx u \\
     \label{eq:convergence2}
     &\leq 2\rho\pi(y) \supnorm{g_{i-1}-g_{i-2}}. 
 \end{align}
This holds for general $i>1$, therefore we can substitute \eqref{eq:convergence2} for $i-1$ and $(x,u), \,(y,u)$ into \eqref{eq:convergence1}:
\begin{align*}
    \vert g_i(x,y) - g_{i-1}(x,y)\vert \leq 4\rho^2 \pi(y)\supnorm{g_{i-2}-g_{i-3}}{\int_{u=x}^y}^* \pi(u)\dx u.
\end{align*}
Again, we can substitute this inequality, for $i-2$ and $(x,u), \, (y,u)$ into \eqref{eq:convergence2}, this gives:
\begin{align*}
    \vert g_i(x,y) - g_{i-1}(x,y)\vert \leq 4\rho^3 \pi(y)\supnorm{g_{i-3}-g_{i-4}}{\int_{u=x}^y}^* \pi(u)\bigg\{{\int_{v=x}^u}^*\pi(v)\dx v+{\int_{\nu=y}^u}^*\pi(v)\dx v\bigg\}\dx u.
\end{align*}
Interchanging the order of integration of the first integrals reveals:
\begin{align*}
    \vert g_i(x,y) - g_{i-1}(x,y)\vert &\leq 4\rho^3\pi(y)\supnorm{g_{i-3}-g_{i-4}}\bigg\{{\int_{v=x}^y}^*\pi(v){\int_{u=v}^y}^* \pi(u)\dx u\dx v + {\int_{u=x}^y}^*\pi(u){\int_{v=y}^u}^* \pi(v)\dx v\dx u\bigg\}\\
    &=4\rho^3\pi(y)\supnorm{g_{i-3}-g_{i-4}}{\int_{u=x}^y}^* \pi(u)\dx u.
\end{align*}
By repeated application of this, we see that in the $i$-th iteration we have:
\begin{align*}
    \vert g_i(x,y) - g_{i-1}(x,y)\vert &\leq 4\rho^{i-1}\pi(y)\supnorm{g_{1}-g_{0}}{\int_{u=x}^y}^* \pi(u)\dx u,
\end{align*}
and hence for $i$ large enough this will be bounded by $\delta$ as $\rho < 1$, i.e. the algorithm converges.\\
For the error bound, we can use this result. Note namely, that the above derivation also implies that for any $i \geq n$:
\begin{align*}
    \vert g_i(x,y) - g_{i-1}(x,y)\vert &\leq 4\rho^{i-n}\pi(y)\supnorm{g_{n}-g_{n-1}}{\int_{u=x}^y}^* \pi(u)\dx u.
\end{align*}
Secondly, we can rewrite the $i$-th iterate as:
\begin{align*}
    g_{i}(x,y) = g_0(x,y) + \sum_{m=1}^i \big[g_m(x,y)-g_{m-1}(x,y)\big],
\end{align*}
with the limit $g^*(x,y) = \lim_{i\to\infty} g_{i}(x,y)$. Combining these two observations shows that:
\begin{align*}
    \big\vert g^*(x,y) - g_{n}(x,y)\big\vert &\leq \bigg\vert \sum_{m=1}^\infty \big[g_m(x,y)-g_{m-1}(x,y)\big] - \sum_{m=1}^n \big[g_m(x,y)-g_{m-1}(x,y)\big] \bigg\vert\\
    &\leq \sum_{m=n+1}^\infty\big\vert g_m(x,y)-g_{m-1}(x,y)\big\vert \leq  \sum_{m=n+1}^\infty 4\rho^{m-n}\pi(y)\supnorm{g_{n}-g_{n-1}}{\int_{u=x}^y}^* \pi(u)\dx u.
\end{align*}
The final bound now follows from the stopping criterion $\supnorm{g_n - g_{n-1}} \leq \delta$ and the trivial bound of $1$ for the integral over the density $\pi$.
\end{proof}
\begin{corollary}
The error in the final approximation of $f$ is bounded as follows:
    \begin{align*}
        \big\vert f_n(x,y) - f^*(x,y)\big\vert\leq \frac{\pi(x)\pi(y)4\rho\delta}{(1-\rho)^2}.
    \end{align*}
\end{corollary}
\begin{proof}
    This statement is a direct consequence of the relation:
    \begin{align*}
        f(x,y) = \frac{\pi(x)g(x,y)}{\rho \pi(y) + 1 - \rho},
    \end{align*}
    and thus:
    \begin{equation*}
        \vert f_{n}(x,y) - f^*(x,y) \vert \leq \frac{\pi(x)}{1-\rho} \vert g_n(x,y) - g^*(x,y)\vert. \qedhere
    \end{equation*}
\end{proof}

In many applications, analytically determining $g_{n+1}$ is (computationally) involved. In Appendix \ref{app:numerical} we therefore propose an iterative procedure using left Riemann sums for the approximation of the integrals. In the appendix we further derive bounds on the resulting error from this approximative procedure.

\section{Batch sojourn time}
\label{sec:sojourn}
In this section, we determine the mean batch sojourn time $\E[S^B]$, using the derived expression for $f$. We first condition on the server's location, batch size and location of the furthest customer:
\begin{align}
\label{eq:batchsojourn_conditioning}
\E[S^B] = \int_{u=0}^1 \big[\rho \pi(u) + 1-\rho\big]\int_{x=0}^1 \sum_{k=1}^{\infty} p_k k\pi(x)&\left({\int_{v = u}^x}^*\pi(v)\dx v\right)^{k-1}\\
&\cdot\E\big[S^B\big\vert X^B = x, K = k, S = u\big]\dx x \dx u.\nonumber
\end{align}
The conditional batch sojourn time is built up from the service time of the furthest customer $\E[B]$ plus the following 4 elements: (i) the extra waiting time for the customer at location $x$, generated by the travelling of the server from location $u$, $T(u,x)$; (ii) the extra waiting time generated by the customer possibly in service at the arrival instant. With probability $\rho\pi(u)/(\rho \pi(u) + 1-\rho)$ the server is working at location $u$, in which case the extra waiting time generated is $S^R(u,x)$; (iii) all $k-1$ customers who arrive in the same batch generate an extra waiting time. A customer arriving at location $y$ generates an extra waiting time of $S(y,x)$, and (iv) the extra waiting time generated by customers already present on the circle. This yields:
\begin{align}
\label{eq:condbatchsojourn}
    \E\big[S^B\big\vert &X^B = x, K = k, S = u\big] \nonumber\\
    &= \E[B] + \E[T(u,x)] + \frac{\rho\pi(u)}{\rho\pi(u)+ 1-\rho}\E[S^R(u,x)] \\
    &\quad + \frac{k-1}{{\int_{y=u}^x}^*\pi(y)\dx y}{\int_{y=u}^x}^*\pi(y)\E[S(y,x)]\dx y + {\int_{y=u}^x}^* f(y,u)\E[S(y,x)]\dx y. \nonumber
\end{align}

Using this as a basis, we can find an expression for $\E[S^B]$, see Theorem \ref{thm:batchsojourn}. The detailed proof of this statement is given in Appendix \ref{app:proofs}.
\begin{theorem}
\label{thm:batchsojourn}
    The mean batch sojourn time, $\E[S^B]$, is given by:
    \begin{equation}
    \label{eq:batchsojourn}
        \begin{aligned}[c]
             \E[S^B] &=  \E[B] + \frac{\alpha}{1-\rho} - \frac{\alpha}{1-\rho}\int_{u=0}^1 \big[\rho \pi(u) + 1-\rho\big]{\int_{x=0}^{1}}^*\big[\rho\pi(x) + 1 - \rho\big]\tilde{K}\left({\int_{v = u}^x}^*\pi(v)\dx v\right)\dx x \dx u\\
             &\quad + \frac{\rho(1+\rho)\E[B^2]}{2(1-\rho)\E[B]} -  \frac{\rho^2\E[B^2]}{\E[B](1-\rho)}\int_{x=0}^1\tilde{K}(x)\dx x \\
             &\quad +  \frac{1}{\lambda}\Big(\exp(\rho) - 1\Big) - \E[B]\exp(\rho) + \rho\E[B]\int_{x=0}^1\tilde{K}\left(x\right)\exp\left(\rho x\right)\dx x \\
            &\quad + \int_{u=0}^1 \big[\rho \pi(u) + 1-\rho\big]\int_{x=0}^1 \pi(x)\tilde{K}'\left({\int_{v = u}^x}^*\pi(v)\dx v\right){\int_{y=u}^x}^* f_K(y,u)\E[B]\exp\bigg(\rho{\int_{\xi = y}^x}^*\pi(\xi) \dx \xi\bigg)\dx y\dx x \dx u,
        \end{aligned}
    \end{equation}
    where $f_K$ is given by the solution to \eqref{eq:FPEnotsolve}.
\end{theorem}
\begin{remark}
    The $\alpha$-terms can be further rewritten. Expanding $[\rho \pi(u) + 1-\rho]\cdot [\rho \pi(x) + 1-\rho]$ shows that:
    \begin{align*}
        \frac{\alpha}{1-\rho} & \int_{u=0}^1 \big[\rho \pi(u) + 1-\rho\big]{\int_{x=0}^{1}}^*\big[\rho\pi(x) + 1 - \rho\big]\tilde{K}\left({\int_{v = u}^x}^*\pi(v)\dx v\right)\dx x \dx u\\
        &=\frac{\rho^2\alpha}{1-\rho}\int_{u=0}^1\pi(u){\int_{x=0}^{1}}\pi(x)\tilde{K}\left({\int_{v = u}^x}^*\pi(v)\dx v\right)\dx x \dx u+\rho\alpha\int_{u=0}^1{\int_{x=0}^{1}}\pi(x)\tilde{K}\left({\int_{v = u}^x}^*\pi(v)\dx v\right)\dx x \dx u \\
        &\quad + \rho\alpha\int_{u=0}^1\pi(u){\int_{x=0}^{1}}\tilde{K}\left({\int_{v = u}^x}^*\pi(v)\dx v\right)\dx x \dx u +  (1-\rho)\alpha\int_{u=0}^1{\int_{x=0}^{1}}\tilde{K}\left({\int_{v = u}^x}^*\pi(v)\dx v\right)\dx x \dx u.
    \end{align*}
    Using the substitution $\omega = {\int_{v = u}^x}^*\pi(v)\dx v$ in the first three integrals, we find the simplification:
    \begin{align*}
         \frac{\alpha}{1-\rho}\int_{u=0}^1 \big[\rho \pi(u) + 1-\rho\big]&{\int_{x=0}^{1}}^*\big[\rho\pi(x) + 1 - \rho\big]\tilde{K}\left({\int_{v = u}^x}^*\pi(v)\dx v\right)\dx x \dx u\\
         &=\alpha\frac{2\rho-\rho^2}{1-\rho}\E\left[\frac{1}{K+1}\right] + (1-\rho)\alpha\int_{u=0}^1{\int_{x=0}^{1}}\tilde{K}\left({\int_{v = u}^x}^*\pi(v)\dx v\right)\dx x \dx u.
    \end{align*}
    Observe that only the last term is still affected by the choice of $\pi$.
\end{remark}

\section{Time to delivery}
\label{sec:del}
The derivation of the expected time to delivery is more challenging. It, namely, can be the case that the server travels for more than an entire cycle before the batch is delivered. This happens when the furthest customer in a batch lies in-between the depot and the server. Because of this, the definitions of $T(\cdot,\cdot)$ and $S(\cdot, \cdot)$ are not sufficient for the derivation of the expected time to delivery. For more details, and a detailed derivation of Theorem \ref{thm:del}, we refer to Appendix \ref{app:Deliver}.

\begin{theorem}
    \label{thm:del}
    The expected time to delivery is given by:
    \begin{align}
    \label{eq:del}
     \E[D] &= \frac{1}{\lambda \E[K]}\cdot\big(\exp(\rho)-1\big)\cdot\int_{u=0}^1[\rho\pi(u)+  1-\rho]\cdot \tilde{K}'\left({\int_{\nu = u}^1}\pi(\nu)\dx \nu\right)\dx u \nonumber\\
     &\quad + \begin{aligned}[t]
        \frac{1}{\lambda \E[K]}\int_{u=0}^1&[\rho\pi(u)+  1-\rho]\cdot\bigg[\E[K]-\tilde{K}'\left({\int_{\nu = u}^1}\pi(\nu)\dx \nu\right)\bigg]\\
         &\cdot\bigg\{\exp\left(\rho \int_{\xi = u}^1 \pi(\xi)\dx \xi + \rho \right)-\rho\int_{\omega = u}^1\pi(\omega)\dx \omega\exp\left(\rho\int_{\xi = u}^1 \pi(\xi)\dx \xi\right) - 1\bigg\}\dx u
    \end{aligned}\nonumber\\
     &\quad +\frac{\alpha}{2(1-\rho)} + \frac{\alpha}{(1-\rho)}\int_{u=0}^1[\rho\pi(u)+  1-\rho]\cdot\bigg[1-\tilde{K}\left({\int_{\nu = u}^1}\pi(\nu)\dx \nu\right)\bigg]\dx u\nonumber\\
     &\quad + \rho\frac{\E[B^2]}{2\E[B]} + \frac{\rho^2\E[B^2]}{2(1-\rho)\E[B]} + \frac{\rho\E[B^2]}{(1-\rho)\E[B]}\E\left[\frac{K}{K+1}\right]\nonumber\\
    &\quad - \frac{\rho\E[B^2]}{\E[B]}\int_{\omega=0}^1\exp\left(\rho\omega\right)\cdot\bigg[1-\tilde{K}\left(\omega\right)\bigg]\dx \omega\\
    &\quad + \int_{u=0}^1[\rho\pi(u)+  1-\rho]\tilde{K}\left(\int_{\nu = u}^1\pi(\nu)\dx \nu\right)\int_{z=u}^1 f_K(z,u)\E[B]\exp\left(\rho\int_{\xi = z}^1\pi(\xi)\dx \xi \right)\dx z\dx u\nonumber\\
        &\quad + 
        \begin{aligned}[t]\int_{u=0}^1[\rho&\pi(u)+  1-\rho]\left[1-\tilde{K}\left(\int_{\nu = u}^1\pi(\nu)\dx \nu\right)\right]\\
        &\cdot\int_{z=u}^1 f_K(z,u)\E[B]\cdot\left[\exp\left(\rho\int_{\xi = z}^1\pi(\xi)\dx \xi+\rho\right)-\rho\int_{\omega=z}^1\pi(z)\exp\left(\rho\int_{\xi = z}^1\pi(\xi)\dx \xi\right)\right] \dx z\dx u
        \end{aligned}\nonumber\\
        &\quad + 
        \int_{u=0}^1[\rho\pi(u)+  1-\rho]\left[1-\tilde{K}\left(\int_{\nu = u}^1\pi(\nu)\dx \nu\right)\right]\cdot \int_{z=0}^u f_K(z,u)\E[B]\exp\left(\rho\int_{\xi = z}^1\pi(\xi)\dx \xi \right)\dx z\dx u.\nonumber
    \end{align}
    where $f_K$ is given by the solution to \eqref{eq:FPEnotsolve}.
\end{theorem}

\section{Limiting Behaviour}
\label{sec:limit}
In this section we focus on the limiting behaviour of the continuous polling model. In particular, we consider the light-traffic case of $\lambda \downarrow 0$ and the heavy-traffic limit when $\lambda \uparrow 1/(\E[B]\E[K])$ (which we interchangeably write as $\rho\uparrow 1$). This limiting behaviour allows for a better understanding of the system behaviour when pushed to either extremes and improves our understanding of (i) how the globally gated and exhaustive policy compare and (ii) the effect of the arrival location distribution of the system performance.

\subsection{Globally Gated}
The average limiting behaviour of the system under the globally gated policy follows almost directly from the results of Corollary \ref{cor:GGED} and \ref{cor:GGSB}. 

\textbf{Light traffic}\\
Under light traffic, it is readily verified that:
\begin{align*}
    \lim_{\lambda \downarrow 0}\E[C] = \alpha, \hspace{1cm} \lim_{\lambda \downarrow 0}\E[C^2] = \alpha^2,
\end{align*}
hence, recalling the results of Section \ref{sec:GG}:
\begin{align*}
    \lim_{\lambda \downarrow 0}\E[S^B] &= \E[B]\E[K] + \frac{3}{2} \alpha - \alpha \int_{x=0}^1 \tilde{K}(\Pi(x))\dx x, & \lim_{\lambda \downarrow 0}\E[D] &=\E[B]\E[K] + \frac{3}{2}\alpha.
\end{align*}

\textbf{Heavy traffic}\\
For the heavy-traffic limit we scale the moments of the cycle time accordingly, and use the following:
\begin{align*}
    \lim_{\rho\uparrow 1}(1-\rho)\E[C] &= \alpha\\
    \lim_{\rho\uparrow 1}(1-\rho)^2\E[C^2] &=\lim_{\rho\uparrow 1}\frac{1}{1+\rho}\bigg\{(1-\rho)\alpha + 2\rho\alpha^2 + \rho\frac{\E[B^2]}{\E[B]}\alpha + \rho\frac{\E[B]\E[K(K-1)]}{\E[K]}\alpha\bigg\}\\
    &=\alpha^2 + \frac{\E[B^2]}{2\E[B]}\alpha + \frac{\E[B]\E[K(K-1)]}{2\E[K]}\alpha.
\end{align*}
And thus $C^*$ satisfies:
\begin{align*}
    \lim_{\rho\uparrow 1}(1-\rho)\E[C^*] = \lim_{\rho\uparrow 1}\frac{(1-\rho)^2\E[C^2]}{(1-\rho)\E[C]} = \alpha + \frac{\E[B^2]}{2\E[B]} + \frac{\E[B]\E[K(K-1)]}{2\E[K]}.
\end{align*}
Consequently:
\begin{align*}
     \lim_{\rho\uparrow 1}(1-\rho)\E[S^B] &= \bigg\{\alpha + \frac{\E[B^2]}{2\E[B]} + \frac{\E[B]\E[K(K-1)]}{2\E[K]}\bigg\}\cdot \bigg\{\frac{1}{2}+\E\left[\frac{K}{K+1}\right]\bigg\};\\
     \lim_{\rho\uparrow 1}(1-\rho)\E[D] &= \frac{3}{2}\bigg\{\alpha + \frac{\E[B^2]}{2\E[B]} + \frac{\E[B]\E[K(K-1)]}{2\E[K]}\bigg\}.
\end{align*}
Note that under heavy traffic we have that $\E[S^B]$ converges to $\E[D]$ as $K\to\infty$, regardless of $\pi(\cdot)$ (this was not true for the case of arbitrary or light traffic). That is because the contribution of the travel time to the furthest customer in a batch disappears in the heavy-traffic limit.

\subsection{Exhaustive}
The average spread of customers on the circle, dictated by $f(\cdot)$, is a crucial element in the analysis of the continuous polling model. In the previous sections, we have derived an integral equation for $f(\cdot)$ and solved it {partly} analytically and partly algorithmically. However, the found expression lacks intuition. To develop a better understanding of $f(\cdot)$ and in turn, $\E[S^B]$, we also consider the light-traffic and heavy-traffic behaviour.\\

\textbf{Light traffic}\\
In light traffic, $\lambda \downarrow 0$, the system is mostly empty as $\E[L] = O(\lambda)$ (cf. Lemma \ref{lemma:ncust}). Consequently, an arriving batch will most likely arrive to an empty system. Moreover, the expected amount of work arriving during a service, or travelling of the server, is $O(\lambda)$. This implies that $\E[S(x,y)] = \E[B] + O(\lambda)$ and $\E[T(x,y)] = d(x,y) + O(\lambda)$. Viz., the average waiting time of a customer is not affected by any future arrivals under light traffic.\\
In combination with Lemma \ref{lemma:LittlesLaw_general} this allows for the derivation of the light-traffic behaviour.

\begin{proposition}
 \label{prop:lighttraffic}
    The continuous polling model under light traffic exhibits the following behaviour:
    \begin{align}
    \label{eq:lighttraffic}
        \lim_{\lambda\downarrow 0} \frac{f(x,y)}{\lambda\E[K]} = \alpha\pi(x)d(x,y) + \E[B]\frac{\E[K(K-1)]}{\E[K]}\pi(x)\pi(y) d(x,y).
    \end{align}
\end{proposition}
\begin{proof}
    We start from Lemma \ref{lemma:LittlesLaw_general} and focus on the right-hand side. As $\E[L] = O(\lambda)$, we infer that the expectation $\E[W_y\vert S = u]$ consists of the following factors: the travel time and the service time of the customers arriving in the batch. All other contributions are of the order $O(\lambda^2)$. Hence:
    \begin{align*}
        {\int_{u=x}^y}^*\big[1+O(\lambda)\big]f(x,u)\dx u = \lambda\E[K]\pi(x){\int_{u=x}^y}^*\bigg\{\alpha d(u,y) + \frac{\E[K(K-1)]}{\E[K]}\int_{z=u}^y \pi(z)\E[B]\dx z + O(\lambda^2)\bigg\}\dx u.
    \end{align*}
    We now take the derivative w.r.t. $y$ on both sides, applying that $\frac{\dx}{\dx y} d(u,y) = 1$ and that $\lim_{u\uparrow y} d(u,y) = 0$. This gives:
    \begin{align*}
    \big[1+O(\lambda)\big]f(x,y) = \lambda\E[K] \pi(x){\int_{u=x}^y}^*\bigg\{\alpha  + \frac{\E[B]\E[K(K-1)]}{\E[K]}\pi(y)\bigg\}\dx u + O(\lambda^2).
    \end{align*}
    The proof is concluded by dividing both sides by $\lambda\E[K]$ and taking $\lambda \downarrow 0$.
\end{proof}

\begin{remark}
    By Proposition \ref{prop:partialsol} it follows that:
    \[ \lim_{\lambda \downarrow 0} \frac{f_K(x,y)}{\lambda\E[K]} = \E[B]\frac{\E[K(K-1)]}{\E[K]}\pi(x)\pi(y) d(x,y).\]
\end{remark}

\begin{example}[Uniform arrival locations]
In the case of uniform arrivals \eqref{eq:lighttraffic} reduces to:
\[\lim_{\lambda \downarrow 0} \frac{f(x,y)}{\lambda\E[K]} = \alpha d(x,y) + \E[B]\frac{\E[K(K-1)]}{\E[K]}d(x,y).\]
Remark that this is also follows from Proposition 4.6 in \citet{Engels}.
\end{example}

Intuitively, \eqref{eq:lighttraffic} can be explained as follows.  Recall that all customers at location $x$ need to have arrived since the server's last visit to this location. During this time, the server has travelled a total distance $d(x,y)$. Under light traffic, no customers are served during this time and hence travelling this distance takes a time $\alpha d(x,y)$. Therefore, on average $\lambda\E[K]\pi(x)\cdot \alpha d(x,y)\dx x$ customers arrived during this time at location $[x,x+\dx x)$. \\
The second term comes from the event that a customer is in service at the time of arrival. This customer arrived alongside other customers in the batch. These so-called siblings might still be in the system at the time of arrival. Observe that, in light traffic, the server is working with probability $\rho\pi(y)/(1-\rho + \rho\pi(y)) \approx \rho \pi(y)$, given that the server is at location $y$.  On average, the customer in service arrived with a total of $\E[K(K-1)]/\E[K]$ other customers in the batch, a fraction $\pi(x)\dx x$ of which arrived in $[x, x+\dx x)$. These customers would only still be in the system if they are being served after the customer at location $y$, which would only happen when the batch arrived while the server was in $(x,y]$. The probability that this happened is $d(x,y)$ in light traffic.

Proposition \ref{prop:lighttraffic} shows that the contribution of $f_K$ to both the average batch sojourn time (cf. Theorem \ref{thm:batchsojourn}) and mean time to delivery (cf. Theorem \ref{thm:del}) disappears. This allows for the immediate derivation of the light-traffic behaviour, see Theorem \ref{thm:lighttraffic}. 
\begin{theorem}
\label{thm:lighttraffic}
    The light-traffic behaviour of the mean batch sojourn time satisfies:
    \begin{align}
    \label{eq:batchsojourn_lighttraffic1}
        \lim_{\lambda \downarrow 0}\E[S^B] = \mathbb{E}[K]\E[B] + \alpha - \alpha \int_{u=0}^1 \int_{x = 0}^1 \tilde{K}\left({\int_{\nu=u}^x}^*\pi(\nu)\dx \nu\right)\dx x \dx u.
    \end{align}
    The light-traffic behaviour of the expected time to delivery satisfies:
    \begin{align}
        \lim_{\lambda \downarrow 0 }\E[D] =  \E[B]\E[K] + \frac{3\alpha}{2} -\alpha\cdot \int_{u=0}^1 \tilde{K}\left(\int_{x=u}^1\pi(x)\dx x\right) \dx u.
    \end{align}
\end{theorem}
\begin{proof}
    The first statement follows from the observation that the batch sojourn time, under light traffic, is given by the service time of the batch plus the travel time to the furthest customer in the batch. The latter can be found by conditioning on the server's location $u$ and the location, $x$, of the furthest customer from this point. Here, we use that the server's location under light traffic is uniformly distributed. Hence:
    \begin{align*}
        \lim_{\lambda \downarrow 0}\E[S^B] &=  \E[K]\E[B] + \int_{u=0}^1 \int_{x=0}^1\sum_{k=1}^\infty k p_k\cdot \pi(x)\cdot \left({\int_{\nu=u}^x}^*\pi(\nu)\dx\nu\right)^{k-1}\cdot \alpha d(u,x)\dx x \dx u\\
        &=  \E[K]\E[B] + \int_{u=0}^1 {\int_{x=u^+}^{u^-}}^* \pi(x)\tilde{K}'\left({\int_{\nu=u}^x}^*\pi(\nu)\dx\nu\right)\cdot \alpha d(u,x)\dx x \dx u.
    \end{align*}
    Partial integration of the integral, using $\frac{\dx}{\dx x}d(u,x) = 1$, now gives the light-traffic result of the mean batch sojourn time.
    
    For the time to delivery, we argue that the only factors remaining under light traffic are that of the travel time and service time of the batch itself. The expected travel distance now consists of the expected residual distance to the depot, plus possibly another entire cycle. The extra cycle is only required when the furthest customer in a batch arrives between the depot and the server. As the server is idle with probability $1$ in the light-traffic limit, the server attains a uniform location at the arrival instant of a batch. Consequently, we have:
    \begin{align*}
        \lim_{\lambda \downarrow 0} \E[D] &= \E[B]\E[K] + \frac{\alpha}{2} + \alpha\cdot \int_{u=0}^1 \bigg\{ 1-\E\left[\left(\int_{x=u}^1\pi(x)\dx x\right)^K\right]\bigg\}\dx u.
    \end{align*}
    The proof is completed by using the definition of $\tilde{K}$.
\end{proof}
Remark that the light-traffic result for $\E[D]$ is equal to the light-traffic result for the  mean batch sojourn time under the globally gated policy. Since the time to delivery is greater than the batch sojourn time by definition, it immediately follows that the exhaustive service policy is preferred under light traffic.\\
Intuitively, this makes sense, as the server aimlessly walks around the circle under light traffic waiting for a new batch to arrive. As soon as this happens, the server must first walk to the depot under the globally gated policy, before she can start serving customers.

\begin{remark}
\label{rem:light_comp}
    These results allow for the comparison of the light-traffic behaviour under the globally gated (denoted by subscript $_{\rm GG}$) and exhaustive policy (denoted by subscript $_{\rm EX}$):
    \begin{align*}
        \lim_{\lambda \downarrow 0} \Big\{\E[S^B_{\rm GG}] - \E[S^B_{\rm EX}]\Big\} &= \frac{\alpha}{2} + \alpha\int_{u=0}^1 \int_{x=0}^1 \bigg\{\tilde{K}\left({\int_{\nu=u}^x}^*\pi(\nu)\dx \nu\right)- \tilde{K}\left({\int_{\nu=0}^x}^*\pi(\nu)\dx \nu\right)\bigg\}\dx x \dx u , \\
        \lim_{\lambda \downarrow 0} \Big\{\E[D_{\rm GG}] - \E[D_{\rm EX}]\Big\} &= \alpha\cdot \int_{u=0}^1 \tilde{K}\left(\int_{x=u}^1\pi(x)\dx x\right) \dx u.
    \end{align*} 
\end{remark}

\textbf{Heavy Traffic}\\
\label{sec:HeavyTraffic}
The analysis of the heavy-traffic behaviour of the system follows by linking $f$ to the case of uniform arrivals. Let $f_K(x,y)$ denote the solution to \eqref{eq:FPEnotsolve} and $f_K^U(x,y)$ refer to the solution of  \eqref{eq:FPEnotsolve} with $\pi \equiv 1$. By \citet{Engels}, we know that $f_K^U$ is proportional to the distance between the customer and the server, hence:
\begin{align}
\label{eq:solutionuniformfk}
    f_K^U(x,y) = \frac{\rho}{1-\rho}\frac{\E[K(K-1)]}{\E[K]}\cdot {\int_{u=x}^y}^*\dx u.
\end{align}
We now prove that the $f_K(x,y)$ can be related to $f_K^U(x,y)$, barring an error that disappears in the heavy-traffic limit. Formally:
\begin{lemma}
    Consider \eqref{eq:FPEnotsolve} and let $f_K(x,y)$ denote the solution to this integral equation. Furthermore let $f_K^U(x,y)$ as in \eqref{eq:solutionuniformfk}. Then the following holds:
    \begin{align*}
        \left\vert \frac{[\rho \pi(y)  + 1-\rho]f_K(x,y)}{\pi(x)} -\pi(y)f_K^U(\Pi(x),\Pi(y))\right\vert \leq  (1 + \rho)\pi(y)\frac{\rho\E[K(K-1)]}{\E[K]}.
    \end{align*}
\end{lemma}
\begin{proof}
     First, as $\pi$ is strictly positive, we propose the transformation:
     \begin{align*}
         g_K(x,y) = \frac{[\rho \pi(y)  + 1-\rho] f_K(x,y)}{\pi(x)},
     \end{align*}
     such that after cancelling a $\pi(x)$ term on either side of \eqref{eq:FPEnotsolve} results in the integral equation:
     \begin{align*}
         g_K(x,y) = \rho\pi(y){\int_{u=x}^y}^* \big[g_K(x,u) + g_K(y,u)\big]\dx u + \frac{\rho\E[K(K-1)]}{\E[K]} \pi(y){\int_{u=x}^y}^*[\rho \pi(u)  + 1-\rho]\dx u.
     \end{align*}
     Remark that since $f_K(x,y) \in C_{\pi}$ (thus $f_K(x,y) =  O(\pi(x))$), the function $g_K(x,y)$ is bounded.
     For $f_K^U(x,y)$ we rewrite the integral equation in terms of $\tilde{f}_K^U(x,y):=\pi(y)f_K^U(\Pi(x),\Pi(y))$:
     \begin{align*}
        f_K^U(x,y)&= \rho {\int_{u=x}^y}^*\big[f_K^U(x,u) +  {f}_K^U(y,u)\big]\dx u + \frac{\rho\E[K(K-1)]}{\E[K]} \pi(y){\int_{u=x}^{y}}^*\dx u\\
       \implies  \tilde{f}_K^U(x,y) &= \rho\pi(y){\int_{u=\Pi(x)}^{\Pi(y)}}^* \big[ f_K^U(\Pi(x),u) +  {f}_K^U(\Pi(y),u)\big]\dx u + \frac{\rho\E[K(K-1)]}{\E[K]} \pi(y){\int_{u=\Pi(x)}^{\Pi(y)}}^*\dx u \nonumber\\
         &=\rho\pi(y){\int_{\nu=x}^{y}}^* \big[\tilde{f}_K^U(x,\nu) + \tilde{f}_K^U(y,\nu)\big]\dx \nu + \frac{\rho\E[K(K-1)]}{\E[K]}\pi(y) {\int_{\nu=x}^{y}}^*\pi(\nu)\dx \nu,
     \end{align*}
     following from the substitution $u = \Pi(\nu)$.\\
     Remark the similarity between the integral equations for $g_K$ and $\tilde{f}_K^U$, and note that:
     \begin{align}
     \label{eq:uniformcoupling_eq1}
         \big \vert g_K(x,y) -   \tilde{f}_K^U(x,y) \big\vert &\leq \rho\pi(y){\int_{u=x}^{y}}^* \Big[\big\vert g_K(x,u) -\tilde{f}_K^U(x,u)\big\vert  + \big\vert g_K(y,u) - \tilde{f}_K^U(y,u)\big\vert \Big]\dx u \nonumber \\
         &\quad + (1-\rho)\frac{\rho\E[K(K-1)]}{\E[K]}\pi(y) \bigg\vert{\int_{u=x}^{y}}^* \big[1- \pi(u)\big]\dx u\bigg\vert,
    \intertext{We use that the integral over $\pi$ is bounded by $0$ and $1$, hence the last absolute value is upper bounded by $1$. Additionally, bounding the first integrand by the supremum norm of $g_K - \tilde{f}_K^U$ yields:}
    \label{eq:uniformcoupling_eq2}
         \big \vert g_K(x,y) -   \tilde{f}_K^U(x,y)\big \vert &\leq 2\rho\pi(y) \supnorm{g_K-\tilde{f}_K^U} + (1-\rho)\pi(y)\frac{\rho\E[K(K-1)]}{\E[K]}.
     \end{align}
     We now apply the inequality \eqref{eq:uniformcoupling_eq2} for the arguments, $(x,u)$ and $(y,u)$, to \eqref{eq:uniformcoupling_eq1}, to see:
     \begin{align}
     \label{eq:uniformcoupling_eq3}
         \big \vert g_K(x,y) -   \tilde{f}_K^U(x,y) \big\vert &\leq 4\rho^2 \pi(y) \supnorm{g_K-\tilde{f}_K^U}{\int_{u=x}^{y}}^* \pi(u)\dx u\\
         &\quad + 2(1-\rho)\pi(y)\frac{\rho^2\E[K(K-1)]}{\E[K]}{\int_{u=x}^{y}}^* \pi(u)\dx u +  (1-\rho)\pi(y)\frac{\rho\E[K(K-1)]}{\E[K]}.\nonumber
     \end{align}
     By applying \eqref{eq:uniformcoupling_eq3} to \eqref{eq:uniformcoupling_eq1} for $(x,u)$ and $(y,u)$, we see:
     \begin{align}
      \label{eq:uniformcoupling_eq4}
         \big \vert g_K(x,y) -   \tilde{f}_K^U(x,y) \big\vert &\leq 4\rho^3 \pi(y) \supnorm{g_K-\tilde{f}_K^U}{\int_{u=x}^{y}}^* \pi(u)\cdot \bigg\{{\int_{z=x}^{u}}^* \pi(z)\dx z + {\int_{z=y}^{u}}^* \pi(z)\dx z\bigg\}\dx u\nonumber\\
         &\quad 2(1-\rho)\pi(y)\frac{\rho^3\E[K(K-1)]}{\E[K]}{\int_{u=x}^{y}}^* \pi(u)\cdot \bigg\{{\int_{z=x}^{u}}^* \pi(z)\dx z + {\int_{z=y}^{u}}^* \pi(z)\dx z\bigg\}\dx u \nonumber\\
         &\quad + 2(1-\rho)\pi(y)\frac{\rho^2\E[K(K-1)]}{\E[K]}{\int_{u=x}^{y}}^* \pi(u)\dx u
          +  (1-\rho)\pi(y)\frac{\rho\E[K(K-1)]}{\E[K]}. 
     \end{align}
     We can use the following identity, following from interchanging the order of integration:
     \begin{align*}
         {\int_{u=x}^{y}}^* \pi(u)\cdot &\bigg\{{\int_{z=x}^{u}}^* \pi(z)\dx z + {\int_{z=y}^{u}}^* \pi(z)\dx z\bigg\}\dx u\\
        &={\int_{z=x}^{y}}^* \pi(z)\cdot {\int_{u=z}^{y}}^* \pi(u)\dx u\dx z +  {\int_{u=x}^{y}}^*\pi(u){\int_{z=y}^{u}}^* \pi(z)\dx z\dx u = {\int_{u=x}^{y}}^* \pi(u)\dx u.
     \end{align*}
     And hence \eqref{eq:uniformcoupling_eq4} simplifies to:
     \begin{align*}
         \big \vert g_K(x,y) -   \tilde{f}_K^U(x,y) \big\vert &\leq 4\rho^3\pi(y) \supnorm{g_K - \tilde{f}_K^U}\cdot  {\int_{u=x}^{y}}^* \pi(u)\dx u \\
         &\quad + (1-\rho)\pi(y)\frac{\rho\E[K(K-1)]}{\E[K]} +  2(1-\rho)\pi(y)\frac{\rho\E[K(K-1)]}{\E[K]}(\rho + \rho^2){\int_{u=x}^{y}}^* \pi(u)\dx u
     \end{align*}
     We can now repeat this by substituting this inequality, for $(x,u)$ and $(y,u)$, into \eqref{eq:uniformcoupling_eq1}. This, ultimately, gives for any $n\geq 1$:
     \begin{align*}
         \big \vert g_K(x,y) -   \tilde{f}_K^U(x,y) \big\vert &\leq 4\rho^{n+1} \pi(y)\supnorm{g_K-\tilde{f}_K^U} {\int_{u=x}^{y}}^*\pi(u)\dx u+  (1-\rho)\pi(y)\frac{\rho\E[K(K-1)]}{\E[K]}\\
         &\quad +2(1-\rho)\pi(y)\frac{\rho\E[K(K-1)]}{\E[K]}{\int_{u=x}^{y}}^* \pi(u)\dx u\cdot \sum_{i=1}^n \rho^i.
     \end{align*}
     As both $\tilde{f}_K^U$ and $g_K$ are bounded (see Proposition \ref{prop:FPE}), we also know that their absolute difference remains bounded. Furthermore, since $\rho < 1$, we now have that for $n\to\infty$ the above inequality simplifies to:
     \begin{align*}
         \big \vert g_K(x,y) -   \tilde{f}_K^U(x,y) \big\vert &\leq (1-\rho)\pi(y)\frac{\rho\E[K(K-1)]}{\E[K]} + 2\pi(y)\frac{\rho^2\E[K(K-1)]}{\E[K]}{\int_{u=x}^{y}}^* \pi(u)\dx u\cdot\\
         &\leq (1 + \rho)\pi(y)\frac{\rho\E[K(K-1)]}{\E[K]}.
     \end{align*}
     The proof can now be finished by using the transformation from $g_K$ to $f_K$.
\end{proof}
\begin{corollary}
\label{cor:heavytrafficf}
    As a direct consequence, the heavy-traffic behaviour for $f_K(x,y)$ can be found:
    \begin{align*}
        \lim_{\rho\uparrow 1}(1-\rho) \frac{[\rho \pi(y)  + 1-\rho]f_K(x,y)}{\pi(x)} &= \lim_{\rho\uparrow 1}(1-\rho)\pi(y)f_K^U(\Pi(x),\Pi(y)) = \frac{\E[K(K-1)]}{\E[K]}\pi(y){\int_{u=\Pi(x)}^{\Pi(y)}}^* \dx u \\
        &= \frac{\E[K(K-1)]}{\E[K]}\pi(y){\int_{z=x}^{y}}^* \pi(z)\dx z.
    \end{align*}
    Moreover, using Proposition \ref{prop:partialsol}, the heavy-traffic behaviour of $f$ follows:
    \begin{align}
        \lim_{\rho\uparrow 1}(1-\rho) f(x,y) &= \lambda\E[K]\alpha \pi(x) {\int_{u=x}^{y}}^* \pi(u)\dx u + \frac{\lambda\E[K]\pi(x)\E[B^2]}{\E[B]}{\int_{u=x}^{y}}^* \pi(u)\dx u \nonumber \\
        &\quad + \frac{\pi(x)\E[K(K-1)]}{\E[K]}{\int_{u=x}^{y}}^* \pi(u)\dx u.
    \end{align}
\end{corollary}

One can now substitute these results into Theorems \ref{thm:batchsojourn} and \ref{thm:del} to find the heavy-traffic behaviour of $\E[S^B]$ and $\E[D]$. 

\begin{theorem}
    The heavy-traffic behaviour of the continuous polling model with the exhaustive service policy is characterized by:
    \begin{align}
        \lim_{\rho \uparrow 1} (1-\rho)\E[S^B] &= \bigg(\alpha + \frac{\E[B^2]}{\E[B]} + \frac{\E[B]\E[K(K-1)]}{\E[K]}\bigg)\E\left[\frac{K}{K+1}\right],\\
       \lim_{\rho \uparrow 1} (1-\rho)\E[D] &= \bigg[\alpha +\frac{\E[B^2]}{\E[B]}+\frac{\E[B]\E[K(K-1)]}{\E[K]}\bigg]\E\left[\frac{K}{K+1}+\frac{1}{2}\right].
    \end{align}
\end{theorem}
\begin{proof}
    Starting from \eqref{eq:batchsojourn}, we have: 
    \begin{align*}
        \lim_{\rho \uparrow 1} (1-\rho)\E[S^B] &= \alpha - \alpha\int_{u=0}^1 \pi(u)\int_{x=0}^1\pi(x)\tilde{K}\left({\int_{\nu = u}^x}^*\pi(\nu)\dx\nu\right)\dx x \dx u+ \frac{\E[B^2]}{\E[B]}\int_{x=0}^1\big[1-\tilde{K}(x)\big]\dx x\\
        &\quad + \frac{\E[K(K-1)]\E[B]}{\E[K]}\cdot \int_{u=0}^1 \pi(u)\int_{x=0}^1\pi(x)\tilde{K}'\left({\int_{\nu = u}^x}^*\pi(\nu)\dx\nu\right)\\
        &\quad\quad\quad\quad\cdot {\int_{y=u}^x}^* \pi(y)\cdot \E[B]\exp\left({\int_{\xi = y}^x}^*\pi(\xi)\dx \xi\right)\lim_{\rho\uparrow 1} (1-\rho)f_K(y,u)\dx y \dx x \dx u.
    \end{align*}
    Remark that the first integration term results in $\E[1/(K+1)]$. Additionally, we can use Corollary \ref{cor:heavytrafficf} for the heavy-traffic limit of $f_K$: 
    \begin{align*}
        \lim_{\rho \uparrow 1} (1-\rho)\E[S^B] 
        &= \Big(\alpha + \frac{\E[B^2]}{\E[B]}\Big)\E\left[\frac{K}{K+1}\right]\\
        &\quad + \frac{\E[K(K-1)]\E[B]}{\E[K]}\cdot \int_{u=0}^1 \pi(u)\int_{x=0}^1\pi(x)\tilde{K}'\left({\int_{\nu = u}^x}^*\pi(\nu)\dx\nu\right)\\
        &\quad\quad\quad\quad\cdot {\int_{y=u}^x}^* \E[B]\exp\left({\int_{\xi = y}^x}^*\pi(\xi)\dx \xi\right)\cdot \pi(y)\frac{\E[K(K-1)]}{\E[K]}{\int_{\nu=y}^u}^*\pi(\nu)\dx\nu\dx y \dx x \dx u.
    \end{align*}
    Since $y$ is in-between $u$ and $x$, we can split the integration over $\nu$ as follows: ${\int_{\nu=y}^{u}}^* \pi(\nu)\dx \nu = {\int_{\nu=y}^{x}}^* \pi(\nu)\dx \nu + {\int_{\nu=x}^{u}}^* \pi(\nu)\dx \nu$. Afterwards, we use the substitution $\zeta = {\int_{\nu=y}^{x}}^* \pi(\nu)\dx \nu$:
    \begin{align*}
        \lim_{\rho \uparrow 1} (1-\rho)\E[S^B] 
        &= \Big(\alpha + \frac{\E[B^2]}{\E[B]}\Big)\E\left[\frac{K}{K+1}\right]\\
        &\quad + \frac{\E[K(K-1)]\E[B]}{\E[K]}\cdot \int_{u=0}^1 \pi(u)\int_{x=0}^1\pi(x)\tilde{K}'\left({\int_{\nu = u}^x}^*\pi(\nu)\dx\nu\right)\\
        &\quad\quad\quad\quad\cdot {\int_{\zeta=0}^{{\int_{\nu = u}^x}^*\pi(\nu)\dx \nu}}^* \frac{\E[B]\E[K(K-1)]}{\E[K]} \exp\left(\zeta\right)\cdot \bigg\{\zeta + {\int_{\nu=x}^u}^*\pi(\nu)\dx\nu\bigg\}\dx \zeta \dx x \dx u.
    \end{align*}
    We now use the substitution: $\omega = {\int_{\nu=u}^{x}}^* \pi(\nu)\dx \nu$ in combination with: ${\int_{\nu=x}^{u}}^* \pi(\nu)\dx \nu = 1- {\int_{\nu=u}^{x}}^* \pi(\nu)\dx \nu$:
    \begin{align*}
         \lim_{\rho \uparrow 1} (1-\rho)\E[S^B] &= \Big(\alpha + \frac{\E[B^2]}{\E[B]}\Big)\E\left[\frac{K}{K+1}\right]\\
         &\quad + \frac{\E[K(K-1)]\E[B]}{\E[K]}\cdot \int_{u=0}^1 \pi(u)\int_{\omega=0}^1\tilde{K}'\left(\omega\right)\cdot {\int_{\zeta=0}^{\omega}} \{\zeta +1-\omega\}\cdot \exp\left(\zeta\right)\dx \zeta \dx x \dx u\\
         &= \Big(\alpha + \frac{\E[B^2]}{\E[B]}\Big)\E\left[\frac{K}{K+1}\right] + \frac{\E[K(K-1)]\E[B]}{\E[K]}\cdot \int_{\omega=0}^1\omega\tilde{K}'\left(\omega\right)\dx \omega.
    \end{align*}
    The proof of the first statement now follows from evaluating the last integral.\\
    Along the same lines, the heavy-traffic behaviour of $\E[D]$ can be derived, starting from \eqref{eq:del}:
    \begin{align*}
         \lim_{\rho \uparrow 1} (1-\rho)\E[D] &= \frac{\alpha}{2} + \alpha \int_{u=0}^1 \pi(u)\bigg[1-\tilde{K}\left(\int_{\nu = u}^1\pi(\nu)\dx\nu\right)\bigg]\dx u + \frac{\E[B^2]}{2\E[B]} + \frac{\E[B^2]}{\E[B]}\E\left[\frac{K}{K+1}\right]\\
         &\quad +
         \begin{aligned}[t]
             \int_{u=0}^1 \pi(u)&\tilde{K}\left({\int_{\nu = u}^1}^*\pi(\nu)\dx\nu\right)\int_{z=u}^1\pi(z)\\
             &\cdot \frac{\E[K(K-1)]\E[B]}{\E[K]}{\int_{\nu=z}^{u}}^* \pi(\nu)\dx \nu\exp\left(\int_{\xi = z}^1\pi(\xi)\dx \xi\right)\dx z \dx u
         \end{aligned} \\
         &\quad +
         \begin{aligned}[t]
         \int_{u=0}^1 \pi(u)&\bigg[1-\tilde{K}\left({\int_{\nu = u}^1}\pi(\nu)\dx\nu\right)\bigg]\int_{z=u}^1\pi(z) \frac{\E[K(K-1)]\E[B]}{\E[K]}\cdot{\int_{\nu=z}^{u}}^* \pi(\nu)\dx \nu\\
         &\cdot\bigg[\exp\left(1+\int_{\xi = z}^1\pi(\xi)\dx \xi\right) - \int_{\omega = z}^1\pi(\omega)\dx \omega \exp\exp\left(\int_{\xi = z}^1\pi(\xi)\dx \xi\right)\bigg]\dx z \dx u
         \end{aligned}\\
         &\quad +
        \begin{aligned}[t]
             \int_{u=0}^1 \pi(u)&\bigg[1-\tilde{K}\left({\int_{\nu = u}^1}\pi(\nu)\dx\nu\right)\bigg]\int_{z=0}^u\pi(z)\\
             &\cdot \frac{\E[K(K-1)]\E[B]}{\E[K]}{\int_{\nu=z}^{u}}^* \pi(\nu)\dx \nu\exp\left(\int_{\xi = z}^1\pi(\xi)\dx \xi\right)\dx z \dx u.
         \end{aligned}
    \end{align*}
    Each of the last three terms can be simplified. We apply the following tricks:
    \begin{itemize}
        \item For $z \leq u < 1$ (i.e. the last term), we use that: ${\int_{\nu=z}^{u}} \pi(\nu)\dx \nu = {\int_{\nu=z}^{1}} \pi(\nu)\dx \nu - {\int_{\nu=u}^{1}} \pi(\nu)\dx \nu$.
        \item For $u<z<1$ (i.e. integrals on second and fourth line), we rewrite: ${\int_{\nu=z}^{u}}^* \pi(\nu)\dx \nu = {\int_{\nu=z}^{1}} \pi(\nu)\dx \nu + 1 - {\int_{\nu=u}^{1}} \pi(\nu)\dx \nu$.
        \item We apply the substitution $\zeta = {\int_{\nu=z}^{1}} \pi(\nu)\dx \nu$ and then use the substitution $\eta = {\int_{\nu=u}^{1}} \pi(\nu)\dx \nu$.
    \end{itemize}
    This gives the following:
    \begin{align*}
         \lim_{\rho \uparrow 1} (1-\rho)\E[D] &= \bigg(\alpha + \frac{\E[B^2}{\E[B]}\bigg)\E\left[\frac{K}{K+1}+\frac{1}{2}\right]\\
         &\quad + \frac{\E[K(K-1)]\E[B]}{\E[K]}\cdot\int_{\eta=0}^1 \tilde{K}(\eta)\int_{\zeta=0}^\eta (\zeta + 1 -\eta)\cdot \exp\left(\zeta\right)\dx \zeta \dx \eta\\
         &\quad + \frac{\E[K(K-1)]\E[B]}{\E[K]}\cdot\int_{\eta=0}^1 \big[1-\tilde{K}(\eta)\big]\int_{\zeta=0}^\eta (\zeta + 1 -\eta)\cdot \Big[\exp\left(\zeta + 1\right) - \zeta \exp(\zeta)\Big]\dx
         \zeta \dx \eta\\
         &\quad + \frac{\E[K(K-1)]\E[B]}{\E[K]}\cdot\int_{\eta=0}^1 \big[1-\tilde{K}(\eta)\big]\int_{\zeta=\eta}^1 (\zeta -\eta)\cdot \exp\left(\zeta\right)\dx \zeta \dx \eta\\
         &=\bigg(\alpha + \frac{\E[B^2}{\E[B]}\bigg)\E\left[\frac{K}{K+1}+\frac{1}{2}\right]+ \frac{\E[K(K-1)]\E[B]}{\E[K]}\cdot\int_{\eta=0}^1 \tilde{K}(\eta)\dx \eta\\
         &\quad + \frac{\E[K(K-1)]\E[B]}{\E[K]}\cdot\int_{\eta=0}^1 \big[1-\tilde{K}(\eta)\big]\cdot\big\{1-\exp(\eta)+\eta+\exp(1)\eta\big\}\dx \eta\\
         &\quad + \frac{\E[K(K-1)]\E[B]}{\E[K]}\cdot\int_{\eta=0}^1 \big[1-\tilde{K}(\eta)\big]\cdot\big\{\exp(\eta) - \exp(1)\eta\big\}\dx \eta.
    \end{align*}
    Remark that all exponentials in the last two terms cancel against one another. Adding together the remaining terms then shows:
    \begin{align*}
        \lim_{\rho \uparrow 1} (1-\rho)\E[D] &= \bigg(\alpha + \frac{\E[B^2}{\E[B]}\bigg)\E\left[\frac{K}{K+1}+\frac{1}{2}\right] + \frac{\E[K(K-1)]\E[B]}{\E[K]}\cdot\int_{\eta=0}^1 \big[1-\tilde{K}(\eta)\big]\dx \eta\\
        &\quad + \frac{\E[K(K-1)]\E[B]}{\E[K]}\cdot\int_{\eta=0}^1 \eta\dx \eta.
    \end{align*}
    The proof is concluded by evaluating the integrals.
\end{proof}
\begin{remark}
\label{rem:heavy_comp}
    We now give a comparison of the heavy-traffic behaviour under the globally gated (denoted by subscript $_{\rm GG}$) and exhaustive policy (denoted by subscript $_{\rm EX}$). For the batch sojourn time we have:
    \begin{align*}
        \lim_{\rho \to 1} (1-\rho)\Big(\E[S^B_{\rm GG}] - \E[S^B_{\rm EX}]\Big) = \frac{\alpha}{2} + \bigg[\frac{\E[B^2]}{2\E[B]} + \frac{\E[B]\E[K(K-1)]}{2\E[K]}\bigg]\cdot\left(\frac{1}{2} - \E\left[\frac{K}{K+1}\right]\right).
    \end{align*}
    Note that the second term is negative whenever $K>1$. Therefore, in cases where $\alpha$ is relatively large, the exhaustive service policy is preferred. However, when the service times are large and $K$ is not identically one, the globally gated policy is preferred. \\
    For the time to delivery we have:
    \begin{align*}
        \lim_{\rho \to 1} (1-\rho)\Big(\E[D_{\rm GG}]& - \E[D_{\rm EX}]\Big)\\ 
        &= \alpha\E\left[\frac{1}{K+1}\right] + \bigg[\frac{\E[B^2]}{2\E[B]} + \frac{\E[B]\E[K(K-1)]}{2\E[K]}\bigg]\cdot\left(\frac{1}{4} - \E\left[\frac{K}{K+1}\right]\right).
    \end{align*}
    Note that the second term is always negative, therefore for $\alpha$ large the exhaustive service policy is preferred, while the opposite is true when $\alpha$ is small (relative to the service times and batch sizes).
\end{remark}

\section{Numerical Results}
\label{sec:numericalresults}
The results in Sections \ref{sec:sojourn} and \ref{sec:del} rely on the algorithmic approach to solving the integral equation in \eqref{eq:integraleq}. Exact results are derived in Section \ref{sec:limit} for the light- and heavy-traffic limits. Due to the algorithmic approach it is, thus far, unclear what the exact effects are of certain parameters on the system performance, in particular that of the arrival location distribution. This section serves two functions: (i) gaining more insight in the performance effect of different parameters, especially under moderate loads and (ii) obtaining a better understanding of the results in practical examples and increasing the intuition for the current model.

This section consists of three parts. First we consider a number of carefully chosen examples that emphasize the differences in performance between different parameter choices. Afterwards, we investigate the effect of the batch size on the system performance. Lastly, we consider an example inspired by a warehouse logistics application and investigate the effect of the arrival location distribution in this practical example. Throughout this section we use a numerical variant of the algorithm of Section \ref{sec:successive}, see Appendix \ref{app:numerical}.

\subsection*{Illustrative examples}
Consider the continuous polling model with a total travel time of $\alpha = 1$, constant service times $B \equiv 0.01$ and $B\equiv 1$ and fixed batch size $K\equiv 15$. These parameters are chosen to highlight the impact of the arrival location distribution and the service policy on the system performance. Specifically, we consider the following choices of the arrival location distribution: Uniform on [0,1], Uniform on [0,0.01], Uniform on [0.99,1], Beta(3,3) (symmetric distribution with peak at $0.5$: \includegraphics[height=1em]{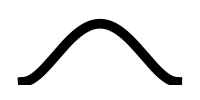}), Beta(25,1.5) (concentrated distribution with peak just before $1$: \includegraphics[height=1em]{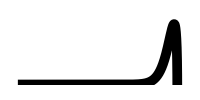}) and a reverse symmetric triangular distribution with a valley at $0.5$ (\includegraphics[height=1em]{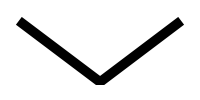}). \\
For the uniform arrival location on $[0,1]$ we determine the performance exactly (using \citet{Engels}), in contrast to other distributions, for which we use the algorithmic approach to determine the expected batch sojourn time under the exhaustive policy. The performance under the globally gated policy is determined exactly.\\
Figures \ref{fig:Extremecases_SB}-\ref{fig:Extremecases2_ED} plot the (approximated) expected performance of the systems across different values of $\rho$. The results illustrate a number of interesting features:
\begin{enumerate}[label = (\roman*)]
    \item When service times are small, the exhaustive service policy outperforms the globally gated policy with respect to the average batch sojourn time (Figure \ref{fig:Extremecases_SB}). The opposite is true for the case of large service times and moderate to high loads of the system (Figure \ref{fig:Extremecases2_SB}). This is in line with Remarks \ref{rem:light_comp} and \ref{rem:heavy_comp}. This change in preferred policy highlights an interesting trade-off in these systems. Under the globally gated policy, the server needs to travel more before serving the last customer in a batch. But, under the exhaustive service policy, the server tends to serve more customers before serving the last customer in a batch, as she now serves \emph{any} customer she encounters.\\
    For the expected time to delivery (Figures \ref{fig:Extremecases_ED} and \ref{fig:Extremecases2_ED}), the same ordering holds, although the advantage of the exhaustive service policy seems to be negligible under heavy traffic. This, however, is due to the parameter choices. For smaller batch sizes, this difference would be more pronounced.
    \item The difference in the expected batch sojourn time between the different location distributions under the exhaustive service policy is relatively small, especially for the heavy-traffic case. For light traffic, the examples still display a relative difference of up to $50\%$. This stems from the differences in the travel time of an idle server (i.e. uniform location) to the furthest customer in a batch. Due to the rotational symmetry of the model, the arrival location distributions Uniform[0.99,1] and Uniform[0, 0.01] arrival location distributions are indistinguishable with respect to the (expected) batch sojourn time.\\
    The time to delivery displays slightly different behaviour. Here, the differences are still quite noticeable for loads up to $0.85$. This is due to the differences between the probabilities that an arbitrary batch is delivered in the current or next cycle. For the Uniform[0.99, 1] case, a large fraction of the batches can be delivered in the same cycle, as batches arriving during the idle walking time of the server (from 0 to 0.99) can all immediately be delivered. In the Uniform[0, 0.01] case the opposite is true, batches arriving during the idle walking time (from 0.01 to 1) need to wait for the residual time \emph{plus} another entire walking time from 0.01 to 1. Interestingly, these two cases give the most extreme differences, unlike for the expected batch sojourn time, where these were indistinguishable.
    \item The globally gated policy shows much larger differences in the average batch sojourn time, driven by the differences in the expected travel time to the furthest customer in a batch. Therefore, cases with smaller service times (and thus relatively large travel times) result in larger differences between arrival location distributions. When all customers arrive close to the depot the average batch sojourn time is low. For larger loads, this difference remains, yet is less pronounced.\\
    As discussed in Remark \ref{rem:GG_ED}, the average time to delivery is indifferent to the arrival location distribution.
\end{enumerate}

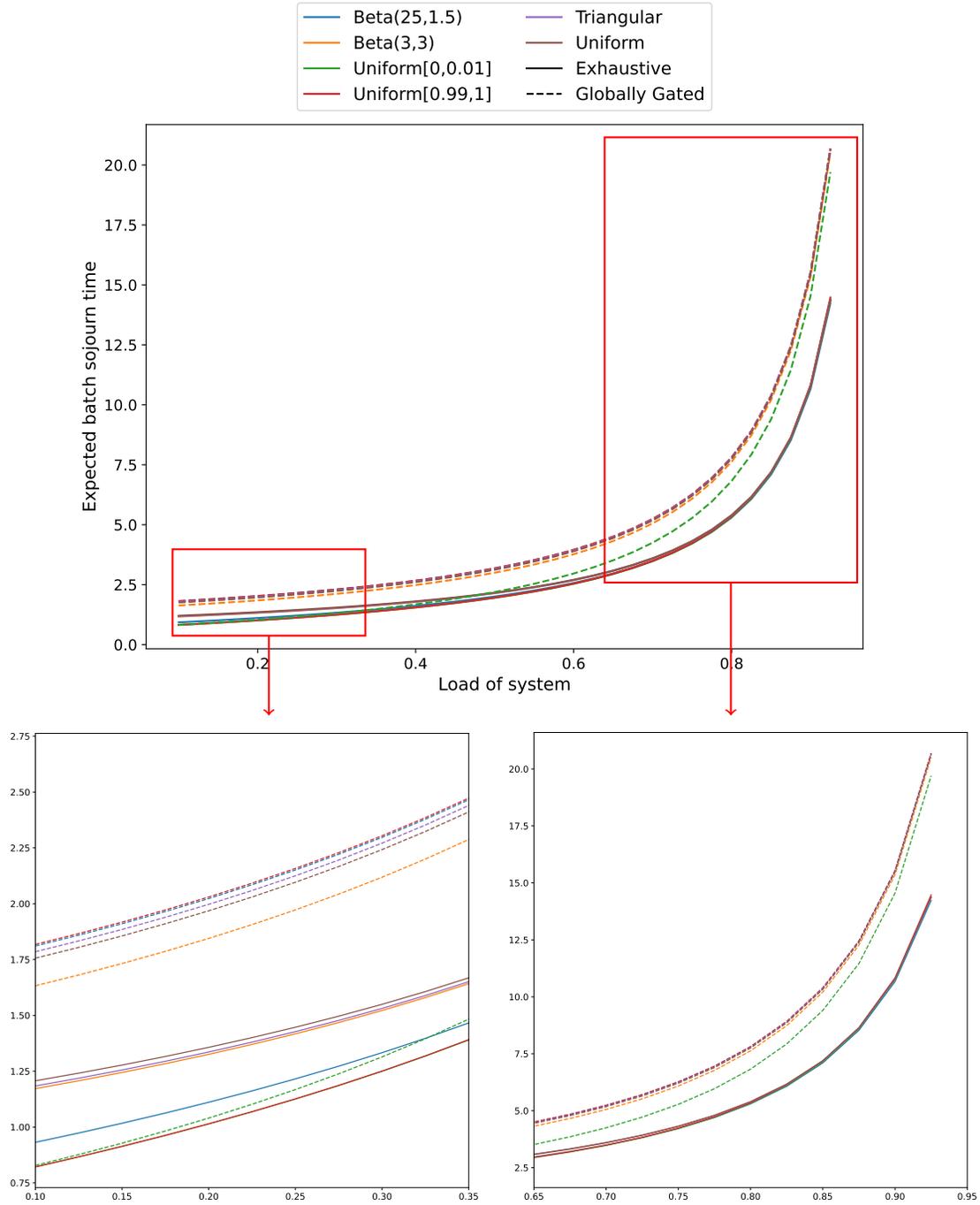
\begin{figure}[H]
    \centering
    \begin{tikzpicture}
        \node[anchor=south west] (main) at (0,0) {\includegraphics[width=0.8\textwidth]{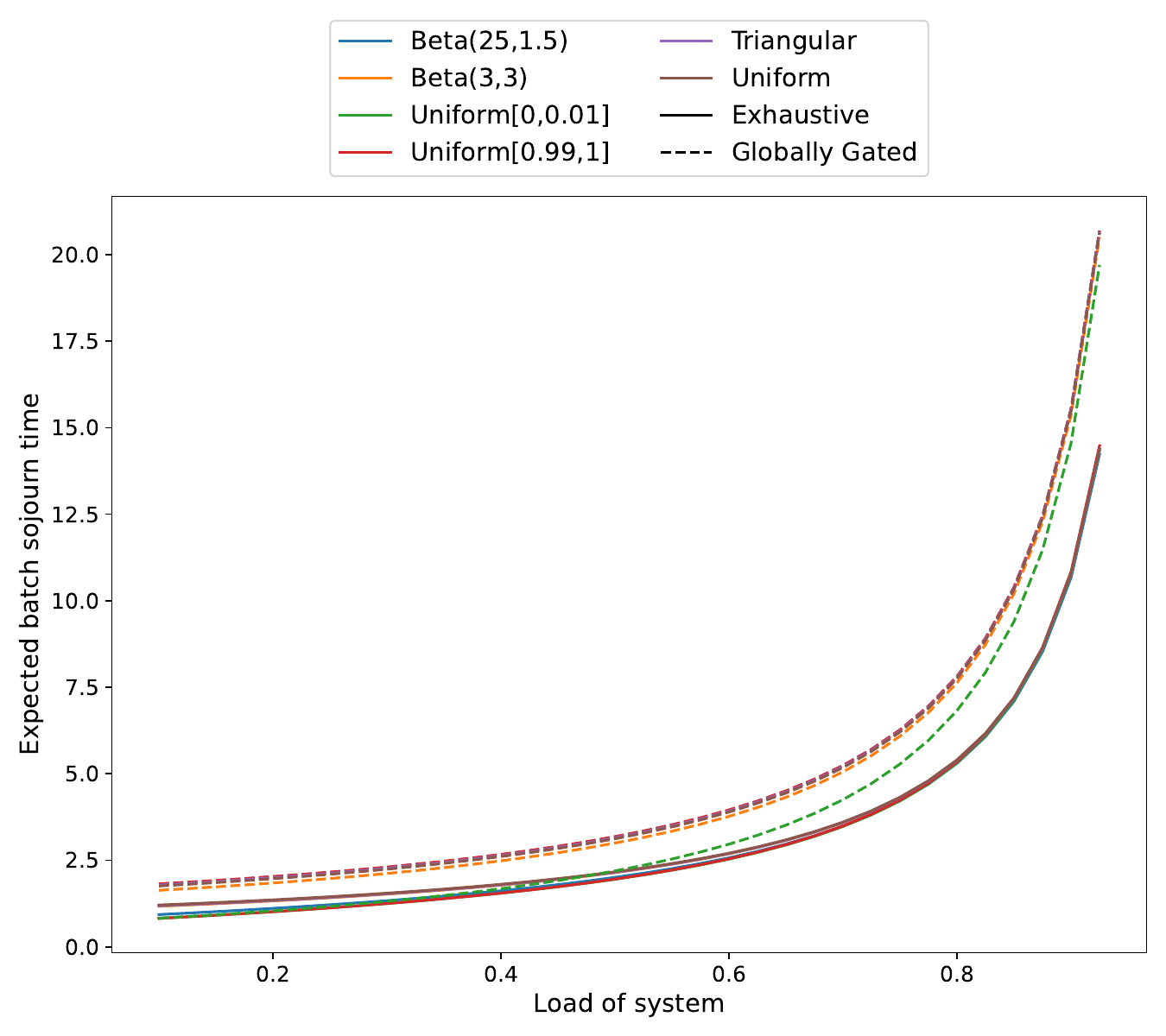}};
        
        \node[anchor=north west] (zoom) at (-1,0) {\includegraphics[width=0.48\textwidth]{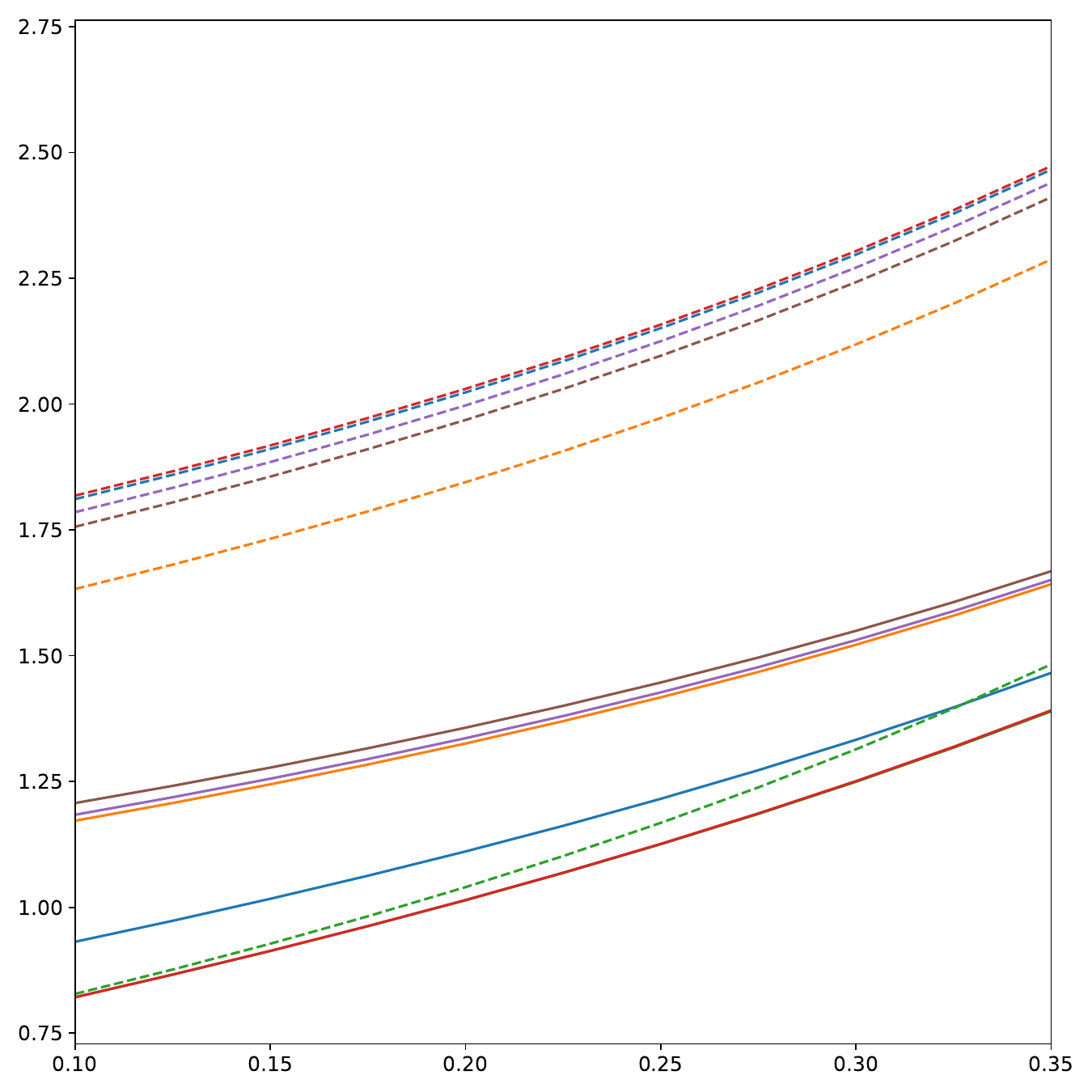}};
        
        \draw[red,thick] (1.7,1.2) rectangle (4.6,2.5);
        
        \draw[red,thick,->] (3.15,1.2) -- (3.15, 0);

        \node[anchor=north west] (zoom) at (6.5,0) {\includegraphics[width=0.48\textwidth]{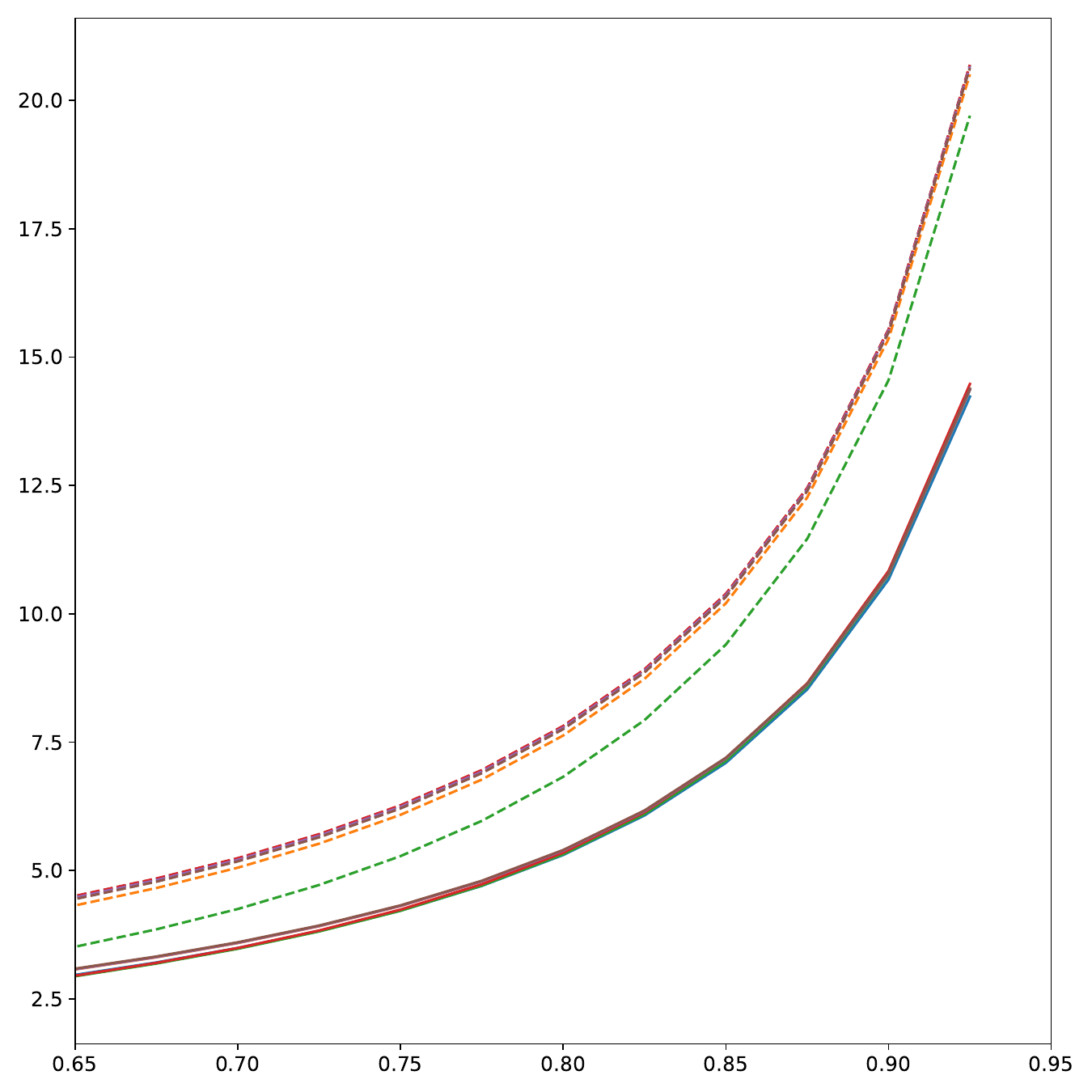}};
        
        \draw[red,thick] (8.2,2) rectangle (12,8.7);
        
        \draw[red,thick,->] (10.1,2) -- (10.1, 0);
    \end{tikzpicture}
    \caption{The expected batch sojourn time of the continuous polling model, under $B\equiv 0.01$, for varying arrival location distributions and loads. The results under exhaustive service are presented as a solid line, the dashed line refers to the globally gated policy.}
    \label{fig:Extremecases_SB}
\end{figure}
\begin{figure}[H]
    \centering
    \begin{tikzpicture}
        \node[anchor=south west] (main) at (0,0) {\includegraphics[width=0.8\textwidth]{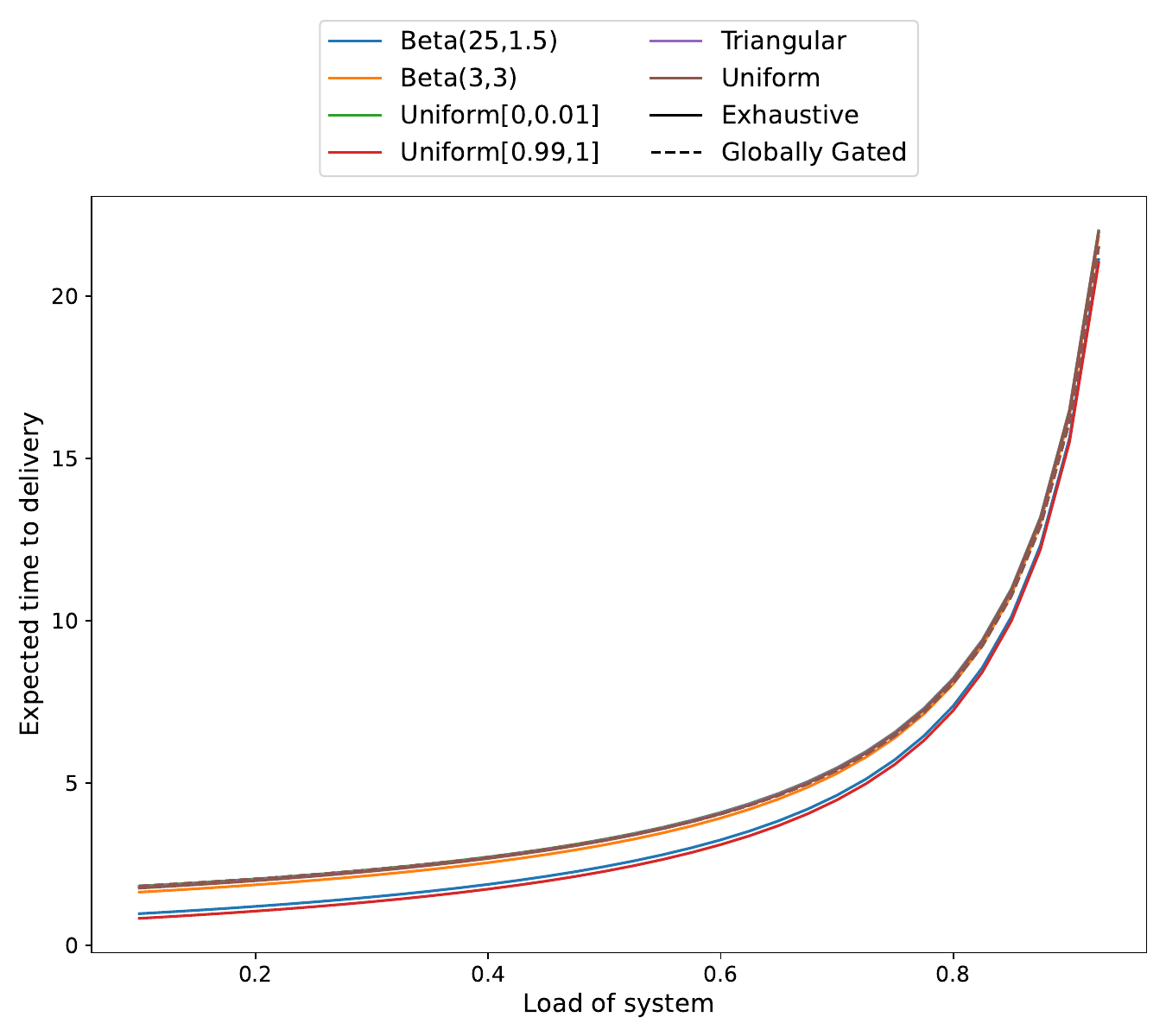}};
        
        \node[anchor=north west] (zoom) at (-1,0) {\includegraphics[width=0.48\textwidth]{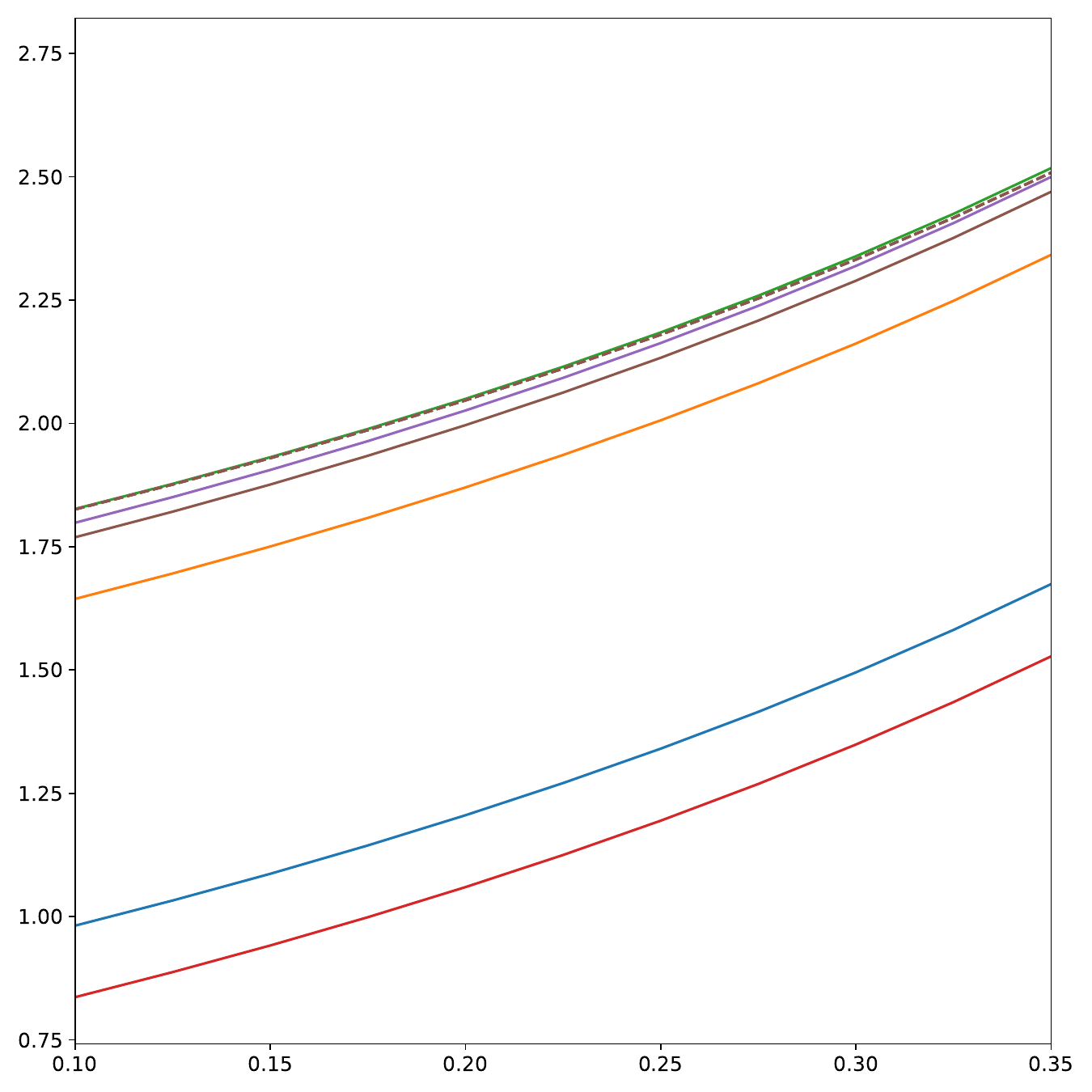}};
        
        \draw[red,thick] (1.5,1.2) rectangle (4.6,2.5);
        
        \draw[red,thick,->] (3.05,1.2) -- (3.05, 0);

        \node[anchor=north west] (zoom) at (6.5,0) {\includegraphics[width=0.48\textwidth]{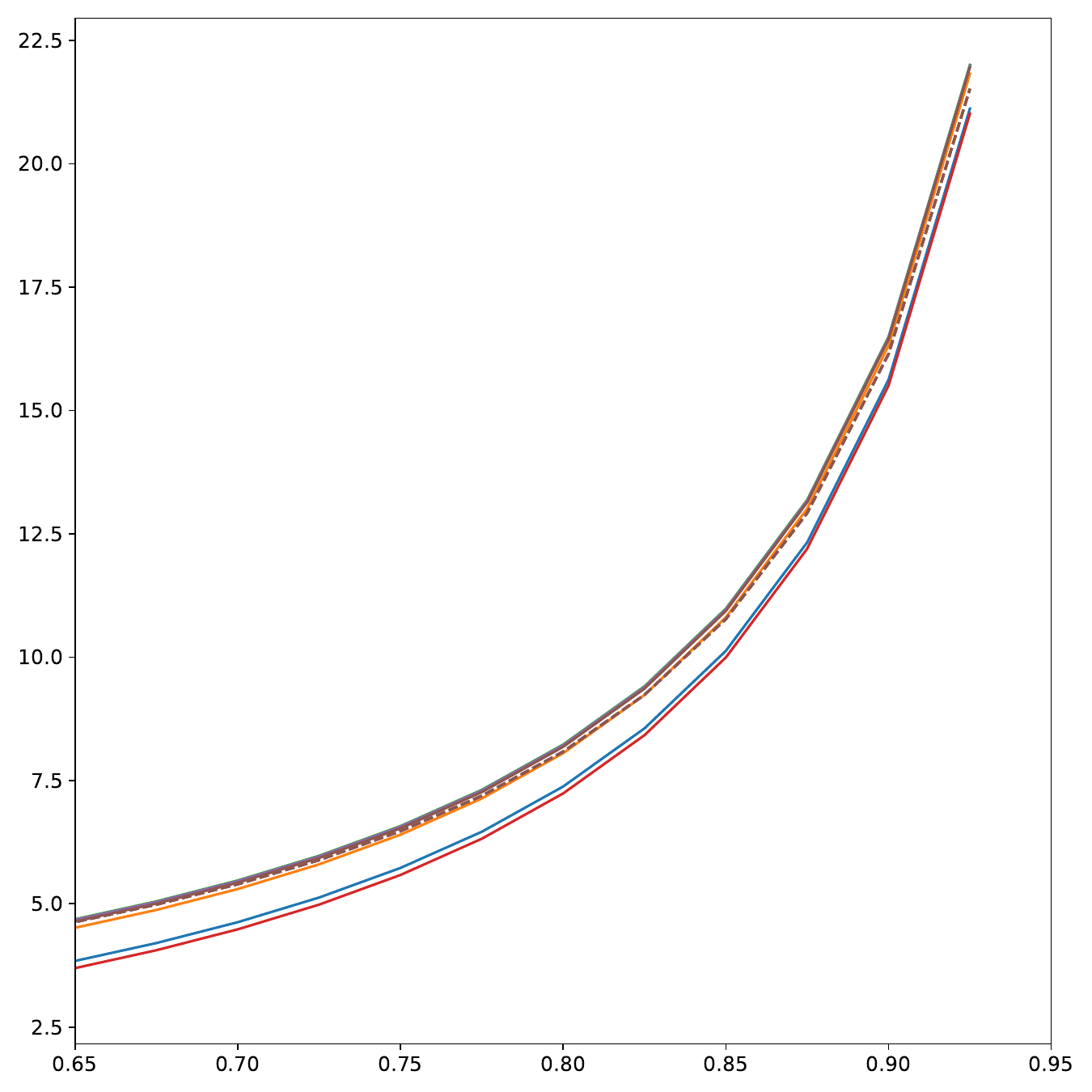}};
        
        \draw[red,thick] (8.2,2) rectangle (12,8.7);
        
        \draw[red,thick,->] (10.1,2) -- (10.1, 0);
    \end{tikzpicture}
    \caption{The expected time to delivery of the continuous polling model, under $B\equiv 0.01$, for varying arrival location distributions and loads. The results under exhaustive service are presented as a solid line, the dashed line refers to the globally gated policy.}
    \label{fig:Extremecases_ED}
\end{figure}

\begin{figure}[H]
    \centering
    \begin{tikzpicture}
        \node[anchor=south west] (main) at (0,0) {\includegraphics[width=0.8\textwidth]{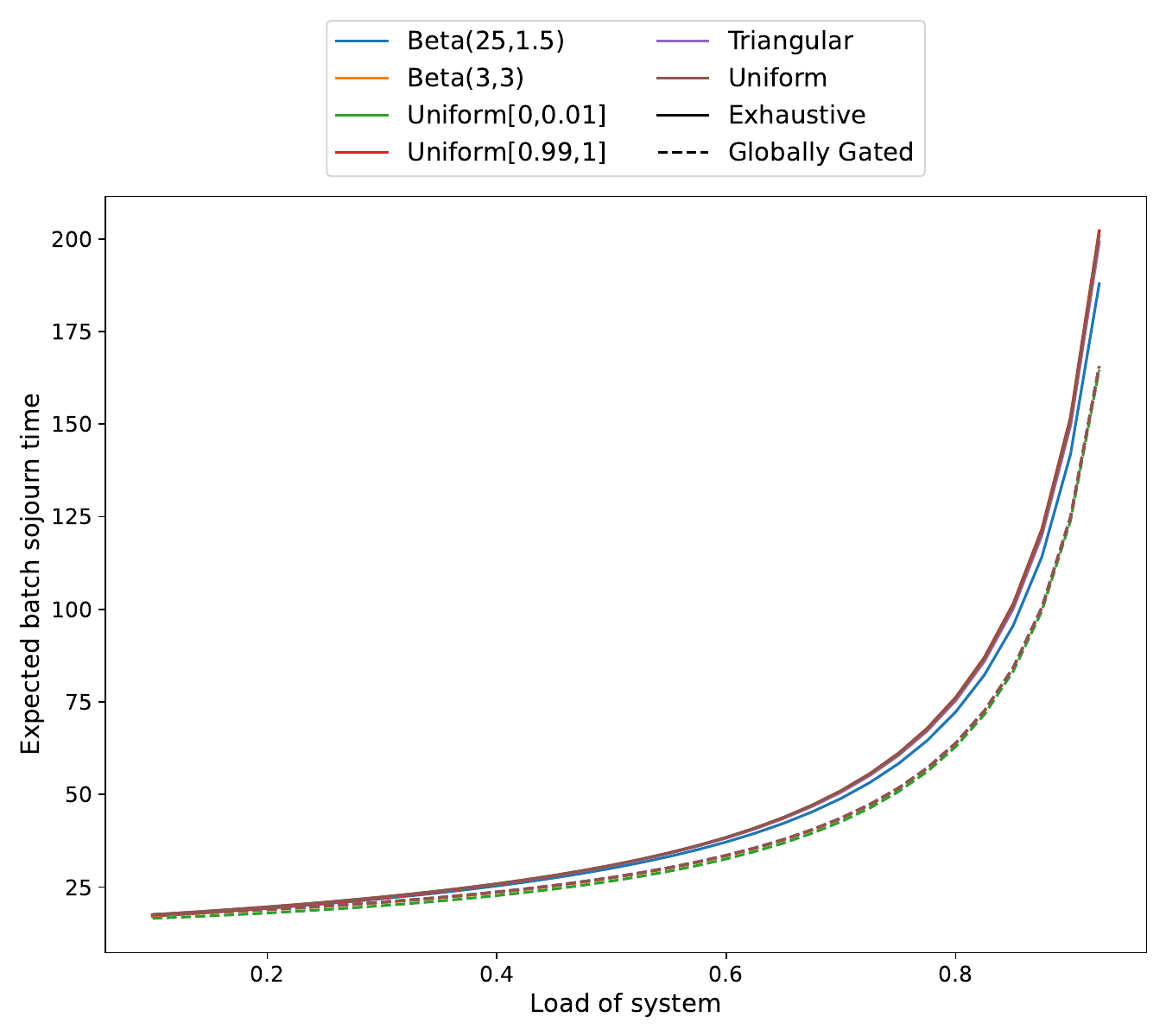}};
        
        \node[anchor=north west] (zoom) at (-1,0) {\includegraphics[width=0.48\textwidth]{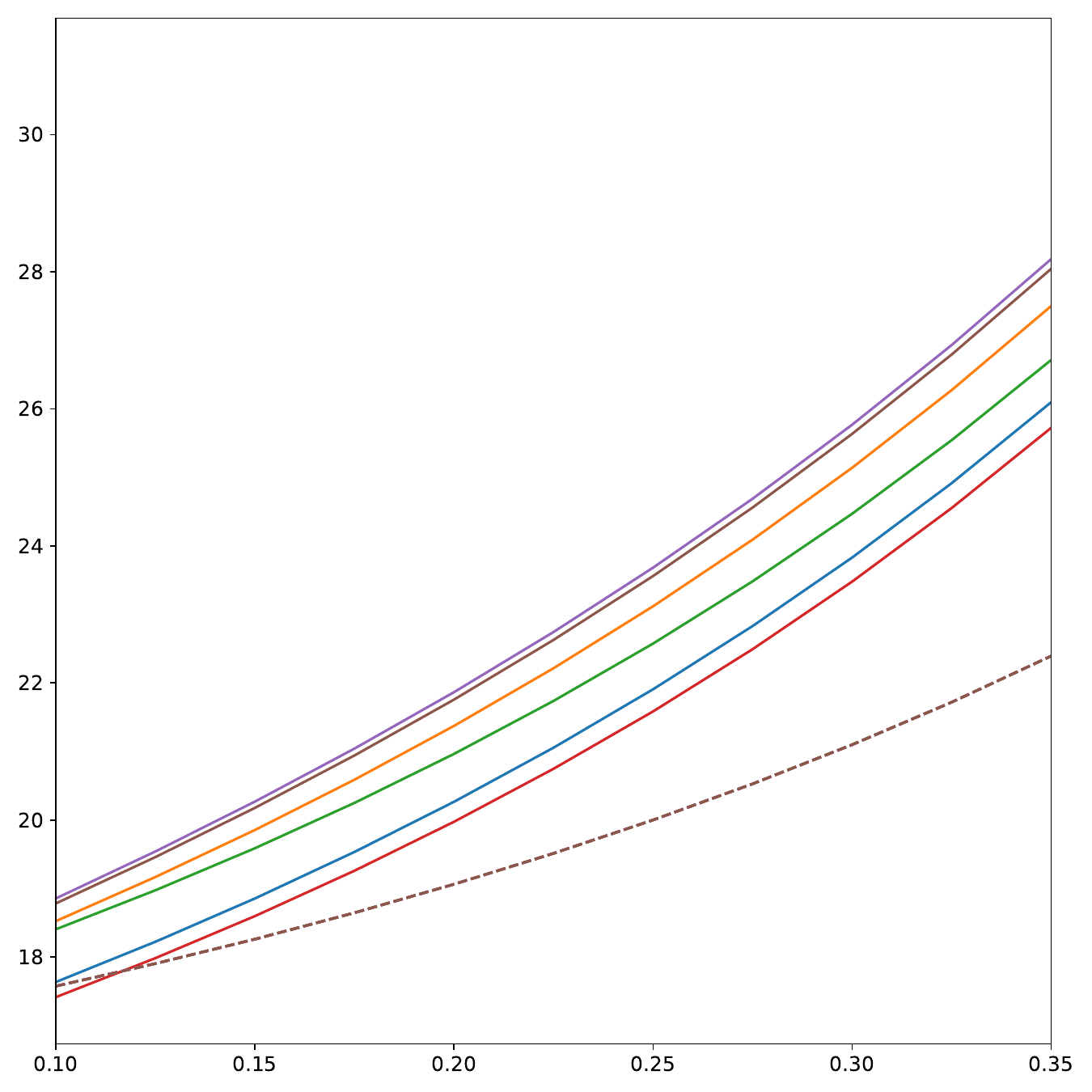}};
        
        \draw[red,thick] (1.5,1.2) rectangle (4.6,2.5);
        
        \draw[red,thick,->] (3.05,1.2) -- (3.05, 0);

        \node[anchor=north west] (zoom) at (6.5,0) {\includegraphics[width=0.48\textwidth]{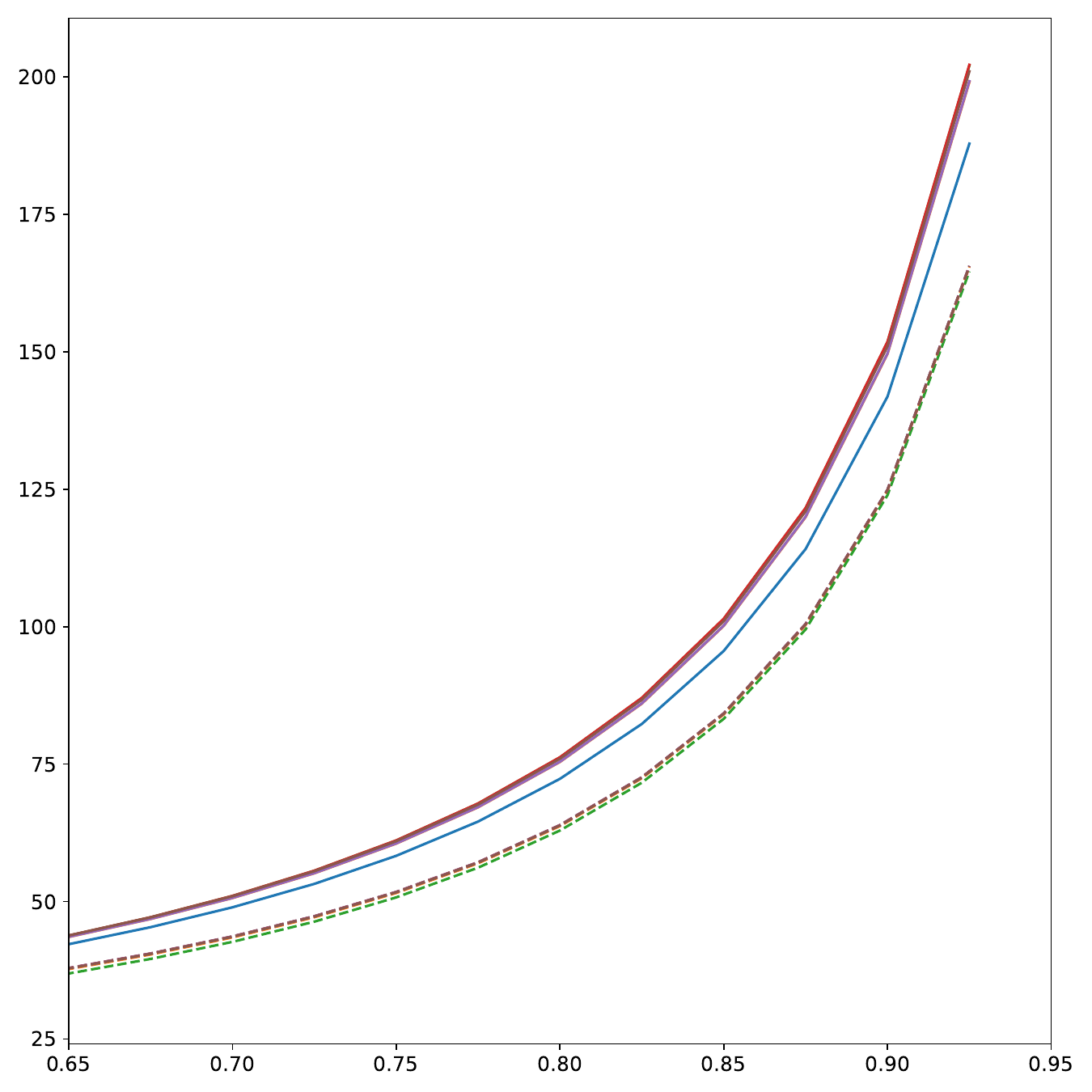}};
        
        \draw[red,thick] (8.2,2) rectangle (12,8.7);
        
        \draw[red,thick,->] (10.1,2) -- (10.1, 0);
        
    \end{tikzpicture}
    \caption{The expected batch sojourn time of the continuous polling model, under $B\equiv 1$, for varying arrival location distributions and loads. The results under exhaustive service are presented as a solid line, the dashed line refers to the globally gated policy.}
    \label{fig:Extremecases2_SB}
\end{figure}
\begin{figure}[H]
    \centering
    \begin{tikzpicture}
        \node[anchor=south west] (main) at (0,0) {\includegraphics[width=0.8\textwidth]{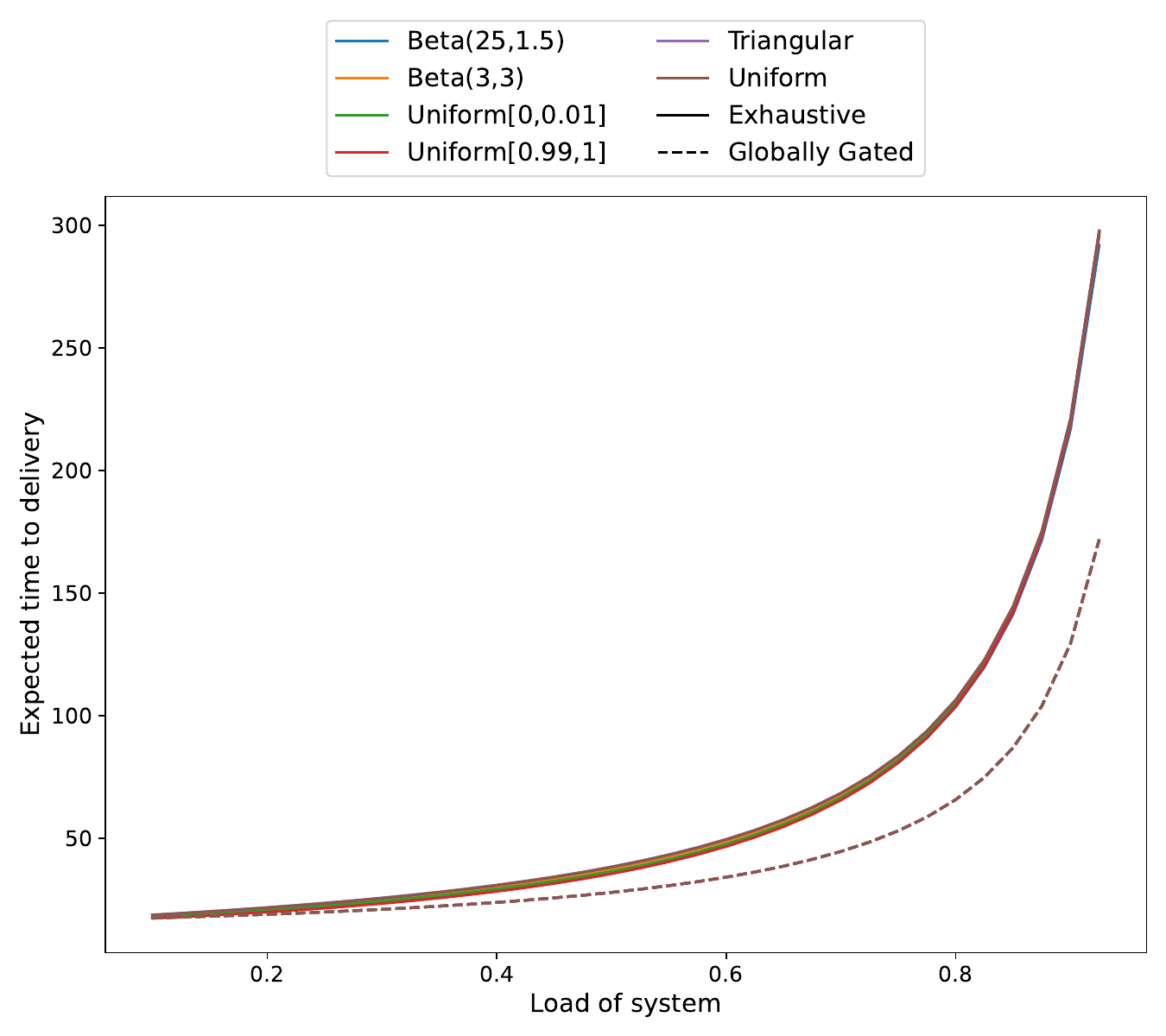}};
        
        \node[anchor=north west] (zoom) at (-1,0) {\includegraphics[width=0.48\textwidth]{Figures/extremecases_Bmed_ED_zoomed1.pdf}};
        
        \draw[red,thick] (1.5,1.2) rectangle (4.6,2.5);
        
        \draw[red,thick,->] (3.05,1.2) -- (3.05, 0);
        
        \node[anchor=north west] (zoom) at (6.5,0) {\includegraphics[width=0.48\textwidth]{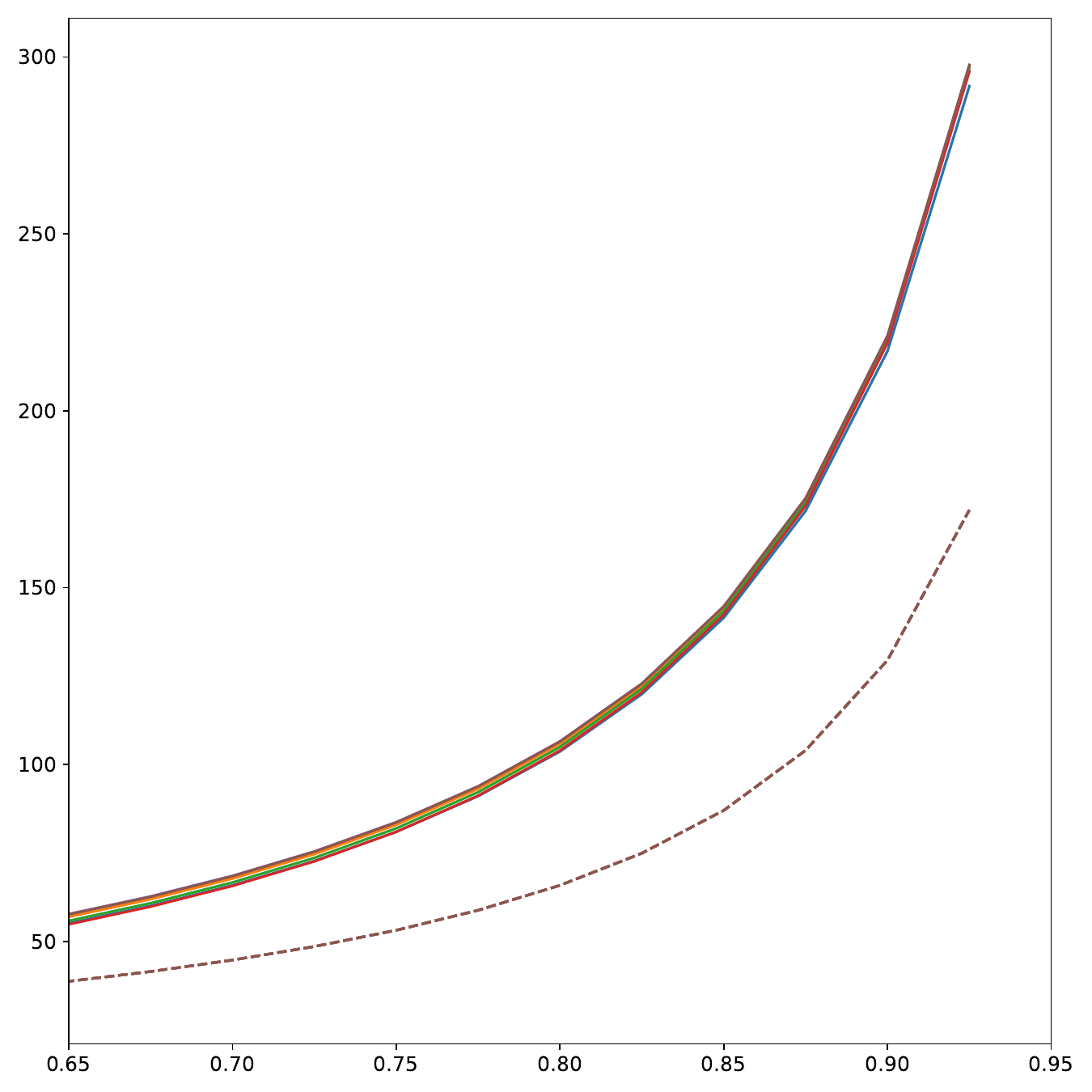}};
        
        \draw[red,thick] (8.2,1.7) rectangle (12,8.7);
        
        \draw[red,thick,->] (10.1,1.7) -- (10.1, 0);
    \end{tikzpicture}
    \caption{The expected time to delivery of the continuous polling model, under $B\equiv 1$, for varying arrival location distributions and loads. The results under exhaustive service are presented as a solid line, the dashed line refers to the globally gated policy.}
    \label{fig:Extremecases2_ED}
\end{figure}

\subsection*{Batch size investigation}
In the previous examples, we have mainly focused on the effect of the arrival location distribution. We now, additionally, consider the interactive effect of this distribution with the batch size.
Consider the continuous polling model with $\alpha = 1$ and assume that $K\equiv k, B\equiv 1/k$, consider the arrival location distributions: Uniform[0,1], Beta(3,3) and Uniform[0.99,1]. Figures \ref{fig:ggplot} and \ref{fig:exhplot} show the average performance of this model under the globally gated and exhaustive policy respectively, for both $\rho = 0.6$ and $\rho = 0.9$.\\

Note that under the globally gated policy, the expected sojourn time is increasing in the batch size. This is due to the fact that, as the batch size increases, also the average location of the furthest customer in a batch grows. This also explains the differences between the arrival location distributions. For small $k$, the Uniform[0,1] and Beta(3,3) cases result in similar results, as the average distances to the furthest customer are not much different when $k$ is small. However, as Beta(3,3) is more centred around $0.5$, the average distance to the furthest customer grows less quickly when $k$ grows, compared to the Uniform[0,1] case.\\
The exhaustive service case reveals that the performance effect of the arrival location distribution is larger for big batch sizes than for small batch sizes. Particularly, the effect in the expected time to delivery is quite large, although this effect is much smaller in the system with a higher load. Remark \ref{rem:heavy_comp} discussed that the system, under heavy traffic, is indifferent to the choice of $\pi$. This is immediately visible for $\E[S^B]$ in the $\rho = 0.9$ case. For $\E[D]$, we see that $\rho = 0.9$ still gives rise to some differences between the different arrival location distributions, which shows that the rate of convergence is smaller for this performance statistic.

\begin{figure}[H]
    \centering
    \includegraphics[width = \textwidth]{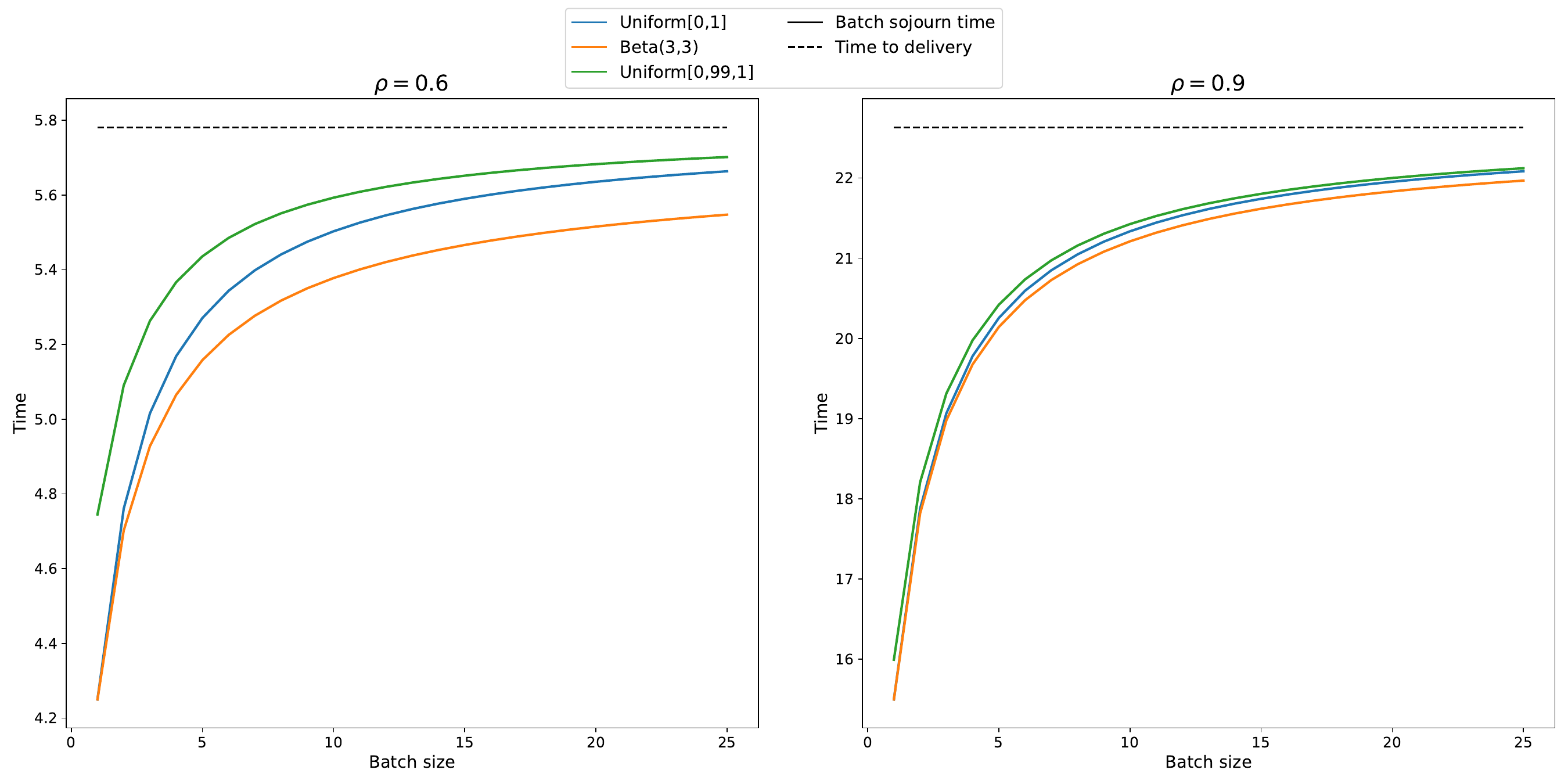}
    \caption{The expected batch sojourn time and time to delivery of the continuous polling model under the globally gated service discipline. The average batch sojourn time is given by the solid line, the single dashed line refers to the expected time to delivery.}
    \label{fig:ggplot}
\end{figure}

\begin{figure}[!htbp]
    \centering
    \includegraphics[width = \textwidth]{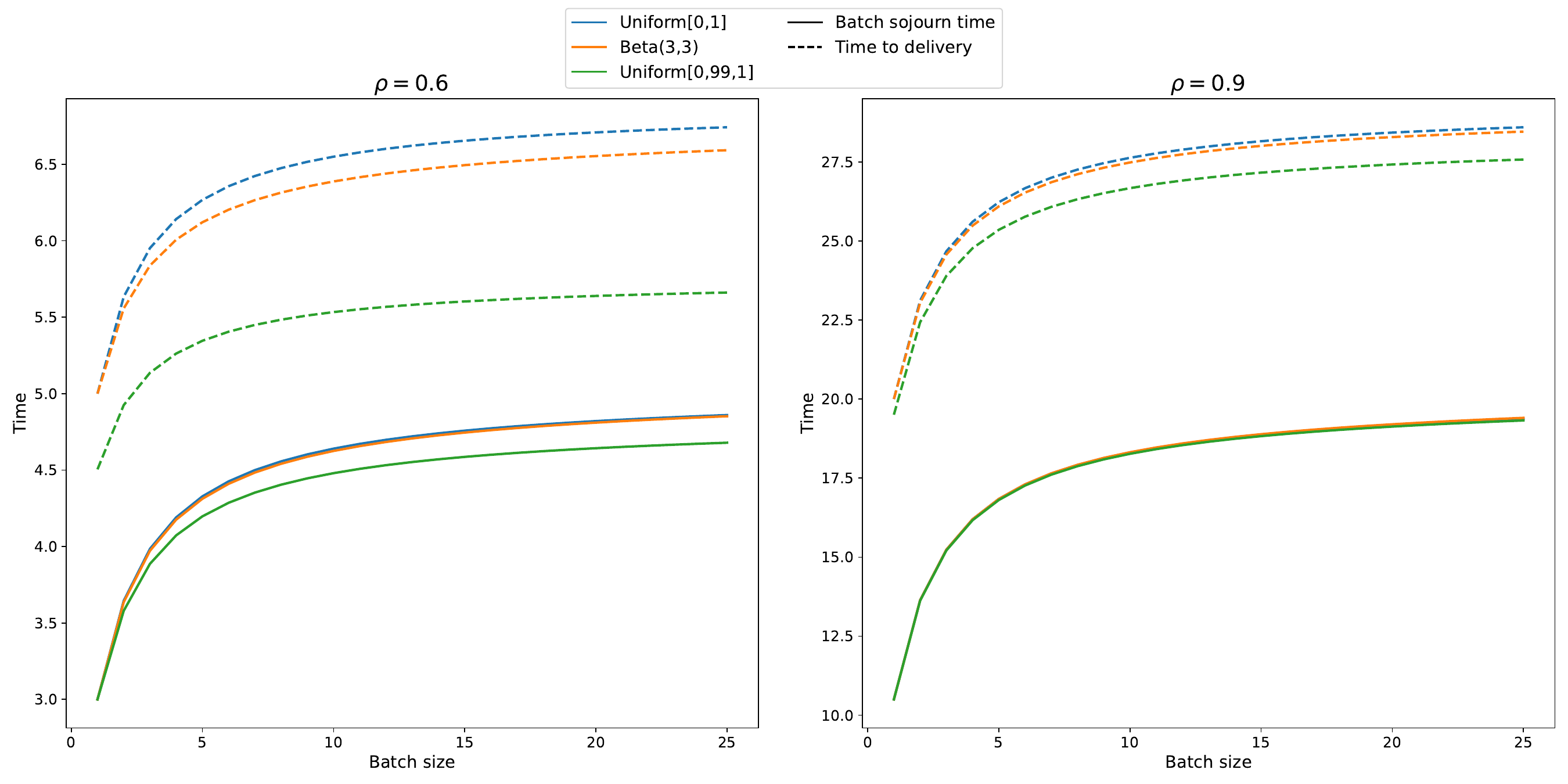}
    \caption{The expected batch sojourn time of the continuous polling model under the exhaustive service discipline. The average batch sojourn time is given by the solid line, the dashed lines refer to the expected time to delivery.}
    \label{fig:exhplot}
\end{figure}

\subsection{Warehouse application}
We now turn to the warehouse logistic application: milkrun systems. In this system, a picker continuously walks through the entire warehouse picking certain (requested) items that she encounters. Under the globally gated picking policy, the picker receives a new pick list each time she passes the depot. This pick list contains all requested items that have not yet been picked at that time. Under the exhaustive policy, the picker instead picks any requested item she encounters.\\

The polling representation of this system is as follows: the picker is represented by the server, while the requested items are represented by customers. An order can thus be seen as a batch of customers, immediately highlighting the importance of both the batch sojourn time and time to delivery, as these primarily dictate the performance of a warehouse. The arrival location distribution is determined by the demand of the related products. This particular property is often influenced by the storage policy of the warehouse. Several well-known policies exist. In most cases, storage policies dictate that fast movers (high demand items) need to be located at easier to reach locations, see for instance the detailed analysis in \citet{Dijkstra2017}. However, thus far, the research on the effect of storage policies on the performance of mikrun systems has been limited.

Some of the best known storage policies include random storage and class-based storage. Under the former, the items are randomly located throughout the warehouse, while under class-based storage, items are grouped into different classes (often three classes). Each class is assigned a section of the warehouse in which the corresponding items are stored at random locations (i.e. uniformly at random). Previous research has shown that class-based storage policies can improve the performance of the warehouse, while limiting the practical overhead (e.g. space requirement) involved in applying this policy \citep{Hausman1976, Jarvis1991}.

Consider a warehouse where orders arrive according to a Poisson process with varying arrival rate $\lambda$ per hour. The size of the orders follows a Poisson distribution with an average size of $15$ item requests. The picking time of a single item (service time) is assumed to be exponentially distributed with a mean of $5$ seconds. The picker takes a total time of $600$ seconds (10 minutes) to traverse the entire warehouse, excluding picking times. We further assume that the classes correspond to a demand distribution of $50/30/20$\% and space requirement distribution of $20/30/50$\%. This translates to a piecewise uniform arrival location distribution: viz. customers obtain a type (A,B or C) making up $50/30/20\%$ of all customers respectively. Each customer type in turn is, uniformly at random, located within $20/30/50$\% of the circle respectively. We assume that the order size follows a Poisson distribution with mean $15$ in case 1 and mean $3$ in case 2.\\
We compare the random storage policy, the class-based storage policy with fast movers at the beginning of the route, middle of the route and end of the route. In Figures \ref{fig:numresmain_sb1} - \ref{fig:numresmain_del2}, the average performance of the warehouses under the different storage policies and picking policies is plotted.

\begin{figure}[!htpb]
    \centering
    \resizebox{\textwidth}{!}{%
    \begin{tikzpicture}
        \node[anchor=south west] (main) at (0,0) {\includegraphics[width=0.8\textwidth]{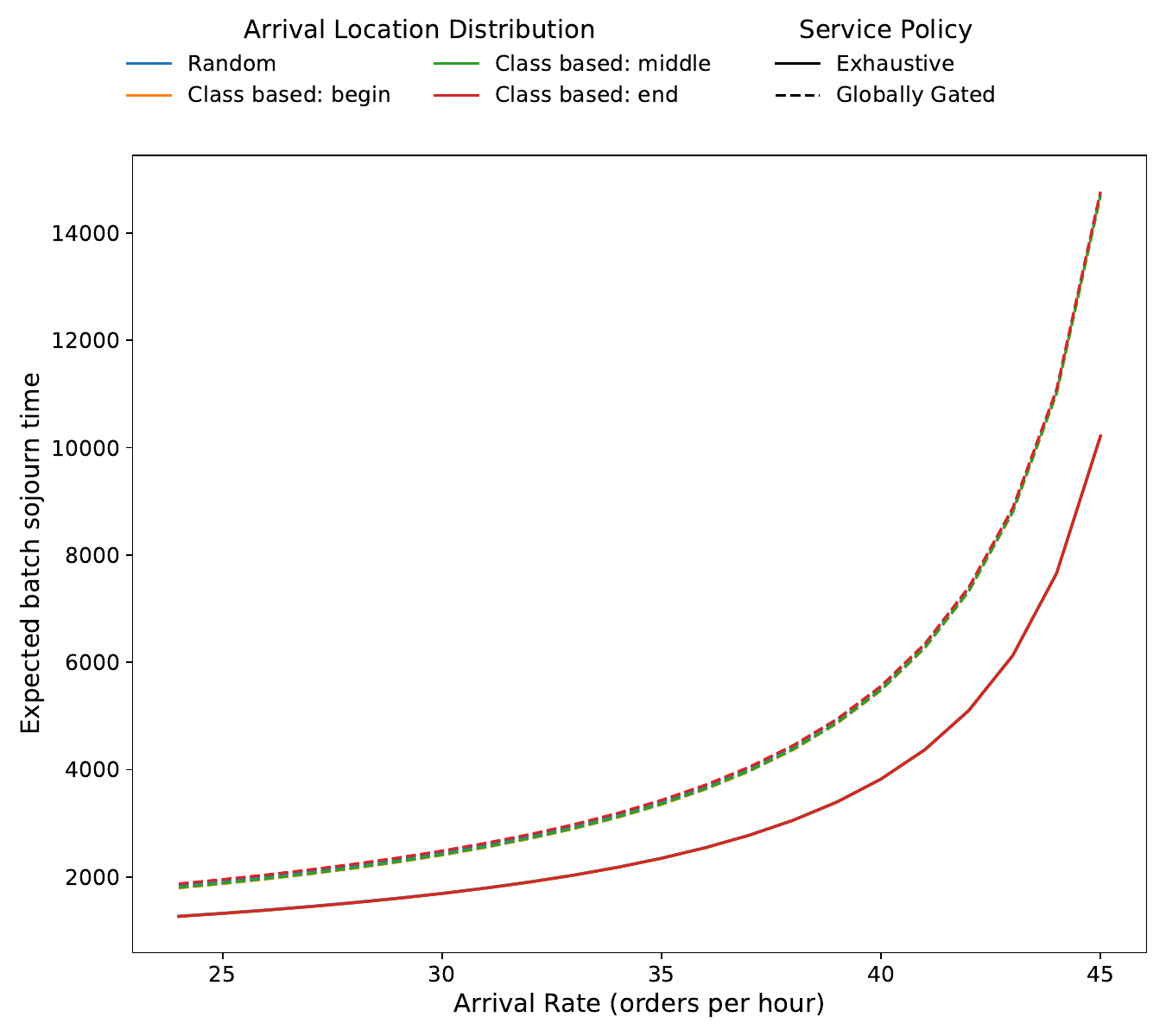}};
        
        \node[anchor=north west] (zoom) at (-1,0) {\includegraphics[width=0.48\textwidth]{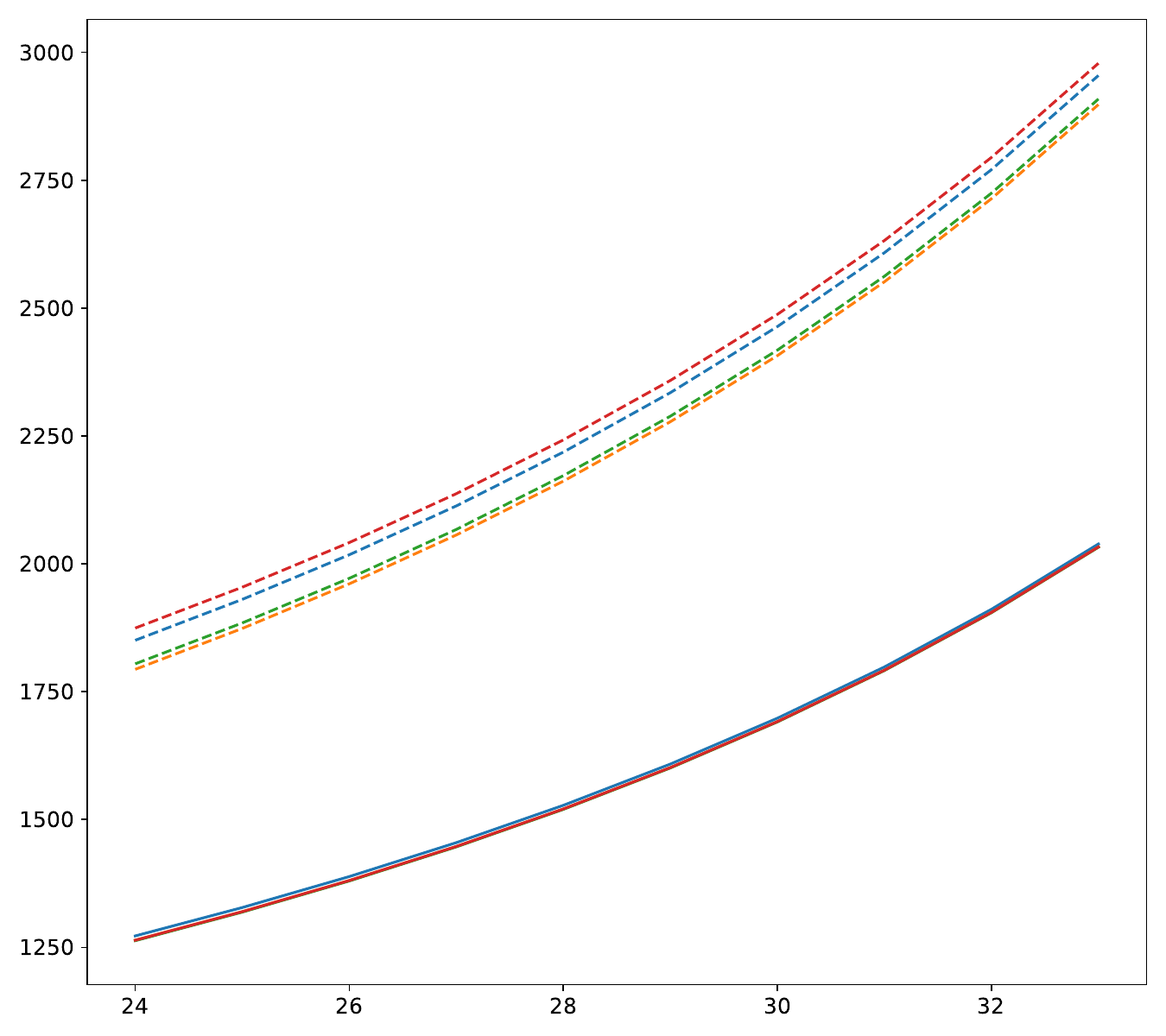}};
        
        \draw[red,thick] (1.7,1.2) rectangle (5.5,2.5);
        
        \draw[red,thick,->] (3.6,1.2) -- (3.6, 0);

        \node[anchor=north west] (zoom) at (6.5,0) {\includegraphics[width=0.48\textwidth]{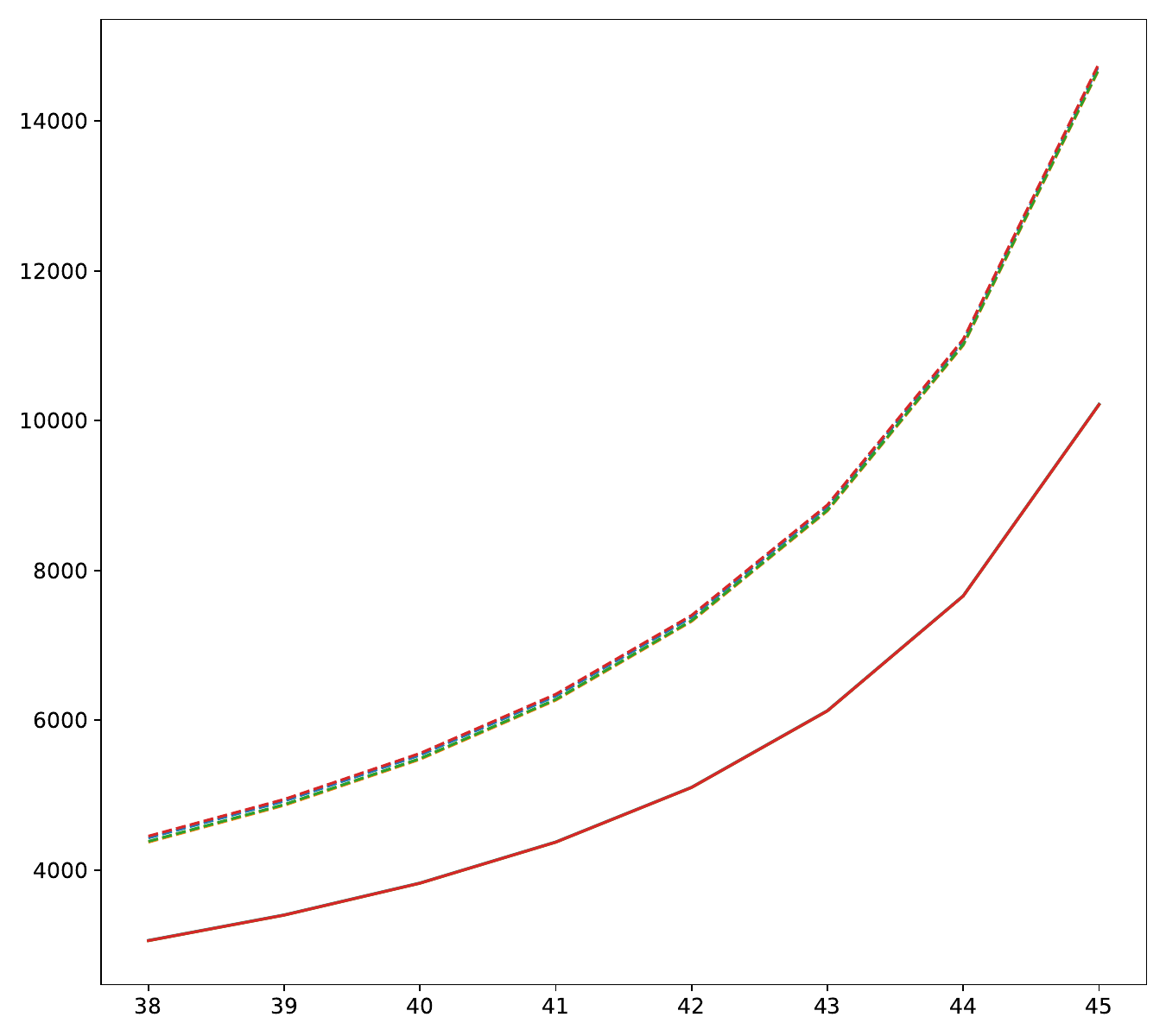}};
        
        \draw[red,thick] (8.2,1.7) rectangle (12,9.2);
        
        \draw[red,thick,->] (10.1,1.7) -- (10.1, 0);
        
    \end{tikzpicture}
    }
    \caption{The expected batch sojourn time in the warehouse model with shifted Poisson order sizes and an average of $15$. The results under exhaustive service are presented as a solid line, the dashed line refers to the globally gated policy.}
    \label{fig:numresmain_sb1}
\end{figure}
\begin{figure}[!htpb]
    \centering
    \resizebox{\textwidth}{!}{%
    \begin{tikzpicture}
        \node[anchor=south west] (main) at (0,0) {\includegraphics[width=0.8\textwidth]{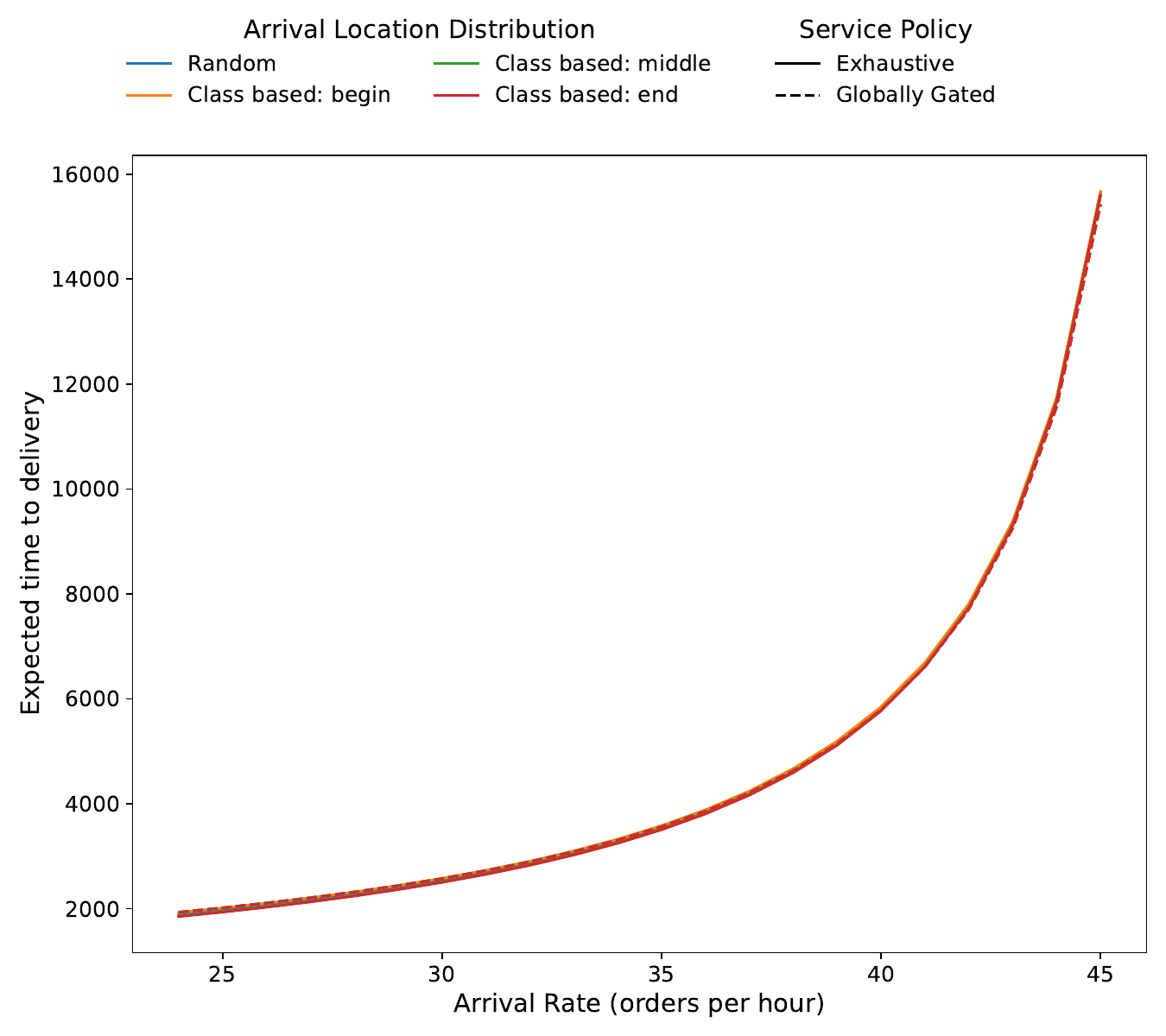}};
        
        \node[anchor=north west] (zoom) at (-1,0) {\includegraphics[width=0.48\textwidth]{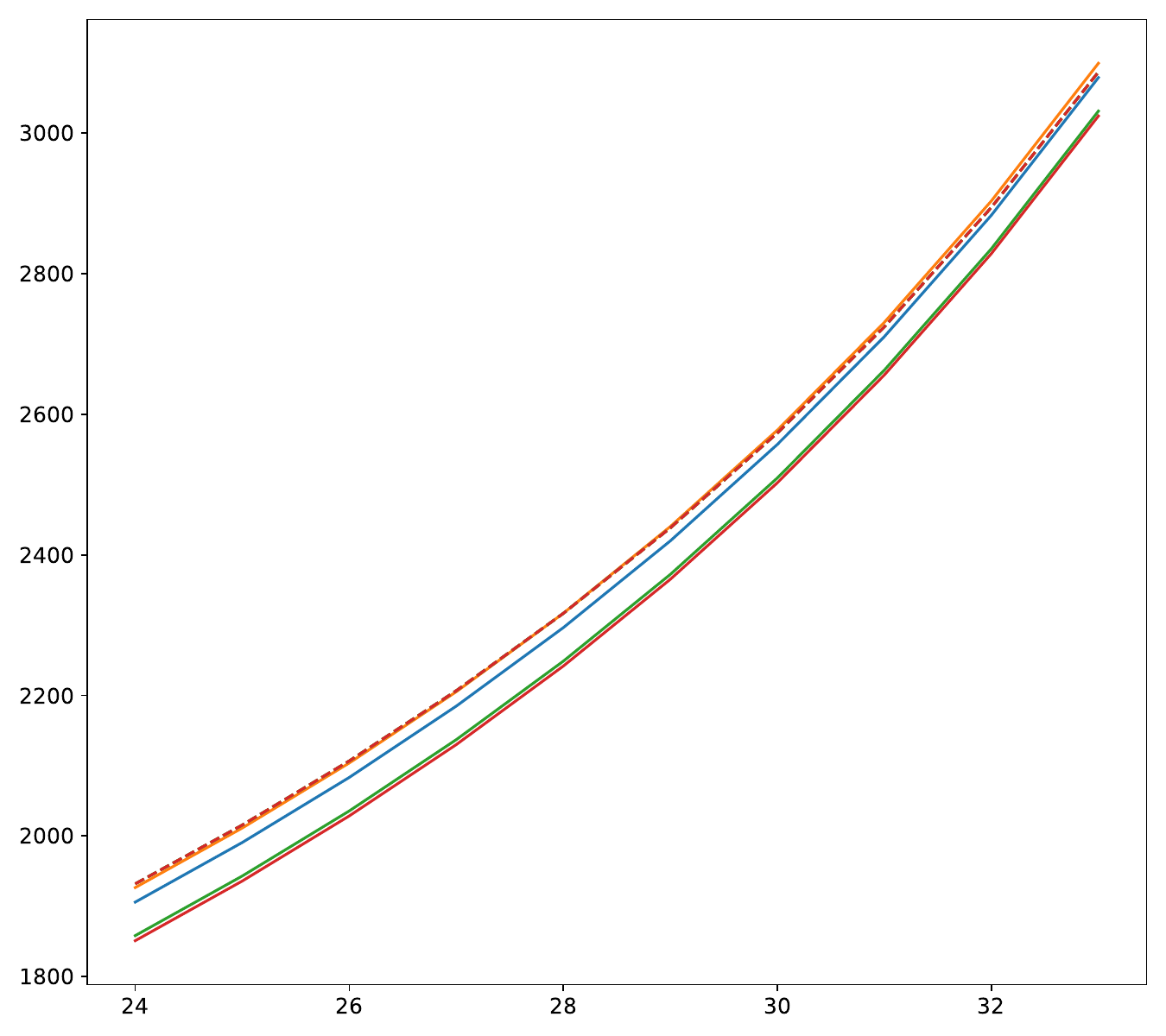}};
        
        \draw[red,thick] (1.7,1.2) rectangle (5.5,2.5);
        
        \draw[red,thick,->] (3.6,1.2) -- (3.6, 0);
        
        \node[anchor=north west] (zoom) at (6.5,0) {\includegraphics[width=0.48\textwidth]{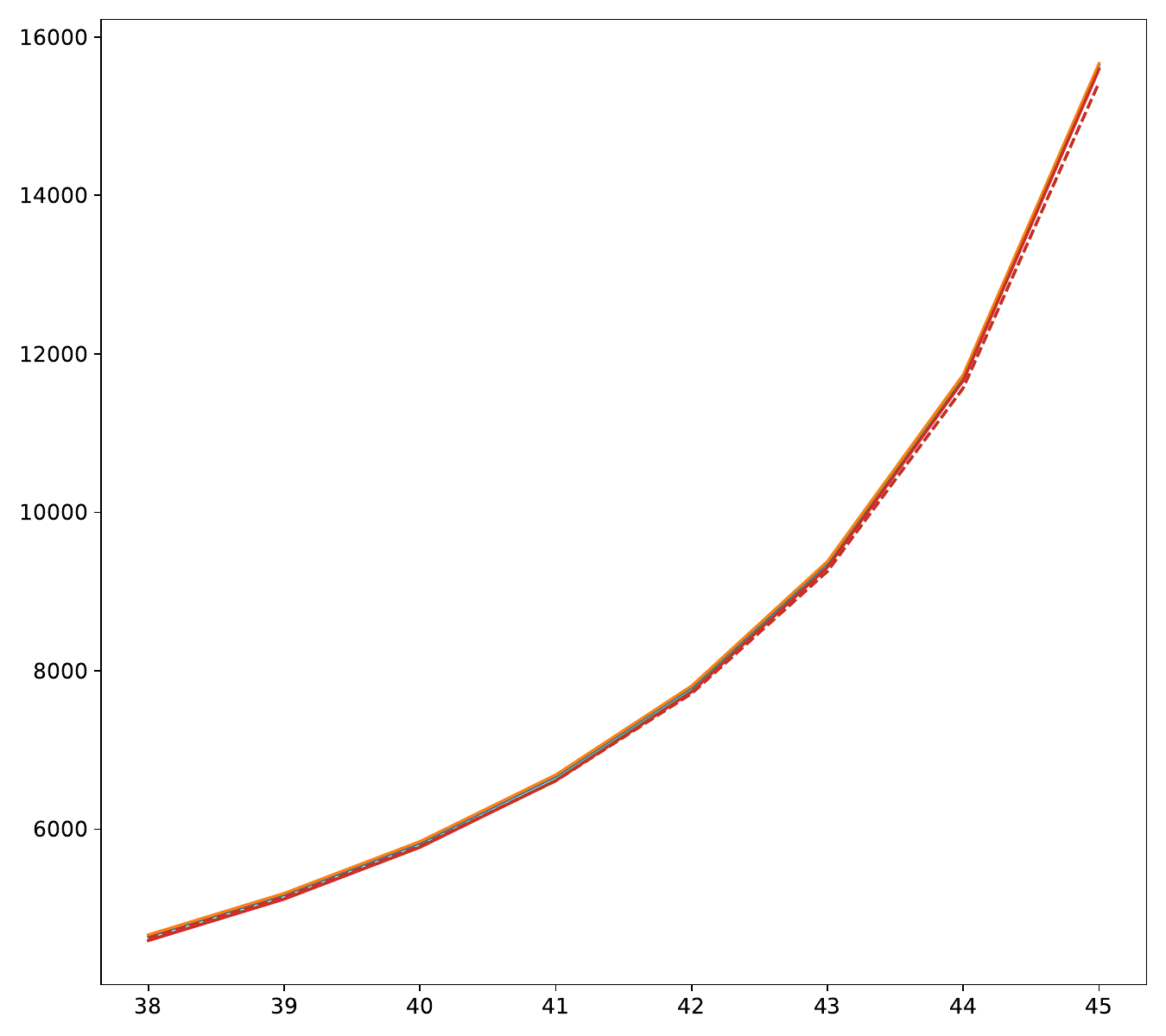}};
        
        \draw[red,thick] (8.2,1.7) rectangle (12,9.2);
        
        \draw[red,thick,->] (10.1,1.7) -- (10.1, 0);
        
    \end{tikzpicture}
    }
    \caption{The expected time to delivery in the warehouse model with shifted Poisson order sizes and an average of $15$. The results under exhaustive service are presented as a solid line, the dashed line refers to the globally gated policy.}
    \label{fig:numresmain_del1}
\end{figure}

\begin{figure}[!htpb]
    \centering
    \resizebox{\textwidth}{!}{%
    \begin{tikzpicture}
        \node[anchor=south west] (main) at (0,0) {\includegraphics[width=0.8\textwidth]{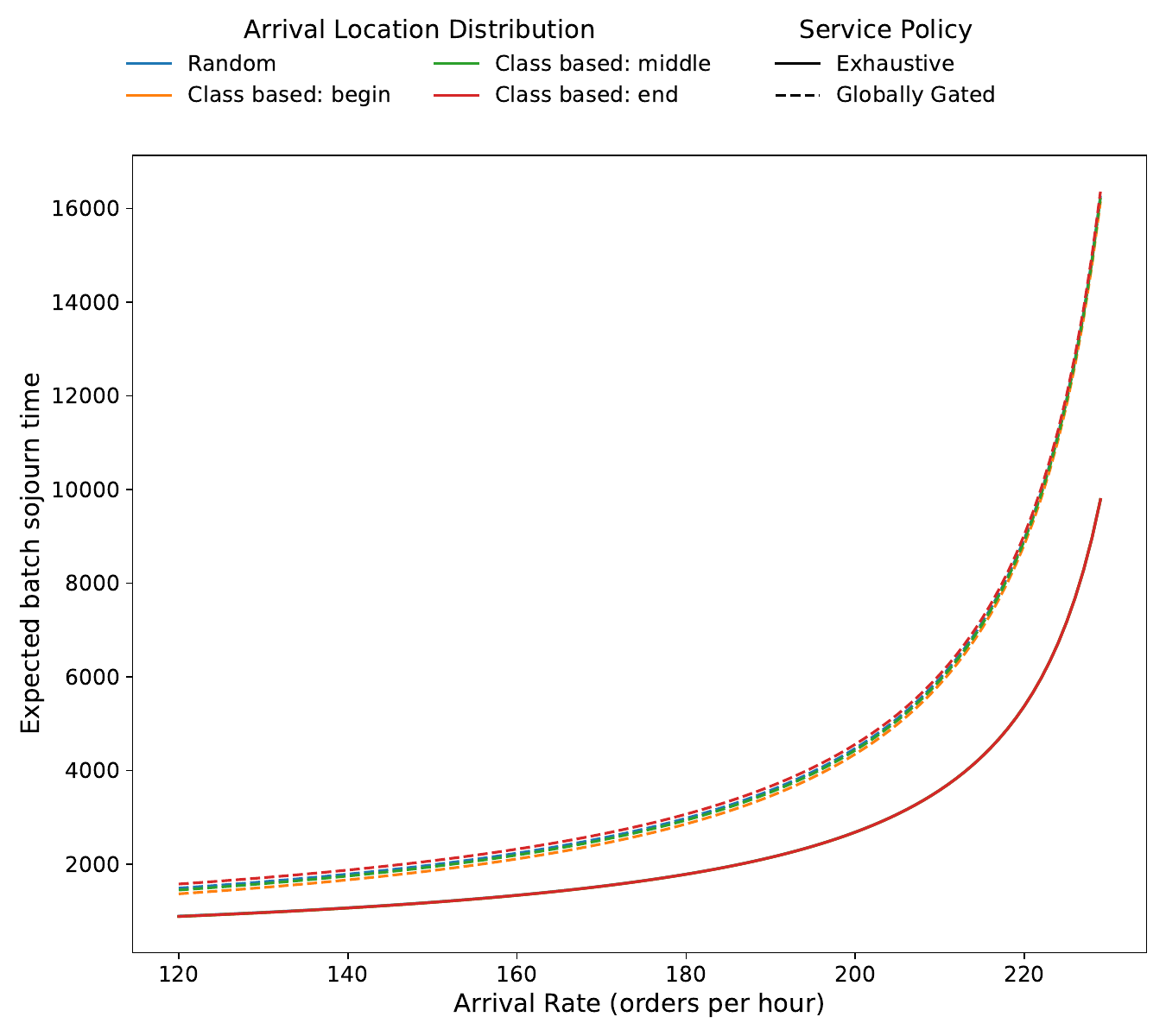}};
        
        \node[anchor=north west] (zoom) at (-1,0) {\includegraphics[width=0.48\textwidth]{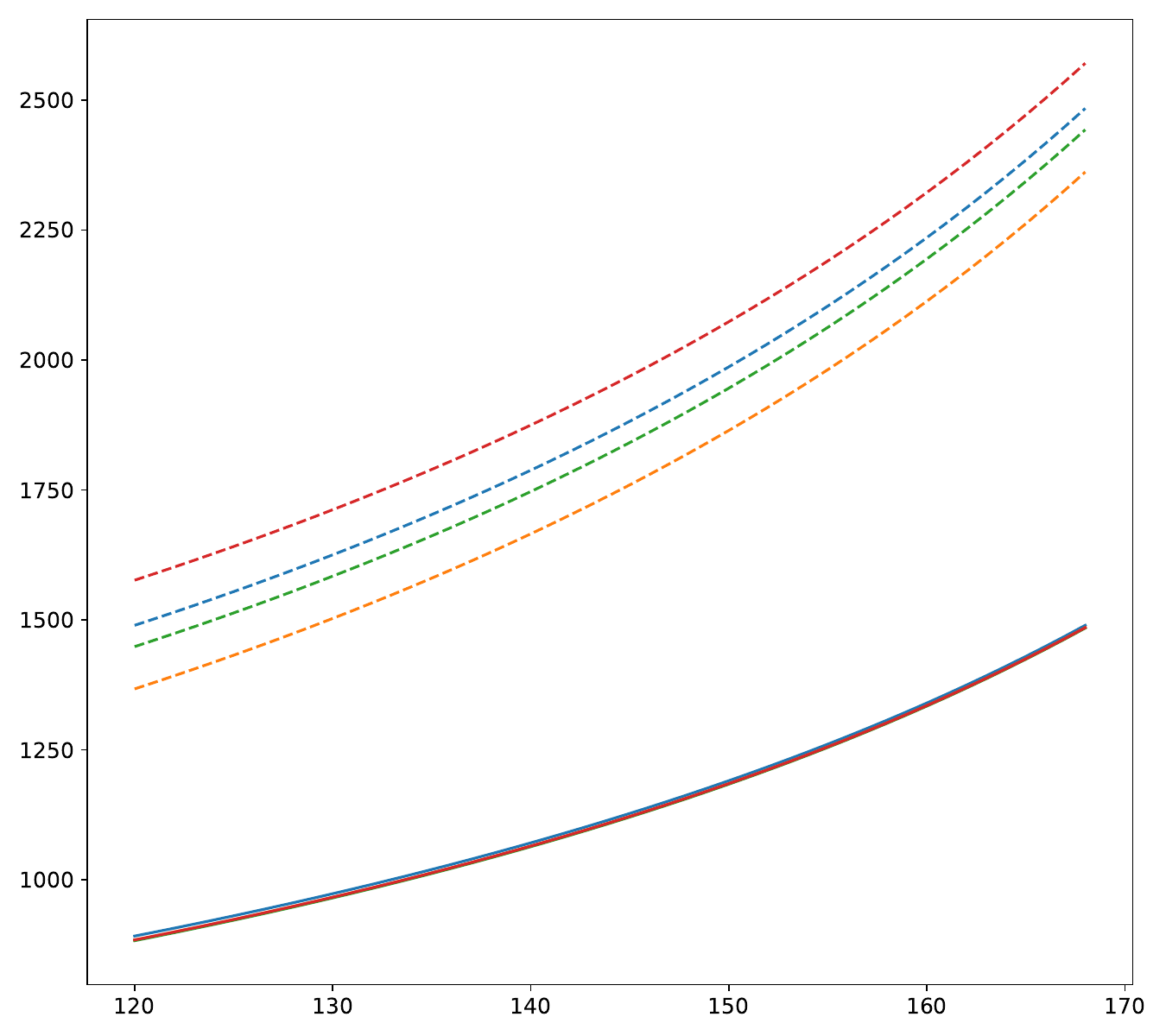}};
        
        \draw[red,thick] (1.7,1.2) rectangle (5.5,2.5);
        
        \draw[red,thick,->] (3.6,1.2) -- (3.6, 0);
        
        \node[anchor=north west] (zoom) at (6.5,0) {\includegraphics[width=0.48\textwidth]{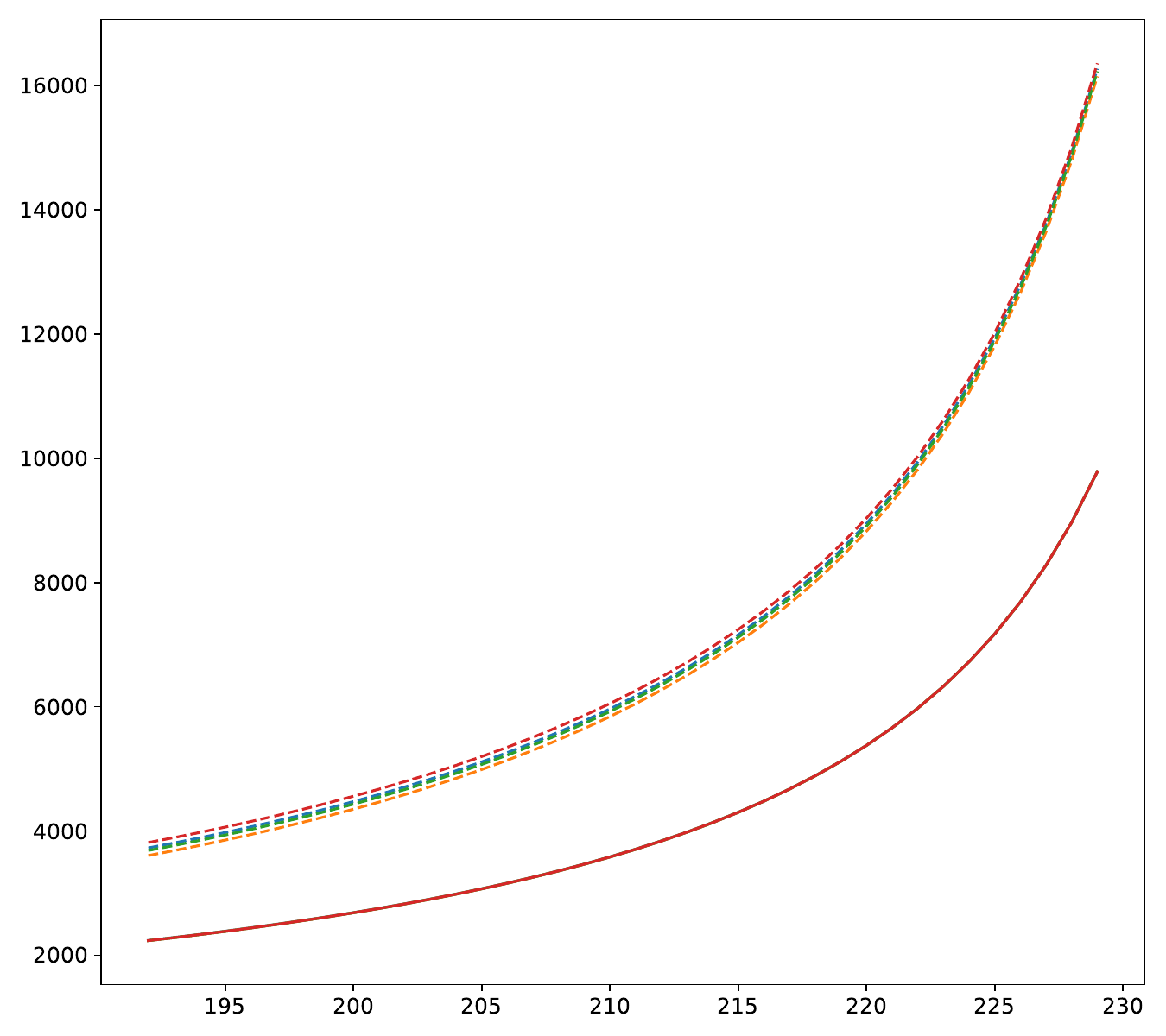}};
        
        \draw[red,thick] (8.2,1.7) rectangle (12,9.2);
        
        \draw[red,thick,->] (10.1,1.7) -- (10.1, 0);

    \end{tikzpicture}
    }
    \caption{The expected batch sojourn time in the warehouse model with shifted Poisson order sizes and an average of $3$. The results under exhaustive service are presented as a solid line, the dashed line refers to the globally gated policy.}
    \label{fig:numresmain_sb2}
\end{figure}
\begin{figure}[!htpb]
    \centering
    \resizebox{\textwidth}{!}{%
    \begin{tikzpicture}
        \node[anchor=south west] (main) at (0,0) {\includegraphics[width=0.8\textwidth]{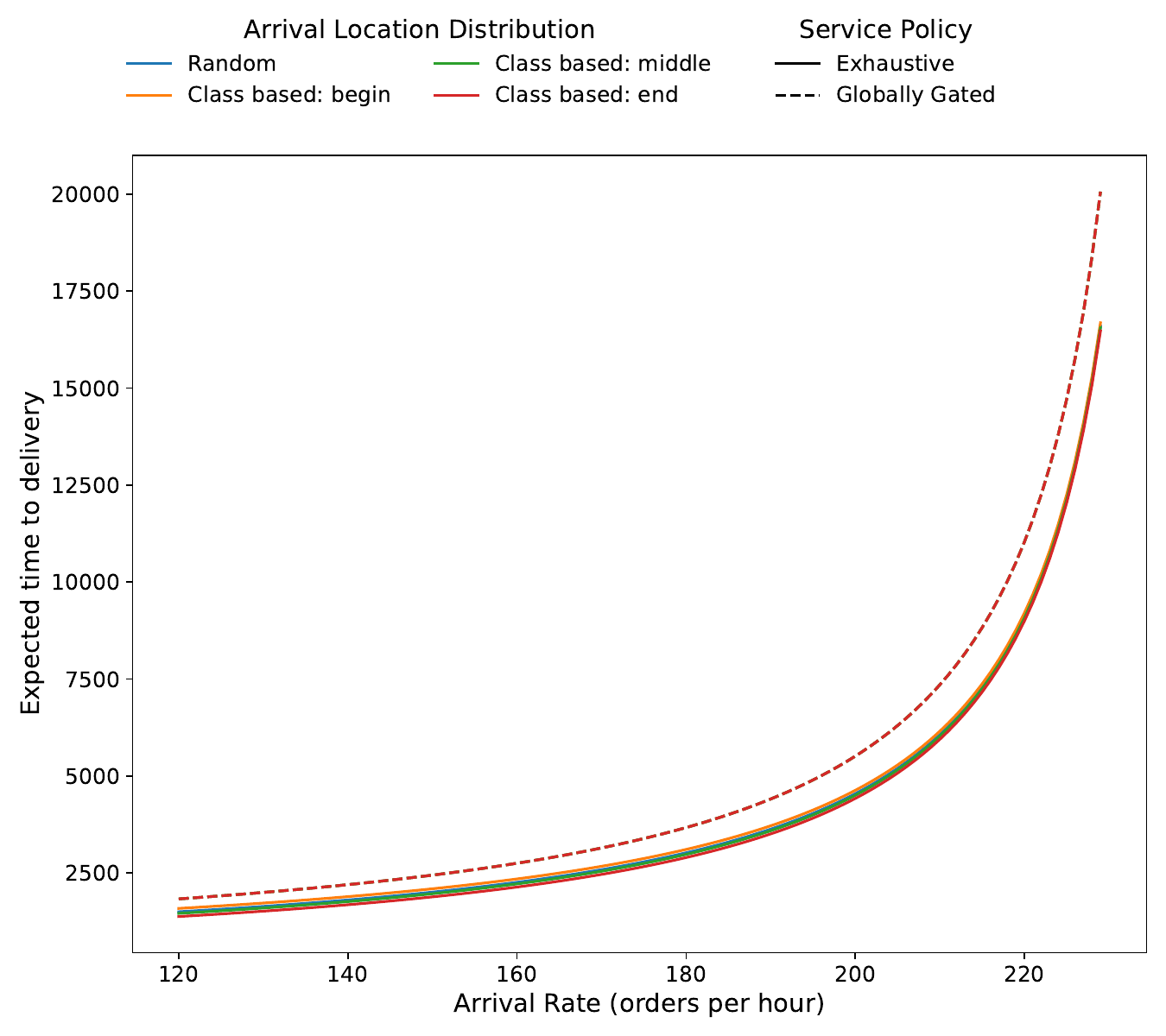}};
        
        \node[anchor=north west] (zoom) at (-1,0) {\includegraphics[width=0.48\textwidth]{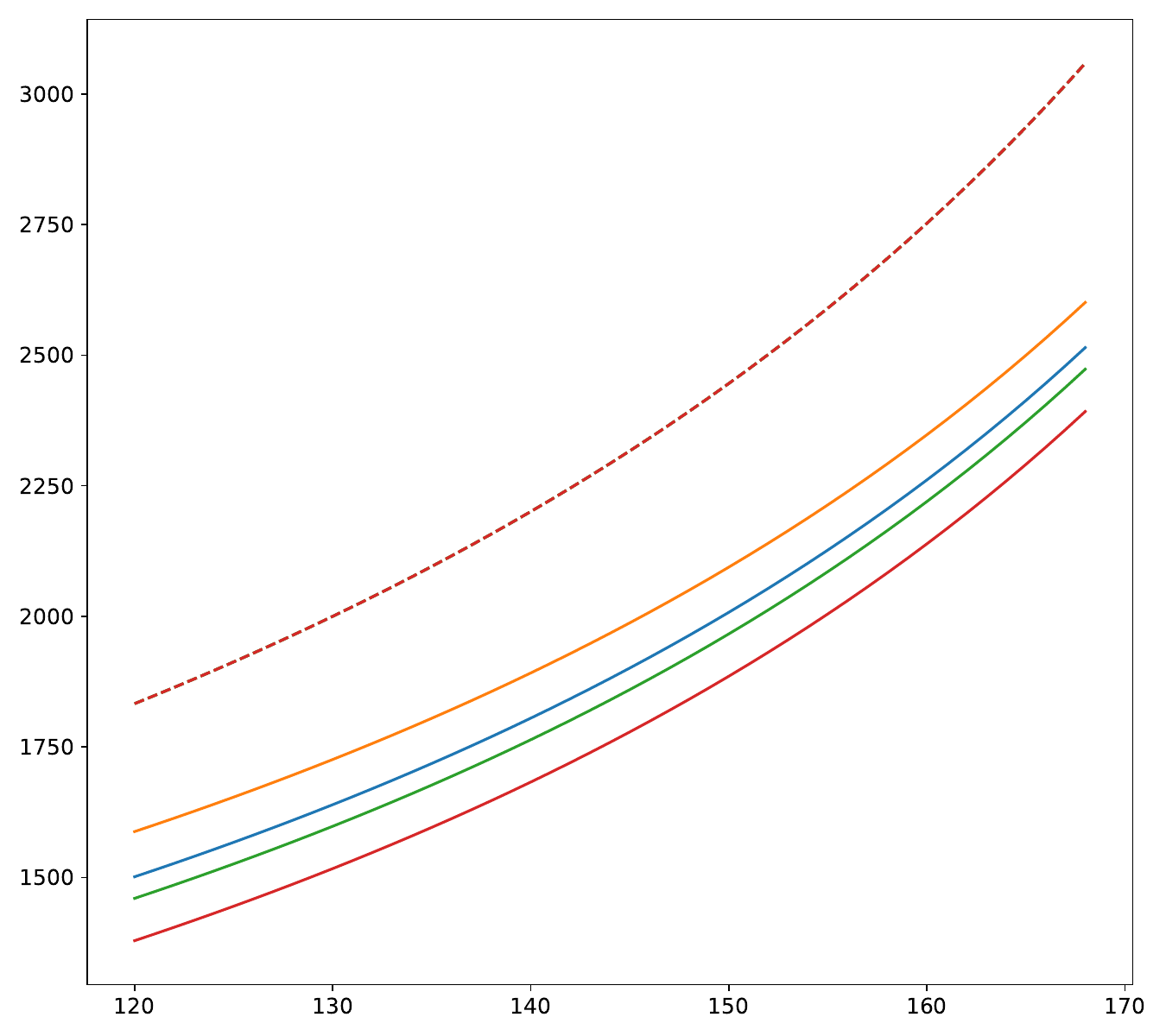}};
        
        \draw[red,thick] (1.7,1.2) rectangle (5.5,2.5);
        
        \draw[red,thick,->] (3.6,1.2) -- (3.6, 0);
        
        \node[anchor=north west] (zoom) at (6.5,0) {\includegraphics[width=0.48\textwidth]{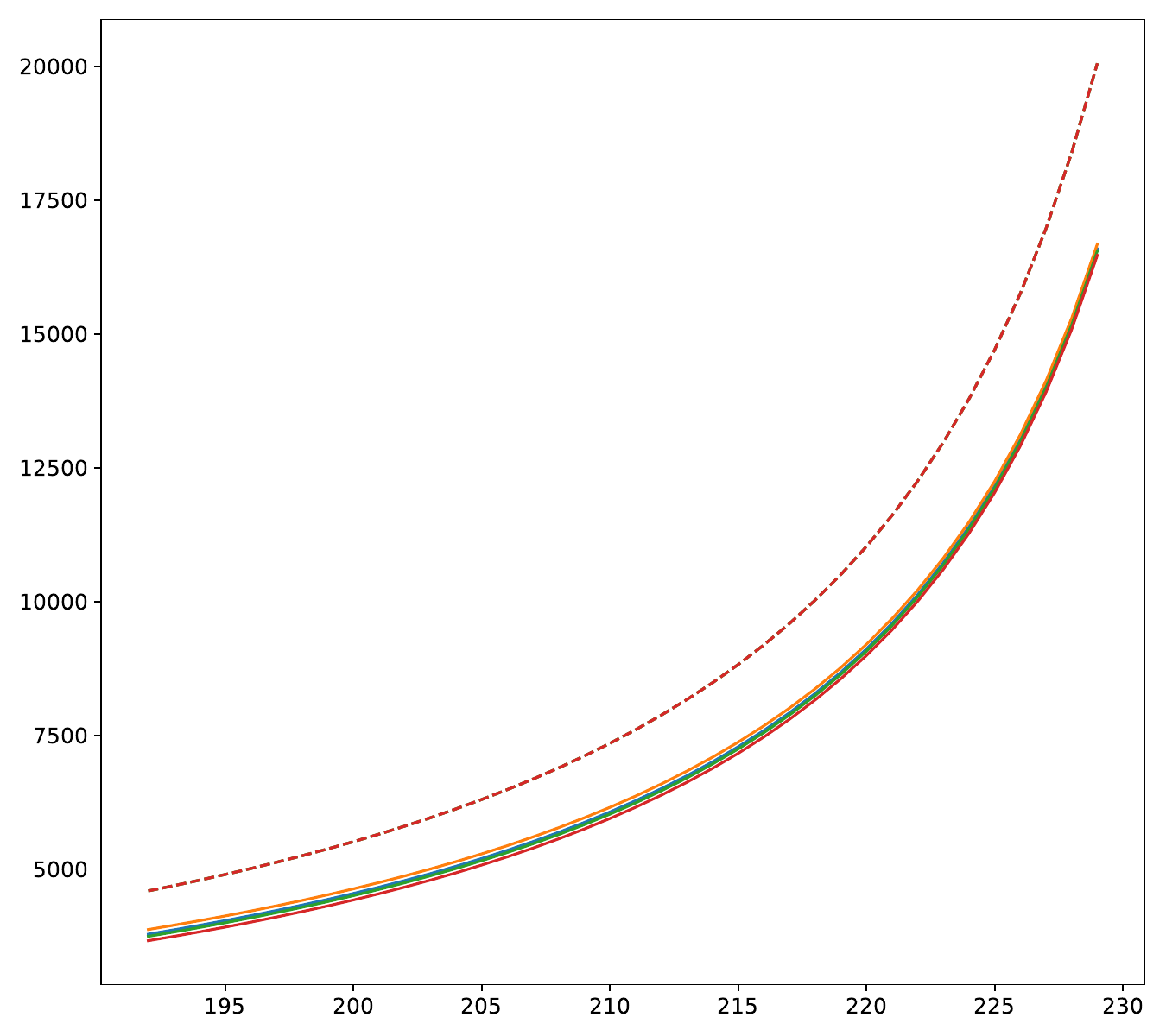}};
        
        \draw[red,thick] (8.2,1.7) rectangle (12,9.2);
        
        \draw[red,thick,->] (10.1,1.7) -- (10.1, 0);        
    \end{tikzpicture}
    }
    \caption{The expected time to delivery in the warehouse model with shifted Poisson order sizes and an average of $3$. The results under exhaustive service are presented as a solid line, the dashed line refers to the globally gated policy.}
    \label{fig:numresmain_del2}
\end{figure}

The results show a couple of striking features. In all figures, we see that the storage policy has almost no influence on the average performance of the warehouse. For the warehouse receiving larger orders, Figure \ref{fig:numresmain_sb1} shows that the differences between the storage policies are negligible. For small loads, the difference between policies is at most $5\%$ for globally gated and $<1\%$ for exhaustive service, decreasing to $<1\%$ and $<0.02\%$ respectively for the high loaded system. The exhaustive policy is clearly preferred over the globally gated policy.

The expected time to delivery in the warehouse with large orders (Poisson with mean 15), see Figure \ref{fig:numresmain_del1}, shows that there is a much smaller difference between the different picking and storage policies. For globally gated, we know that the storage policy has no effect, whereas there is an effect under the exhaustive policy, particularly when $\rho$ is small. There is a small preference for using class-based storage with fast movers at the end. Class-based storage with fast movers at the beginning performs the worst. This differences is $<5\%$ for small loads of the system and decreases to $<0.5\%$ for large loads. Remark here, that the exhaustive service policy is preferred over the globally gated policy for smaller $\rho$, while the opposite is true for larger loads.

The warehouse that receives smaller orders (Poisson with mean 3) shows similar results, although here the relative differences are larger for small arrival rates, see Figures \ref{fig:numresmain_sb2} and \ref{fig:numresmain_del2}. The mean batch sojourn time differs $<15.5\%$ for globally gated, and at most $1\%$ in the batch sojourn time. This decreases to $<1.5\%$ and $<0.01\%$ for larger loads of the system. On the other hand, the difference for exhaustive service between the storage policies is $<15.5\%$ for small loads in the average time to delivery, decreasing to $<1.5\%$ for large loads. For this warehouse, the exhaustive service policy is clearly preferred over the globally gated policy. \\

The effect of storage policies in milkrun systems is mostly dictated by the travel time of the server to the furthest customer in a batch (and then the depot). However, due to the relatively small differences in the storage policies, this effect is rather small. Only for warehouses with a low arrival rate (relative to service capacity) these differences amount to a substantial effect of the storage policy. 

There is a clear preference for the exhaustive picking policy in the examples. The globally gated policy results in too much walking time, due to the relatively large $\alpha$. For smaller $\alpha$, this preference might be reversed when looking at the average time to delivery.

\section{Conclusion}
In this paper, we have analysed continuous polling models with batch arrivals with a focus on the batch sojourn time and time to delivery.  Under the globally gated policy, the current paper finds exact results for the Laplace-Stieltjes transform of both the batch sojourn time and time to delivery, resulting in exact expressions for their mean. The system under exhaustive service is analysed using a mean-value analysis, resulting in an integral equation for the average spread of customers on the circle relative to the server's location. We find an exact solution to this integral equation when customers arrive alone, and provide an algorithm for the solution for the case with batch arrivals. In turn, exact expressions for the mean batch sojourn time and mean time to delivery are found, albeit still dependent on the output from the algorithm. Moreover, exact light- and heavy-traffic results are given for the average system performance under both the globally gated and the exhaustive service policy. The results in this paper are used to investigate the effect of several model parameters on the performance. Most notably, it is shown that the arrival location distribution of the customers does not have a large influence on the average performance. This is especially true for the case of milkrun order picking systems.\\

In this paper we made several assumptions, like the symmetry of service times, fixed walking times and independent arrival locations. Extending the results to cases with these assumptions loosened is of practical relevance. For example, in warehousing systems, the products locations in an order often exhibit dependence. Being able to control this dependence, by means of a storage policy, can prove to be useful for the performance of the warehouse. It seems possible to extend the approach in the present paper in that direction. However, this would require a very careful revisit of the model definitions and proofs. It would also be relevant and interesting to extend the analysis to multi-server systems and/or more advanced service policies.

\bibliography{main}

\appendix
\section{Proof of Theorem \ref{thm:batchsojourn}}
\label{app:proofs}
With the ingredients discussed in Section \ref{sec:EX}, we can immediately prove Theorem \ref{thm:batchsojourn}.
\begin{proof}[Proof of Theorem \ref{thm:batchsojourn}]
    We know that $f$ can be rewritten as:
    \begin{align*}
        f(x,y) &= f_K(x,y) + \underbrace{\frac{\lambda\E[K]\alpha}{1-\rho}\pi(x){\int_{u=x}^y}^*\big[\rho\pi(u) + 1-\rho \big]\dx u}_{f_\alpha(x,y)}\\
        &\quad + \underbrace{\frac{\lambda\E[K]\E[B^2]}{2\E[B]}\frac{\rho\pi(x)\pi(y)}{\rho \pi(y) + 1-\rho} + \frac{\lambda\E[K]\rho\E[B^2]}{(1-\rho)\E[B]}\frac{\rho\pi(x)\pi(y)}{\rho \pi(y) + 1-\rho}{\int_{u=x}^y}^*\pi(u)\dx u}_{f_{B^R}(x,y)}.
    \end{align*}
    Substitution of this expression in \eqref{eq:condbatchsojourn} and then deconditioning the expression using \eqref{eq:batchsojourn_conditioning} gives \eqref{eq:batchsojourn}. 
    We first substitute the expression for $f_\alpha$ into the right-hand side of the integral equation \eqref{eq:integraleq}, where we have:
    \begin{align*}
         {\int_{y=u}^x}^* f_{\alpha}(y,u)\E[S(y,x)]\dx y = \frac{\alpha}{1-\rho}{\int_{y=u}^x}^* \rho\pi(y)\exp\Big(\rho {\int_{\xi = y}^x}^*\pi(\xi) \dx \xi\Big){\int_{\nu = y}^u}^*\big[\rho \pi(\nu) + 1-\rho\big]\dx \nu\dx y.
    \end{align*}
    We integrate by parts to simplify this expression. When doing so, it is important to realise that:
    \begin{align*}
        \lim_{y\downarrow u} {\int_{\nu = y}^u}^* g(\nu)\dx \nu = {\int_{\nu = 0}^1} g(\nu)\dx \nu,
    \end{align*}
    for any function $g$. To see this, remark that as $y$ is slightly bigger than $u$, the extended integral definition $\int^*$ is almost equal to the integral over the entire circle. This explains the first term after the second equality sign below.
    \begin{align*}
        {\int_{y=u}^x}^* f_{\alpha}(y,u)\E[S(y,x)]\dx y &=  \frac{\alpha}{1-\rho}\bigg[-\exp\Big(\rho {\int_{\xi = y}^x}^*\pi(\xi) \dx \xi\Big){\int_{\nu = y}^u}^*\big[\rho \pi(\nu) + 1-\rho\big]\dx \nu\bigg]_{y = u}^x\\
        &\quad - \frac{\alpha}{1-\rho}{\int_{y=u}^x}^* \big[\rho \pi(y) + 1-\rho\big]\exp\Big(\rho {\int_{\xi = y}^x}^*\pi(\xi) \dx \xi\Big)\dx y\\
        &=\frac{\alpha}{1-\rho}\exp\Big(\rho {\int_{\xi = u}^x}^*\pi(\xi) \dx \xi\Big) - \frac{\alpha}{1-\rho}{\int_{\nu = x}^u}^*\big[\rho \pi(\nu) + 1-\rho\big]\dx \nu\\
        &\quad + \frac{\alpha}{1-\rho}\bigg[1 - \exp\Big(\rho {\int_{\xi = u}^x}^*\pi(\xi) \dx \xi\Big)\bigg] - \alpha{\int_{y=u}^x}^* \exp\Big(\rho {\int_{\xi = y}^x}^*\pi(\xi) \dx \xi\Big)\dx y.
    \end{align*}
    Remark that many terms cancel. Furthermore, recall the expression for $\E[T(u,x)]$ (cf. Lemma \ref{lemma:generatedwaiting}) to see:
    \begin{align}
        {\int_{y=u}^x}^* f_{\alpha}(y,u)\E[S(y,x)]\dx y&= \frac{\alpha}{1-\rho}\bigg(1-{\int_{y=x}^u}^*\big[\rho \pi(y) + 1-\rho\big]\dx y\bigg) - \E[T(u,x)] \nonumber\\
        \label{eq:intfalpha_sb}
        &=\frac{\alpha}{1-\rho}{\int_{y=u}^x}^*\big[\rho \pi(y) + 1-\rho\big]\dx y - \E[T(u,x)].
    \end{align}
    For the term with $f_{B^R}$ we obtain a similar result:
    \begin{align*}
         \big[\rho\pi(u) &+ 1-\rho\big]{\int_{y=u}^x}^* f_{B^R}(y,u)\E[S(y,x)]\dx y \\
         &= \frac{\rho\E[B^2]\pi(u)}{2\E[B]}{\int_{y=u}^x}^* \rho\pi(y)\exp\bigg(\rho{\int_{\xi = y}^x}^*\pi(\xi)\dx \xi\bigg)\dx y\\
         &\quad + \frac{\rho\E[B^2]\pi(u)}{\E[B](1-\rho)}{\int_{y=u}^x}^*\rho^2\pi(y)\exp\bigg(\rho{\int_{\xi = y}^x}^*\pi(\xi)\dx \xi\bigg){\int_{\nu = y}^u}^*\pi(\nu) \dx \nu\dx y.
    \end{align*}
    We evaluate the first integral and apply partial integration to the second integral. Here, we use the same limit of the integral as in the derivation for the previous term, which explains the second term in the right-hand side below. This shows that:
    \begin{align*}
        \big[\rho\pi(u) &+ 1-\rho\big]{\int_{y=u}^x}^* f_{B^R}(y,u)\E[S(y,x)]\dx y \\
         &= \frac{\rho\E[B^2]\pi(u)}{2\E[B]}\bigg[\exp\bigg(\rho{\int_{\xi = u}^x}^*\pi(\xi)\dx \xi\bigg) - 1\bigg]\\
         &\quad + \frac{\rho^2\E[B^2]\pi(u)}{\E[B](1-\rho)}\exp\bigg(\rho{\int_{\xi = u}^x}^*\pi(\xi)\dx \xi\bigg) - \frac{\rho^2\E[B^2]\pi(u)}{\E[B](1-\rho)}{\int_{\nu = x}^u}^*\pi(\nu) \dx \nu\\
         &\quad -  \frac{\rho\E[B^2]\pi(u)}{\E[B](1-\rho)}{\int_{y=u}^x}^*\rho\pi(y)\exp\bigg(\rho{\int_{\xi = y}^x}^*\pi(\xi)\dx \xi\bigg)\dx y.\\
     \intertext{We combine the exponential terms and evaluate the last integral, yielding}
         \big[\rho\pi(u) &+ 1-\rho\big]{\int_{y=u}^x}^* f_{B^R}(y,u)\E[S(y,x)]\dx y \\
         &=\frac{\rho(1+\rho)\E[B^2]\pi(u)}{2(1-\rho)\E[B]}\exp\bigg(\rho{\int_{\xi = u}^x}^*\pi(\xi)\dx \xi\bigg) - \frac{\rho\E[B^2]\pi(u)}{2\E[B]} - \frac{\rho^2\E[B^2]\pi(u)}{(1-\rho)\E[B]}{\int_{\nu = x}^u}^*\pi(\nu) \dx \nu\\
         &\quad + \frac{\rho\E[B^2]\pi(u)}{(1-\rho)\E[B]}\bigg[1-\exp\bigg(\rho{\int_{\xi = u}^x}^*\pi(\xi)\dx \xi\bigg)\bigg].
    \end{align*}
    We can isolate a factor $\E[S^R(u,x)] = \frac{\E[B^2]}{2\E[B]} \exp(\rho {\int_{\xi = u}^x}^* \pi(\xi) {\rm d}\xi)$ and use that ${\int_{\nu=x}^u}^*\pi(\nu)\dx \nu = 1-{\int_{\nu=u}^x}^*\pi(\nu)\dx \nu $. We then see that:
    \begin{align}
        \big[\rho\pi(u) &+ 1-\rho\big]{\int_{y=u}^x}^* f_{B^R}(y,u)\E[S(y,x)]\dx y \nonumber\\ 
        &= \frac{\rho(1+\rho)\E[B^2]\pi(u)}{2(1-\rho)\E[B]} - \frac{\rho^2\E[B^2]\pi(u)}{(1-\rho)\E[B]}{\int_{\nu = x}^u}^*\pi(\nu) \dx \nu - \rho\pi(u)\E[S^R(u,x)]
        \nonumber\\
        \label{eq:intfbr_sb}
        &=\frac{\rho\E[B^2]\pi(u)}{2\E[B]} + \frac{\rho^2\E[B^2]\pi(u)}{(1-\rho)\E[B]}{\int_{\nu = u}^x}^*\pi(\nu) \dx \nu - \rho\pi(u)\E[S^R(u,x)].
    \end{align}
    We substitute \eqref{eq:intfalpha_sb} and \eqref{eq:intfbr_sb} into \eqref{eq:condbatchsojourn}, which cancels the $\E[T(u,x)]$ and $\E[S^R(u,x)]$ terms:
    \begin{align*}
        \E\big[S^B\big\vert &X^B = x, K = k, S = u\big] \nonumber\\
        &=\E[B] + \frac{\alpha}{1-\rho}{\int_{y=u}^x}^*\big[\rho \pi(y) + 1-\rho\big]\dx y\\
        &\quad + \frac{\pi(u)}{\rho\pi(u)+1-\rho}\frac{\rho\E[B^2]}{2\E[B]} + \frac{\pi(u)}{\rho\pi(u)+1-\rho}\frac{\rho^2\E[B^2]}{(1-\rho)\E[B]}{\int_{y = u}^x}^*\pi(y) \dx y \\
        &\quad + \frac{k-1}{{\int_{y=u}^x}^*\pi(y)\dx y}{\int_{y=u}^x}^*\pi(y)\E[S(y,x)]\dx y + {\int_{y=u}^x}^* f_K(y,u)\E[S(y,x)]\dx y.
    \end{align*}
    We now decondition, as in \eqref{eq:batchsojourn_conditioning}, to find $\E[S^B]$. In this, we use the definition of the probability generating function $\tilde{K}(\cdot)$ of $K$: $\tilde{K}(x) = \sum_{k=1}^\infty p_k x^k$.
    \begin{align}
    \label{eq:thisformulaislong_buthasbraces}
         \E\big[S^B\big] &= \E[B]\\
         &\quad + \underbrace{
         \begin{aligned}[t]
         \frac{\alpha}{1-\rho}\int_{u=0}^1 \big[\rho \pi(u) + 1-\rho\big]\int_{x=0}^1 &\pi(x)\tilde{K}'\left({\int_{v = u}^x}^*\pi(v)\dx v\right)\\
         &\cdot{\int_{y=u}^x}^*\big[\rho \pi(y) + 1-\rho\big]\dx y \dx x \dx u
         \end{aligned}}_{(i)}\nonumber\\
         &\quad + \underbrace{\frac{\rho\E[B^2]}{2\E[B]}\int_{u=0}^1\pi(u)\int_{x=0}^1 \pi(x)\tilde{K}'\left({\int_{v = u}^x}^*\pi(v)\dx v\right)\dx x \dx u}_{(ii)} \nonumber\\
         &\quad + \underbrace{\frac{\rho^2\E[B^2]}{\E[B](1-\rho)}\int_{u=0}^1\pi(u)\int_{x=0}^1 \pi(x)\tilde{K}'\left({\int_{v = u}^x}^*\pi(v)\dx v\right){\int_{y=u}^x}^*\pi(y) \dx y \dx x \dx u}_{(ii)}\nonumber\\
          &\quad + \underbrace{
        \begin{aligned}[t]
          \int_{u=0}^1 \big[\rho \pi(u) + 1-\rho\big]\int_{x=0}^1 &\pi(x)\tilde{K}''\left({\int_{v = u}^x}^*\pi(v)\dx v\right)\\
          &\cdot{\int_{y=u}^x}^*\pi(y)\E[B]\exp\bigg(\rho{\int_{\xi = y}^x}^*\pi(\xi) \dx \xi\bigg)\dx y \dx x \dx u
         \end{aligned}}_{(iii)}\nonumber\\
         &\quad + \underbrace{
         \begin{aligned}[t]
         \int_{u=0}^1 \big[\rho \pi(u) + 1-\rho\big]\int_{x=0}^1 &\pi(x)\tilde{K}'\left({\int_{v = u}^x}^*\pi(v)\dx v\right)\\
         &{\int_{y=u}^x}^* f_K(y,u)\E[B]\exp\bigg(\rho{\int_{\xi = y}^x}^*\pi(\xi) \dx \xi\bigg)\dx y\dx x \dx u.
         \end{aligned}}_{(iv)}\nonumber
    \end{align}
    Before we start simplifying this expression, we remark that the integral over $x$ contains an important discontinuity to consider. Let $g$ be an arbitrary function, then we have that by definition of $\int^*$:
    \begin{align*}
        \int_{x=0}^1 g\left({\int_{v = u}^x}^*\pi(v)\dx v\right)\dx x &= \int_{x=0}^u g\left({\int_{v = u}^1}\pi(v)\dx v + {\int_{v = 0}^x}\pi(v)\dx v\right)\dx x +\int_{x=u}^1 g\left({\int_{v = u}^x}\pi(v)\dx v\right)\dx x.
    \end{align*}
    Note that we can also interpret this as the integral over the entire circle, starting at $u$. This eliminates the discontinuity of the integrand at $u$ and properly deals with this definition. Let $u^{-}$ and $u^{+}$ denote the left- and right limit to $u$ respectively, then we thus have:
    \begin{equation}
        \label{eq:integrationhelp}
        \int_{x=0}^1 g\left({\int_{v = u}^x}^*\pi(v)\dx v\right)\dx x = {\int_{x={u^+}}^{u^{-}}}^* g\left({\int_{v = u}^x}^*\pi(v)\dx v\right)\dx x.
    \end{equation}
    For the term (\ref{eq:thisformulaislong_buthasbraces}i) we immediately apply this, combined with partial integration:
    \begin{align*}
        (\ref{eq:thisformulaislong_buthasbraces}i) &= \frac{\alpha}{1-\rho}\int_{u=0}^1 \big[\rho \pi(u) + 1-\rho\big]\bigg[\tilde{K}\left({\int_{v = u}^x}^*\pi(v)\dx v\right){\int_{y=u}^x}^*\big[\rho\pi(y) + 1 - \rho\big] \dx y\bigg]_{x=u^{+}}^{u^{-}}\dx u\\
        &\quad -\frac{\alpha}{1-\rho}\int_{u=0}^1 \big[\rho \pi(u) + 1-\rho\big]\int_{x=u^{+}}^{u^{-}}\big[\rho\pi(x) + 1 - \rho\big]\tilde{K}\left({\int_{v = u}^x}^*\pi(v)\dx v\right)\dx x \dx u\\
        &= \frac{\alpha}{1-\rho} - \frac{\alpha}{1-\rho}\int_{u=0}^1 \big[\rho \pi(u) + 1-\rho\big]{\int_{x=u^{+}}^{u^{-}}}^*\big[\rho\pi(x) + 1 - \rho\big]\tilde{K}\left({\int_{v = u}^x}^*\pi(v)\dx v\right)\dx x \dx u.
    \end{align*}
    For term (\ref{eq:thisformulaislong_buthasbraces}ii) we first use the substitution $\omega = {\int_{y=u}^x}^*\pi(y)\dx y$ and then apply partial integration:
    \begin{align*}
        (\ref{eq:thisformulaislong_buthasbraces}ii) &= \frac{\rho\E[B^2]}{2\E[B]} + \frac{\rho^2\E[B^2]}{\E[B](1-\rho)}\int_{u=0}^1\pi(u)\int_{\omega=0}^1\omega\tilde{K}'\left(\omega\right) \dx \omega \dx u\\
        &=\frac{\rho\E[B^2]}{2\E[B]}  +  \frac{\rho^2\E[B^2]}{\E[B](1-\rho)} -\frac{\rho^2\E[B^2]}{\E[B](1-\rho)}\int_{\omega=0}^1 \tilde{K}\left(\omega\right)\dx \omega.
    \end{align*}
    Term (\ref{eq:thisformulaislong_buthasbraces}iii) can be simplified by evaluating the most inner integral and using the substitution: $\omega = {\int_{y=u}^x}^*\pi(y)\dx y$.
    \begin{align*}
        (\ref{eq:thisformulaislong_buthasbraces}iii) &= \frac{1}{\lambda \E[K]}\int_{u=0}^1 \big[\rho \pi(u) + 1-\rho\big]{\int_{x=u^+}^{u^-}} \pi(x)\tilde{K}''\left({\int_{v = u}^x}^*\pi(v)\dx v\right)\bigg\{\exp\left(\rho{\int_{\xi = u}^x}^* \pi(\xi) \dx \xi \right) - 1\bigg\}\dx x \dx u\\
        &=  \frac{1}{\lambda \E[K]}\int_{u=0}^1 \big[\rho \pi(u) + 1-\rho\big]\int_{\omega = 0}^1 \tilde{K}''(\omega)\big\{\exp(\rho\omega)-1\big\}\dx \omega\dx u  
        \\
        &= \frac{1}{\lambda \E[K]} \int_{\omega = 0}^1 \tilde{K}''(\omega)\big\{\exp(\rho\omega)-1\big\}\dx \omega.
        \intertext{We now apply partial integration twice:}
        (\ref{eq:thisformulaislong_buthasbraces}iii)&= \frac{1}{\lambda \E[K]}\bigg[\tilde{K}'(\omega)\big\{\exp(\rho\omega)-1\big\}\bigg]_{\omega = 0}^1 - \frac{1}{\lambda \E[K]}\int_{\omega = 0}^1 \rho \tilde{K}'(\omega)\exp(\rho\omega)\dx \omega\\
        &= \frac{1}{\lambda}\Big(\exp(\rho) -1\Big)  - \E[B]\exp(\rho) + \E[B]\rho\int_{\omega = 0}^1 \tilde{K}(\omega)\exp(\rho\omega)\dx \omega, 
    \end{align*}
    where we used $\tilde{K}'(1) = \E[K]$. Combining these results and including (\ref{eq:thisformulaislong_buthasbraces}iv) gives \eqref{eq:batchsojourn}.
\end{proof}

\section{Proof of Theorem \ref{thm:del}}
\label{app:Deliver}
The derivation of the expected time to delivery becomes more involved, as this time might involve more than an entire cycle of the server.
In turn, this implies that the effect that a customer service has on the time to delivery might involve services that occur more than a cycle later.
Hence, we have to distinguish between several cases.
In this section, we provide the main steps for the derivation and find an exact expression for the expected time to delivery, again dependent on the unknown function $f_K$.\\
The main insight for the derivation is the following result.
Using this result, the expected time to delivery is simply found by substituting in the expression for $f$.
The proof of this statement is deferred to Appendix \ref{app:ProofDel2}.
\begin{appproposition}
    \label{prop:del_pre}
        The expected time to delivery of a batch of customers is given by the following:
    \begin{align*}
        \E[D] &=  \int_{u=0}^1[\rho\pi(u)+  1-\rho]\int_{x=u}^1 \sum_{k=1}^\infty p_k k \pi(x)\bigg({\int_{\nu = u}^x}^* \pi(\nu)\dx \nu\bigg)^{k-1} \\
        &\hspace{6cm}\cdot \E[D\vert S = u, X^B = x, K = k]\dx x \dx u\\
        &\quad + \int_{u=0}^1[\rho\pi(u)+  1-\rho]\int_{x=0}^u \sum_{k=1}^\infty p_k k \pi(x)\bigg({\int_{\nu = u}^x}^* \pi(\nu)\dx \nu\bigg)^{k-1} \\
        &\hspace{6cm}\cdot \E[D\vert S = u, X^B = x, K = k]\dx x \dx u,
    \end{align*}
    where for $u\leq x$ we have:
    \begin{align}
    \label{eq:deliver_conditioning1}
         \E[&D\vert S = u, X^B = x, K = k] \nonumber\\
         &= \E[B]\exp\bigg(\rho\int_{\xi = x}^1 \pi(\xi)\dx \xi\bigg) + \alpha\int_{z=u}^1\exp\bigg(\rho\int_{\xi = z}^1 \pi(\xi)\dx \xi\bigg)\dx z\\
         &\quad + \frac{(k-1)}{{\int_{\nu = u}^x}^* \pi(\nu)\dx \nu}\cdot \frac{1}{\lambda\E[K]}\bigg\{\exp\left(\rho\int_{\xi = u}^1 \pi(\xi)\dx \xi\right) - \exp\left(\rho\int_{\xi = x}^1 \pi(\xi)\dx \xi\right)\bigg\}\nonumber\\
         &\quad + \frac{\rho\pi(u)}{\rho\pi(u) + 1-\rho}\cdot \frac{\E[B^2]}{2\E[B]}\exp\bigg(\rho\int_{\xi = u}^1 \pi(\xi)\dx \xi\bigg) \nonumber\\
         &\quad + \int_{z = u}^1 \E[B]f(z,u)\exp\bigg(\rho\int_{\xi = z}^1 \pi(\xi)\dx \xi\bigg)\dx z,\nonumber
    \end{align}
    and for $u>x$ we have:
    \begin{align}
    \label{eq:deliver_conditioning2}
         \E[&D\vert S = u, X^B = x, K = k]\\
         &= \E[B]\exp\bigg(\rho\int_{\xi = x}^1 \pi(\xi)\dx \xi\bigg)
         + \alpha\int_{z=0}^1\exp\left(\rho\int_{\xi = z}^1 \pi(\xi)\dx \xi\right)\dx z \nonumber \\
         &\quad + \alpha\int_{z=u}^1\bigg\{\exp\bigg(\rho + \rho\int_{\xi = z}^1 \pi(\xi)\dx \xi\bigg)-\rho\int_{\omega=z}^1 \pi(\omega)\dx \omega \exp\left(\rho\int_{\xi = z}^1 \pi(\xi)\dx \xi\right)\bigg\}\dx z \nonumber\\
         &\quad+   \frac{(k-1)}{{\int_{\nu = u}^x}^* \pi(\nu)\dx \nu}\cdot \frac{1}{\lambda\E[K]}\bigg\{\exp\left(\rho\int_{\xi = u}^1 \pi(\xi)\dx \xi\right)  - \exp\left(\rho\int_{\xi = x}^1 \pi(\xi)\dx \xi\right)\bigg\}\nonumber\\
        &\quad + \frac{(k-1)}{{\int_{\nu = u}^x}^* \pi(\nu)\dx \nu}\cdot \frac{1}{\lambda\E[K]}\bigg\{\exp\left(\rho + \rho\int_{\xi = u}^1 \pi(\xi)\dx \xi\right) \nonumber \\
        &\hspace{5.2cm}- \rho\int_{\omega = u}^1\pi(\omega)\dx \omega\exp\left(\rho\int_{\xi = u}^1 \pi(\xi)\dx \xi\right) - 1 \bigg\}\nonumber\\
         &\quad + \frac{\rho\pi(u)}{\rho\pi(u) + 1-\rho}\cdot \frac{\E[B^2]}{2\E[B]}\bigg\{\exp\bigg(\rho + \rho\int_{\xi = u}^1 \pi(\xi)\dx \xi\bigg) \nonumber\\
         &\hspace{5.2cm}-\rho\int_{\omega=u}^1 \pi(\omega)\dx \omega \exp\left(\rho\int_{\xi = u}^1 \pi(\xi)\dx \xi\right)\bigg\}\nonumber\\
         &\quad + \int_{z = 0}^u \E[B]f(z,u)\exp\bigg(\rho\int_{\xi = z}^1 \pi(\xi)\dx \xi\bigg)\dx z\nonumber\\
         &\quad + \int_{z = u}^1 \E[B]f(z,u)\bigg\{\exp\bigg(\rho + \rho\int_{\xi = z}^1 \pi(\xi)\dx \xi\bigg) \nonumber\\
         &\hspace{5.2cm}-\rho\int_{\omega=z}^1 \pi(\omega)\dx \omega \exp\left(\rho\int_{\xi = z}^1 \pi(\xi)\dx \xi\right)\bigg\}\dx z.\nonumber
    \end{align}
\end{appproposition}

\begin{proof}[Proof of Theorem \ref{thm:del} subject to Proposition \ref{prop:del_pre}]
        We substitute the expression for $f$, found in Proposition \ref{prop:partialsol}, into the expression of Proposition \ref{prop:del_pre}. We use that the expression for $f(z,u)$ (excluding the contribution of $f_K(z,u)$) consists of terms proportional to: $\pi(z), \pi(z)\cdot{\int_{\nu = z}^u}^* \dx \nu$, and $\pi(z)\cdot{\int_{\nu = z}^u}^* \pi(\nu)\dx \nu$:
    \begin{equation}
    \label{eq:fsplit}
    \begin{aligned}
            f(z,u) &= f_K(z,u) + \underbrace{\lambda\E[K]\alpha\cdot \bigg\{\pi(z){\int_{\nu = z}^u}^* \dx \nu + \frac{\rho}{1-\rho}\pi(z){\int_{\nu = z}^u}^* \pi(\nu)\dx \nu\bigg\}}_{f_{\alpha}}\\
            &\quad + \underbrace{\lambda\E[K]\frac{\E[B^2]}{\E[B]}\cdot \frac{\rho\pi(u)}{\rho\pi(u)+1-\rho}\cdot \bigg\{\frac{1}{2}\pi(z) + \frac{\rho}{1-\rho}\pi(z){\int_{\nu = z}^u}^* \pi(\nu)\dx \nu\bigg\}}_{f_{B^R}}.
    \end{aligned}
    \end{equation}
    For each of these terms we find the integral over $f$ in \eqref{eq:deliver_conditioning1}. When $u<x$ the following holds:
    \begin{align*}
        \int_{z=u}^1\rho \pi(z) &\exp\left(\rho{\int_{\xi=z}^1}\pi(\xi) \dx \xi\right)\dx z  = \exp\left(\rho{\int_{\xi=u}^1}\pi(\xi) \dx \xi\right)-1.
    \end{align*}
    Using partial integration, we also find:
    \begin{align}
    \label{eq:firstcase_term2}
        \int_{z=u}^1\rho \pi(z) &\exp\left(\rho{\int_{\xi=z}^1}\pi(\xi) \dx \xi\right) {\int_{\nu = z}^u}^* \dx \nu \dx z \\
        &= \bigg[-\exp\left(\rho{\int_{\xi=z}^1}\pi(\xi) \dx \xi\right) {\int_{\nu = z}^u}^* \dx \nu\bigg]_{z = u}^1 - \int_{z=u}^1\exp\left(\rho{\int_{\xi=z}^1}\pi(\xi) \dx \xi\right)\dx z \nonumber \\
        &= \exp\left(\rho{\int_{\xi=u}^1}\pi(\xi) \dx \xi\right) - \int_{\nu=0}^u \dx \nu - \int_{z=u}^1\exp\left(\rho{\int_{\xi=z}^1}\pi(\xi) \dx \xi\right)\dx z. \nonumber 
    \end{align}
    Similarly, we obtain:
    \begin{align*}
        \int_{z=u}^1 &\rho \pi(z) \exp\left(\rho{\int_{\xi=z}^1}\pi(\xi) \dx \xi\right) {\int_{\nu = z}^u}^* \pi(\nu)\dx \nu \dx z\\
        &= \bigg[-\exp\left(\rho{\int_{\xi=z}^1}\pi(\xi) \dx \xi\right) {\int_{\nu = z}^u}^* \pi(\nu)\dx \nu\bigg]_{z = u}^1 - \int_{z=u}^1\pi(z)\exp\left(\rho{\int_{\xi=z}^1}\pi(\xi) \dx \xi\right)\dx z\\
        &= -\frac{1-\rho}{\rho}\exp\left(\rho{\int_{\xi=u}^1}\pi(\xi) \dx \xi\right) - \int_{\nu=0}^u \pi(\nu)\dx \nu + \frac{1}{\rho}.
    \end{align*}
    From this, and \eqref{eq:fsplit}, it follows that for $f_\alpha$ we have:
    \begin{align*}
        \int_{z=u}^1 &f_\alpha(z,u) \E[B] \exp\left(\rho{\int_{\xi=z}^1}\pi(\xi) \dx \xi\right)\dx z \\
        &= \alpha \exp\left(\rho{\int_{\xi=u}^1}\pi(\xi) \dx \xi\right) - \frac{\alpha}{1-\rho}\int_{\nu=0}^u \big[1-\rho\big]\dx \nu - \alpha\int_{z=u}^1\exp\left(\rho{\int_{\xi=z}^1}\pi(\xi) \dx \xi\right)\dx z\\
        &\quad -\alpha\exp\left(\rho{\int_{\xi=u}^1}\pi(\xi) \dx \xi\right) - \frac{\alpha}{1-\rho}\int_{\nu=0}^u \big[\rho\pi(\nu)\big]\dx \nu + \frac{\alpha}{1-\rho}\\
        &= \frac{\alpha}{1-\rho}\int_{\nu=u}^1 \big[\rho\pi(\nu)+1-\rho\big]\dx \nu - \alpha\int_{z=u}^1\exp\left(\rho{\int_{\xi=z}^1}\pi(\xi) \dx \xi\right)\dx z. 
    \end{align*}
    Substituting $f_{B^R}$, see \eqref{eq:fsplit}, into \eqref{eq:deliver_conditioning1} gives:
    \begin{align*}
        \frac{\rho\pi(u)+1-\rho}{\rho\pi(u)}&\int_{z=u}^1 f_{B^R}(z,u)\E[B] \exp\left(\rho{\int_{\xi=z}^1}\pi(\xi) \dx \xi\right)\dx z\\
        &= \frac{\E[B^2]}{2\E[B]}\bigg\{\exp\left(\rho{\int_{\xi=u}^1}\pi(\xi) \dx \xi\right)-1\bigg\} - \frac{\E[B^2]}{\E[B]}\exp\left(\rho{\int_{\xi=u}^1}\pi(\xi) \dx \xi\right)\\
        &\quad - \frac{\rho\E[B^2]}{(1-\rho)\E[B]} \int_{\nu=0}^u\pi(\nu) \dx \nu  + \frac{\E[B^2]}{(1-\rho)\E[B]}\\
        &=\frac{\E[B^2]}{2\E[B]} + \frac{\rho\E[B^2]}{(1-\rho)\E[B]} \int_{\nu=u}^1\pi(\nu) \dx \nu - \frac{\E[B^2]}{2\E[B]}\exp\left(\rho{\int_{\xi=u}^1}\pi(\xi) \dx \xi\right),
    \end{align*}
    where the last step follows from $1/(1-\rho) = 1+ \rho/(1-\rho)$.\\
    Combined, this gives the following expression for \eqref{eq:deliver_conditioning1}, when $u<x$:
    \begin{align*}
    \E[&D\vert S = u, X^B = x, K = k] \nonumber\\
         &= \E[B]\exp\bigg(\rho\int_{\xi = x}^1 \pi(\xi)\dx \xi\bigg) 
         + \frac{\alpha}{1-\rho}{\int_{\nu = u}^1}\big[\rho \pi(\nu) + 1 - \rho]\dx \nu \nonumber\\
         &\quad + \frac{(k-1)}{{\int_{\nu = u}^x}^* \pi(\nu)\dx \nu}\cdot \frac{1}{\lambda\E[K]}\bigg\{\exp\left(\rho\int_{\xi = u}^1 \pi(\xi)\dx \xi\right) - \exp\left(\rho\int_{\xi = x}^1 \pi(\xi)\dx \xi\right)\bigg\}\nonumber\\
         &\quad +\frac{\rho\pi(u)}{\rho\pi(u)+1-\rho}\frac{\rho\E[B^2]}{(1-\rho)\E[B]}\int_{\nu = u}^1\pi(\nu)\dx \nu +\frac{\rho\pi(u)}{\rho\pi(u)+1-\rho}\frac{\E[B^2]}{2\E[B]}\\
           &\quad + \int_{z = u}^1 \E[B]f_K(z,u)\exp\bigg(\rho\int_{\xi = z}^1 \pi(\xi)\dx \xi\bigg)\dx z.
    \end{align*}

    The case $u>x$ requires some additional integration results. We apply the following trick:
    \begin{align}
    \label{eq:delproof_trick}
        {\int_{\xi = z}^u}^*\pi(\xi)\dx \xi + {\int_{\xi = u}^1}\pi(\xi)\dx \xi = \begin{dcases}
             {\int_{\xi = z}^1}\pi(\xi)\dx \xi &\text{if } z\leq u\\
             1 + {\int_{\xi = z}^1}\pi(\xi)\dx \xi &\text{if } z>u.
        \end{dcases}
    \end{align}
    Therefore, we have for instance that for $u>x$:
    \begin{align}
    \label{eq:delivery_secondcase_term3}
        \int_{z=u}^{1}\rho&\pi(z)\bigg\{\exp\left(\rho + \rho{\int_{\xi = z}^1}\pi(\xi)\dx \xi\right) - \rho\int_{\omega = z}^1\pi(\omega)\dx \omega\exp\left(\rho{\int_{\xi = z}^1}\pi(\xi)\dx \xi\right)\bigg\}\dx z \nonumber\\
        &\quad + \int_{z=0}^u \rho\pi(z) \cdot \exp\left( \rho{\int_{\xi = z}^1}\pi(\xi)\dx \xi\right)\dx z \nonumber\\
        & = {\int_{z=u^+}^{u^-}}^*\rho\pi(z)\exp\left(\rho {\int_{\xi = z}^u}^*\pi(\xi)\dx \xi  +  \rho{\int_{\xi = u}^1}\pi(\xi)\dx \xi\right) \dx z\nonumber\\
        &\quad -  \int_{z=u}^1 \rho^2\pi(z)\cdot\int_{\omega = z}^1\pi(\omega)\dx \omega\exp\left(\rho{\int_{\xi = z}^1}\pi(\xi)\dx \xi\right)\dx z. \nonumber\\
        \intertext{By substitution of $\zeta = {\int_{\xi = z}^u}^*\pi(\xi)\dx \xi$ in the first term on the right-hand side and $\zeta = \int_{\xi = z}^1 \pi(\xi)\dx \xi$ in the second term on the right-hand side, we now find:}
        \int_{z=u}^{1}\rho&\pi(z)\bigg\{\exp\left(\rho + \rho{\int_{\xi = z}^1}\pi(\xi)\dx \xi\right) - \rho\int_{\omega = z}^1\pi(\omega)\dx \omega\exp\left(\rho{\int_{\xi = z}^1}\pi(\xi)\dx \xi\right)\bigg\}\dx z \nonumber\\
        &\quad + \int_{z=0}^u \rho\pi(z) \cdot \exp\left( \rho{\int_{\xi = z}^1}\pi(\xi)\dx \xi\right)\dx z \nonumber\\
        &= \exp\left(\rho {\int_{\xi = u}^1}\pi(\xi)\dx \xi\right)\int_{\zeta = 0}^1\rho\exp(\rho \zeta)\dx \zeta 
        - \int_{\zeta = 0}^{\int_{\xi=u}^1\pi(\xi)\dx \xi}\rho^2\zeta\exp(\rho\zeta)\dx \zeta \nonumber\\
        &=\exp\left(\rho +\rho{\int_{\xi = u}^1}\pi(\xi)\dx \xi\right) 
        - \exp\left(\rho{\int_{\xi = u}^1}\pi(\xi)\dx \xi\right) 
        -  1 
        + \exp\left(\rho{\int_{\xi = u}^1}\pi(\xi)\dx \xi\right) \nonumber\\
        &\quad - \rho {\int_{\omega=z}^1}\pi(\omega)\dx \omega\exp\left(\rho{\int_{\xi = u}^1}\pi(\xi)\dx \xi\right).
    \end{align}
    Using \eqref{eq:delproof_trick} and partial integration also shows the following:
    \begin{align}
    \label{eq:delivery_secondcase_term1}
        {\int_{z=u^+}^{u^-}}^*&\rho\pi(z){\int_{\nu=z}^u}^*\dx \nu \exp\left(\rho {\int_{\xi = z}^u}^*\pi(\xi)\dx \xi +  \rho{\int_{\xi = u}^1}\pi(\xi)\dx \xi\right) \dx z\nonumber\\
        &\quad -  \int_{z=u}^1 \rho^2\pi(z){\int_{\nu=z}^u}^*\dx \nu \cdot\int_{\omega = z}^1\pi(\omega)\dx \omega\exp\left(\rho{\int_{\xi = z}^1}\pi(\xi)\dx \xi\right)\dx z \nonumber\\
        &=\underbrace{\bigg[-{\int_{\nu=z}^u}^*\dx \nu\exp\left(\rho {\int_{\xi = z}^u}^*\pi(\xi)\dx \xi +  \rho{\int_{\xi = u}^1}\pi(\xi)\dx \xi\right)\bigg]_{z=u^+}^{u^-}}_{(i)} \nonumber\\
        &\quad - \underbrace{{\int_{z=u^+}^{u^-}}^*\exp\left(\rho {\int_{\xi = z}^u}^*\pi(\xi)\dx \xi +  \rho{\int_{\xi = u}^1}\pi(\xi)\dx \xi\right) \dx z}_{(ii)}\nonumber\\
        &\quad - \underbrace{\bigg[-\rho{\int_{\nu=z}^u}^*\dx \nu \cdot\int_{\omega = z}^1\pi(\omega)\dx \omega\exp\left(\rho{\int_{\xi = z}^1}\pi(\xi)\dx \xi\right)\bigg]_{z=u}^{1}}_{(iii)}\nonumber \\
        &\quad + \underbrace{\int_{z=u}^1 \rho\pi(z){\int_{\nu=z}^u}^*\dx \nu \exp\left(\rho{\int_{\xi = z}^1}\pi(\xi)\dx \xi\right)\dx z}_{(iv)}\nonumber\\
        &\quad + \underbrace{\int_{z=u}^1 \rho{\int_{\omega=z}^u}^*\pi(\omega)\dx \omega\exp\left(\rho{\int_{\xi = z}^1}\pi(\xi)\dx \xi\right)\dx z}_{(v)}\nonumber\\
        &=\underbrace{\exp\left(\rho +  \rho{\int_{\xi = u}^1}\pi(\xi)\dx \xi\right) }_{(i)}
        - \underbrace{{\int_{z=u^+}^{u^-}}^*\exp\left(\rho {\int_{\xi = z}^u}^*\pi(\xi)\dx \xi +  \rho{\int_{\xi = u}^1}\pi(\xi)\dx \xi\right) \dx z}_{(ii)}\nonumber\\
        &\quad - \underbrace{\rho\int_{\omega = u}^1\pi(\omega)\dx \omega\exp\left(\rho{\int_{\xi = u}^1}\pi(\xi)\dx \xi\right)}_{(iii)}
        + 
        \underbrace{\exp\left(\rho{\int_{\xi = u}^1}\pi(\xi)\dx \xi\right) - \int_{\nu = 0}^u\dx \nu}_{(iv)}\nonumber\\
        &\quad - \underbrace{\int_{z=u}^1 \exp\left(\rho{\int_{\xi = z}^1}\pi(\xi)\dx \xi\right)\dx z}_{(iv)}
        + \underbrace{\int_{z=u}^1 \rho{\int_{\omega=z}^u}^*\pi(\omega)\dx \omega\exp\left(\rho{\int_{\xi = z}^1}\pi(\xi)\dx \xi\right)\dx z}_{(v)},
    \end{align}
    where we simplified $(iv)$ by means of partial integration, similar to \eqref{eq:firstcase_term2}. Reordering the terms, splitting the second integral over the ranges $u$ to 1 and 0 to $u$, and applying \eqref{eq:delproof_trick}, we find:
    \begin{align*}
        \eqref{eq:delivery_secondcase_term1} &= 
        \underbrace{\exp\left(\rho +  \rho{\int_{\xi = u}^1}\pi(\xi)\dx \xi\right)}_{(i)}
        - \underbrace{\rho\int_{\omega = u}^1\pi(\omega)\dx \omega\exp\left(\rho{\int_{\xi = u}^1}\pi(\xi)\dx \xi\right)}_{(iii)} \\
        &\quad  + \underbrace{\exp\left(\rho{\int_{\xi = u}^1}\pi(\xi)\dx \xi\right)
        - \int_{\nu = 0}^u\dx \nu}_{(iv)}\\
        &\quad - \underbrace{\int_{z=u}^1 \bigg\{\exp\left(\rho + \rho{\int_{\xi = z}^1}\pi(\xi)\dx \xi\right)-\rho{\int_{\omega=z}^1}\pi(\omega)\dx \omega\exp\left(\rho{\int_{\xi = z}^1}\pi(\xi)\dx \xi\right)\bigg\}\dx z}_{(ii)+(v)}\\
        &\quad - \underbrace{\int_{z=0}^1 \exp\left(\rho{\int_{\xi = z}^1}\pi(\xi)\dx \xi\right)\dx z}_{(ii)+(iv)}.
    \end{align*}
    Lastly, for the terms proportional to $\pi(z)\cdot{\int_{\nu = z}^u}^* \pi(\nu)\dx \nu$, the substitution $\zeta = {\int_{\xi = z}^u}^*\pi(\xi)\dx \xi$, after applying \eqref{eq:delproof_trick}, gives:
   \begin{align}
    \label{eq:delivery_secondcase_term2}
        \int_{z=u}^1 &\rho\pi(z){\int_{\nu=z}^u}^*\pi(\nu)\dx \nu \cdot \bigg\{\exp\left(\rho + \rho{\int_{\xi = z}^1}\pi(\xi)\dx \xi\right) \nonumber\\
        &\hspace{4.1cm}- \rho\int_{\omega = z}^1\pi(\omega)\dx \omega\exp\left(\rho{\int_{\xi = z}^1}\pi(\xi)\dx \xi\right)\bigg\}\dx z \nonumber\\
        &\quad +\int_{z=0}^u \rho\pi(z){\int_{\nu=z}^u}^*\pi(\nu)\dx \nu \cdot \exp\left( \rho{\int_{\xi = z}^1}\pi(\xi)\dx \xi\right)\dx z \nonumber\\
        & =\int_{z=u^+}^{u^-}\rho\pi(z){\int_{\nu=z}^u}^*\pi(\nu)\dx \nu \exp\left(\rho {\int_{\xi = z}^u}^*\pi(\xi)\dx \xi +  \rho{\int_{\xi = u}^1}\pi(\xi)\dx \xi\right) \dx z\nonumber\\
        &\quad -  \begin{aligned}[t]  \int_{z=u}^1 \rho^2\pi(z){\int_{\nu=z}^u}^*\pi(\nu)\dx \nu \cdot&\bigg({\int_{\omega = z}^u}^*\pi(\omega)\dx \omega - {\int_{\omega = 0}^u}\pi(\omega)\dx \omega\bigg)\\
        &\cdot\exp\left(\rho{\int_{\xi = z}^u}^*\pi(\xi)\dx \xi-\rho{\int_{\xi = 0}^u}\pi(\xi)\dx \xi\right)\dx z 
        \end{aligned}\nonumber\\
        &=  \exp\left(\rho{\int_{\xi = u}^1}\pi(\xi)\dx \xi\right)\int_{\zeta=0}^{1} \rho \zeta\exp(\rho \zeta)\dx \zeta\\
        &\quad -\int_{\zeta=\int_{\nu=0}^u\pi(\nu)\dx \nu}^{1} \rho^2 \zeta\cdot\bigg(\zeta- {\int_{\omega = 0}^u}\pi(\omega)\dx \omega\bigg)\exp\left(\rho \zeta-\rho{\int_{\xi = 0}^u}\pi(\xi)\dx \xi\right)\dx \zeta \nonumber.
    \end{align}
    Evaluating the integrals gives:
    \begin{align*}
        \eqref{eq:delivery_secondcase_term2} &=  \frac{1}{\rho}\exp\left(\rho{\int_{\xi = u}^1}\pi(\xi)\dx \xi\right) - \frac{1-\rho}{\rho}\exp\left(\rho + \rho{\int_{\xi = u}^1}\pi(\xi)\dx \xi\right)\\
        &\quad + \frac{2}{\rho} - \int_{\nu=0}^u\pi(\nu)\dx \nu - \frac{2-\rho}{\rho}\exp\left(\rho{\int_{\xi = u}^1}\pi(\xi)\dx \xi\right)\\
        &\quad + (1-\rho){\int_{\omega=u}^1}\pi(\omega)\dx \omega\exp\left(\rho{\int_{\xi = u}^1}\pi(\xi)\dx \xi\right)\\
        &=-\frac{1-\rho}{\rho}\exp\left(\rho{\int_{\xi = u}^1}\pi(\xi)\dx \xi\right)-\frac{1-\rho}{\rho}\exp\left(\rho + \rho{\int_{\xi = u}^1}\pi(\xi)\dx \xi\right)\\
        &\quad  + (1-\rho){\int_{\omega=u}^1}\pi(\omega)\dx \omega\exp\left(\rho{\int_{\xi = u}^1}\pi(\xi)\dx \xi\right) + \frac{2}{\rho} -  \int_{\nu=0}^u\pi(\nu)\dx \nu.
    \end{align*}
    We can now use this for the calculation of the integral over $f$ in \eqref{eq:deliver_conditioning2}. First we substitute in $f_{\alpha}$ (see \eqref{eq:splitting}), this is equal to $\alpha$ times \eqref{eq:delivery_secondcase_term1} plus $\alpha \rho/(1-\rho)$ times \eqref{eq:delivery_secondcase_term2}:
    \begin{align*}
        \int_{z = 0}^u &\E[B]f_{\alpha}(z,u)\exp\bigg(\rho\int_{\xi = z}^1 \pi(\xi)\dx \xi\bigg)\dx z\\
         &\quad + \int_{z = u}^1 \E[B]f_{\alpha}(z,u)\bigg\{\exp\bigg(\rho + \rho\int_{\xi = z}^1 \pi(\xi)\dx \xi\bigg)\\
         &\hspace{4cm}-\rho\int_{\omega=z}^1 \pi(\omega)\dx \omega \exp\left(\rho\int_{\xi = z}^1 \pi(\xi)\dx \xi\right)\bigg\}\dx z\\
         &=\alpha\exp\left(\rho +  \rho{\int_{\xi = u}^1}\pi(\xi)\dx \xi\right) 
        - \alpha\rho \int_{\omega = u}^1\pi(\omega)\dx \omega\exp\left(\rho{\int_{\xi = u}^1}\pi(\xi)\dx \xi\right) \\
        &\quad  + \alpha\exp\left(\rho{\int_{\xi = u}^1}\pi(\xi)\dx \xi\right)
        - \frac{\alpha}{1-\rho}\int_{\nu = 0}^u\big[1-\rho\big]\dx \nu\\
        &\quad - \alpha\int_{z=u}^1 \bigg\{\exp\left(\rho + \rho{\int_{\xi = z}^1}\pi(\xi)\dx \xi\right)-\rho{\int_{\omega=z}^1}\pi(\omega)\dx \omega\exp\left(\rho{\int_{\xi = z}^1}\pi(\xi)\dx \xi\right)\bigg\}\dx z\\
        &\quad - \alpha\int_{z=0}^1 \exp\left(\rho{\int_{\xi = z}^1}\pi(\xi)\dx \xi\right)\dx z\\
        &\quad - \alpha\exp\left(\rho{\int_{\xi = u}^1}\pi(\xi)\dx \xi\right) - \alpha\exp\left(\rho + \rho{\int_{\xi = u}^1}\pi(\xi)\dx \xi\right)\\
        &\quad +\alpha\rho\int_{\omega = u}^1\pi(\omega)\dx \omega\exp\left(\rho{\int_{\xi = u}^1}\pi(\xi)\dx \xi\right) + \frac{2\alpha}{1-\rho} - \frac{\alpha}{1-\rho}\int_{\nu = 0}^u\big[\rho \pi(\nu)\big]\dx \nu\\
        &=\frac{\alpha}{1-\rho} + \frac{\alpha}{1-\rho}\int_{\nu = u}^1\big[\rho \pi(\nu) + 1-\rho\big]\dx \nu \\
        &\quad - \alpha\int_{z=u}^1 \bigg\{\exp\left(\rho + \rho{\int_{\xi = z}^1}\pi(\xi)\dx \xi\right)-\rho{\int_{\omega=z}^1}\pi(\omega)\dx \omega\exp\left(\rho{\int_{\xi = z}^1}\pi(\xi)\dx \xi\right)\bigg\}\dx z\\
        &\quad - \alpha\int_{z=0}^1 \exp\left(\rho{\int_{\xi = z}^1}\pi(\xi)\dx \xi\right)\dx z.
    \end{align*}
    Secondly, we substitute $f_{B^R}$ (see \eqref{eq:splitting}), noting that this gives $\frac{\E[B^2]}{2\E[B]}$ times \eqref{eq:delivery_secondcase_term3} plus  $\frac{\rho\E[B^2]}{(1-\rho)\E[B]}$ times \eqref{eq:delivery_secondcase_term2}:
    \begin{align*}
        \frac{\rho\pi(u)+1-\rho}{\rho\pi(u)}&\begin{aligned}[t]
            \cdot\bigg
            \{&\int_{z = 0}^u \E[B]f_{B^R}(z,u)\exp\bigg(\rho\int_{\xi = z}^1 \pi(\xi)\dx \xi\bigg)\dx z\\
            &\quad + \int_{z = u}^1 \E[B]f_{B^R}(z,u)\bigg\{\exp\bigg(\rho + \rho\int_{\xi = z}^1 \pi(\xi)\dx \xi\bigg)\\
            &\hspace{4.3cm}-\rho\int_{\omega=z}^1 \pi(\omega)\dx \omega \exp\left(\rho\int_{\xi = z}^1 \pi(\xi)\dx \xi\right)\bigg\}\dx z
        \end{aligned}\\
         &\hspace{-1cm}=-\frac{\E[B^2]}{\E[B]}\exp\left(\rho \int_{\xi = u}^1 \pi(\xi)\dx \xi\right) - \frac{\E[B^2]}{\E[B]}\exp\left(\rho + \rho \int_{\xi = u}^1 \pi(\xi)\dx \xi\right)\\
         &\hspace{-1cm}\quad + \frac{\rho\E[B^2]}{\E[B]}\int_{\omega = u}^1 \pi(\omega)\dx\omega \exp\left(\rho\int_{\xi = u}^1\pi(\xi)\dx \xi\right) + \frac{2\E[B^2]}{(1-\rho)\E[B]}\\
         &\hspace{-1cm}\quad- \frac{\rho\E[B^2]}{(1-\rho)\E[B]}\int_{\nu=0}^u\pi(\nu)\dx\nu+ \frac{\E[B^2]}{2\E[B]}\exp\left(\rho + \rho \int_{\xi = u}^1 \pi(\xi)\dx \xi\right)\\
         &\hspace{-1cm}\quad - \frac{\E[B^2]}{2\E[B]} - \frac{\rho\E[B^2]}{2\E[B]}\int_{\omega = u}^1 \pi(\omega)\dx\omega \exp\left(\rho\int_{\xi = u}^1\pi(\xi)\dx \xi\right)\\
         &\hspace{-1cm}= \frac{\E[B^2]}{2\E[B]} + \frac{\E[B^2]}{(1-\rho)\E[B]} + \frac{\rho\E[B^2]}{(1-\rho)\E[B]}\int_{\nu=u}^1\pi(\nu)\dx\nu - \frac{\E[B^2]}{\E[B]}\exp\left(\rho\int_{\xi=u}^1\pi(\xi)\dx\xi\right)\\
         &\hspace{-1cm}\quad - \frac{\E[B^2]}{2\E[B]}\cdot\bigg\{\exp\left(\rho + \rho \int_{\xi = u}^1 \pi(\xi)\dx \xi\right) - \rho\int_{\omega = u}^1 \pi(\omega)\dx\omega \exp\left(\rho\int_{\xi = u}^1\pi(\xi)\dx \xi\right)\bigg\},
    \end{align*}
    where the last step follows from the fact that $2/(1-\rho) = \rho/(1-\rho)+1/(1-\rho) + 1$.\\
    Substituting this into \eqref{eq:deliver_conditioning2}, we obtain the following for $u>x$:
    \begin{align*}
        \E[&D\vert S = u, X^B = x, K = k]\\
        &=\E[B]\exp\bigg(\rho\int_{\xi = x}^1 \pi(\xi)\dx \xi\bigg) 
        + \frac{\alpha}{1-\rho}+\frac{\alpha}{1-\rho}\int_{\nu=u}^1\big[\rho\pi(\nu)+1-\rho\big]\dx \nu \nonumber\\
        &\quad+   \frac{(k-1)}{{\int_{\nu = u}^x}^* \pi(\nu)\dx \nu}\cdot \frac{1}{\lambda\E[K]}\bigg\{\exp\left(\rho\int_{\xi = u}^1 \pi(\xi)\dx \xi\right)  - \exp\left(\rho\int_{\xi = x}^1 \pi(\xi)\dx \xi\right)\bigg\}\nonumber\\
         &\quad + \frac{(k-1)}{{\int_{\nu = u}^x}^* \pi(\nu)\dx \nu}\cdot \frac{1}{\lambda\E[K]}\bigg\{\exp\left(\rho + \rho\int_{\xi = u}^1 \pi(\xi)\dx \xi\right)  \nonumber\\
         &\hspace{5.3cm}- \rho\int_{\omega = u}^1\pi(\omega)\dx \omega\exp\left(\rho\int_{\xi = u}^1 \pi(\xi)\dx \xi\right) - 1 \bigg\}\nonumber\\
         &\quad \\
         &\quad + \frac{\rho\pi(u)}{\rho\pi(u)+1-\rho}\frac{\E[B^2]}{\E[B]}\cdot\bigg\{\frac{1}{2} + \frac{1}{(1-\rho)}  \\
         &\hspace{5.3cm}+ \frac{\rho}{(1-\rho)}\int_{\nu=u}^1\pi(\nu)\dx\nu- \exp\left(\rho\int_{\xi=u}^1\pi(\xi)\dx\xi\right)\bigg\}\\
         &\quad + \int_{z = 0}^u \E[B]f_K(z,u)\exp\bigg(\rho\int_{\xi = z}^1 \pi(\xi)\dx \xi\bigg)\dx z\\
         &\quad + \int_{z= u}^1 \E[B]f_K(z,u)\bigg\{\exp\bigg(\rho + \rho\int_{\xi =z}^1 \pi(\xi)\dx \xi\bigg)\\
         &\hspace{5.3cm}-\rho\int_{\omega=z}^1 \pi(\omega)\dx \omega \exp\left(\rho\int_{\xi = z}^1 \pi(\xi)\dx \xi\right)\bigg\}\dx z.
    \end{align*}
    
We now turn to derive the expected time to delivery, where we observe the similarity between the two cases ($u\leq x$ and $u>x$) and realize the absence of $x$ in many of the terms. For sake of notation, we introduce $\E[D_K]$ as the contribution of $f_K(x,y)$ to $\E[D]$, that is:
\begin{align}
\label{eq:DKdef}
       \E[D_K] &= \int_{u=0}^1[\rho\pi(u)+  1-\rho]\tilde{K}\left(\int_{\nu = u}^1\pi(\nu)\dx \nu\right)\nonumber\\
       &\hspace{1cm}\cdot\int_{z=u}^1 f_K(z,u)\E[B]\exp\left(\rho\int_{\xi = z}^1\pi(\xi)\dx \xi \right)\dx z\dx u\nonumber\\
        &\quad + \int_{u=0}^1[\rho\pi(u)+  1-\rho]\left[1-\tilde{K}\left(\int_{\nu = u}^1\pi(\nu)\dx \nu\right)\right]\nonumber\\
        &\hspace{1cm}\cdot\int_{z=u}^1 f_K(z,u)\E[B]\cdot\Bigg[\exp\left(\rho\int_{\xi = z}^1\pi(\xi)\dx \xi+\rho\right)\\
        &\hspace{5cm}-\rho\int_{\omega=z}^1\pi(z)\exp\left(\rho\int_{\xi = z}^1\pi(\xi)\dx \xi\right)\Bigg] \dx z\dx u\nonumber\\
        &\quad + 
        \int_{u=0}^1[\rho\pi(u)+  1-\rho]\left[1-\tilde{K}\left(\int_{\nu = u}^1\pi(\nu)\dx \nu\right)\right]\nonumber\\
        &\hspace{1cm}\cdot \int_{z=0}^u f_K(z,u)\E[B]\exp\left(\rho\int_{\xi = z}^1\pi(\xi)\dx \xi \right)\dx z\dx u.\nonumber
\end{align}

We denote the terms dependent on $f_K$ as $\E[D_K]$. Using that $\tilde{K}(x) = \sum_{k=1}^\infty p_k x^k$, we find:
\begin{align*}
    \E[D] &= \underbrace{\E[B]\int_{u=0}^1[\rho\pi(u)+  1-\rho]{\int_{x=0}^{1}}^*\pi(x)\tilde{K}'\bigg({\int_{\nu = u}^x}^*\pi(\nu)\dx \nu\bigg)\exp\left(\rho \int_{\xi = x}^1 \pi(\xi)\dx \xi\right)\dx x \dx u}_{(i)}\\
    &\quad + \underbrace{\begin{aligned}[t]
        \frac{1}{\lambda \E[K]}\int_{u=0}^1[\rho&\pi(u)+  1-\rho]{\int_{x=0}^{1}}^*\pi(x)\tilde{K}''\bigg({\int_{\nu = u}^x}^*\pi(\nu)\dx \nu\bigg)\\
        &\cdot\bigg\{\exp\left(\rho \int_{\xi = u}^1 \pi(\xi)\dx \xi\right)-\exp\left(\rho \int_{\xi = x}^1 \pi(\xi)\dx \xi\right)\bigg\} \dx x\dx u
    \end{aligned}}_{(ii)} \\
    &\quad + \underbrace{\begin{aligned}[t]
        \frac{1}{\lambda \E[K]}\int_{u=0}^1[\rho&\pi(u)+  1-\rho]\int_{x=0}^{u}\pi(x)\tilde{K}''\bigg({\int_{\nu = u}^x}^*\pi(\nu)\dx \nu\bigg)\\
        &\cdot\bigg\{\exp\left(\rho + \rho \int_{\xi = u}^1 \pi(\xi)\dx \xi \right)\\
        &\quad\quad-\rho\int_{\omega = u}^1\pi(\omega)\dx \omega\exp\left(\rho\int_{\xi = u}^1 \pi(\xi)\dx \xi\right) - 1\bigg\}\dx x \dx u
    \end{aligned}}_{(iii)} \\
    &\quad + \underbrace{\frac{\alpha}{1-\rho}\int_{u=0}^1[\rho\pi(u)+  1-\rho]\int_{\nu = u}^1\big[\rho \pi(\nu) + 1-\rho]\dx \nu \dx u}_{(iv)}\\
    &\quad +\underbrace{\frac{\alpha}{1-\rho}\int_{u=0}^1[\rho\pi(u)+  1-\rho]\int_{x=0}^{u}\pi(x)\tilde{K}'\bigg({\int_{\nu = u}^x}^*\pi(\nu)\dx \nu\bigg)\dx x \dx u}_{(v)}\\
    &\quad + \underbrace{\frac{\rho\E[B^2]}{2\E[B]}\int_{u=0}^1 \pi(u)\dx u}_{(vi)} + \underbrace{\frac{\rho^2\E[B^2]}{\E[B](1-\rho)}\int_{u=0}^1\pi(u)\int_{\nu = u}^1 \pi(\nu)\dx \nu \dx u}_{(vii)}\\
    &\quad +\underbrace{\frac{\rho\E[B^2]}{\E[B](1-\rho)}\int_{u=0}^1\pi(u)\int_{x=0}^{u}\pi(x)\tilde{K}'\bigg({\int_{\nu = u}^x}^*\pi(\nu)\dx \nu\bigg)\dx x\dx u}_{(viii)}\\
    &\quad - \underbrace{\frac{\rho\E[B^2]}{\E[B]}\int_{u=0}^1\pi(u)\exp\left(\rho\int_{\xi =u}^1\pi(\xi)\dx \xi\right)\int_{x=0}^{u}\pi(x)\tilde{K}'\bigg({\int_{\nu = u}^x}^*\pi(\nu)\dx \nu\bigg)\dx x\dx u}_{(ix)}\\
    &\quad +\E[D_K].
\end{align*}
We realize the following: 
\begin{itemize}
    \item The double integrals in $(iv)$ and $(vii)$ describe the probability $\mathbb{P}(X_1 \leq X_2)$ for two i.i.d. random variables $X_1,X_2$. Hence, this simplifies to $1/2$.
    \item The integral of $(vi)$ is simply 1.
    \item The integrals over $x$ of terms $(v)$, $(viii)$ and $(ix)$ can be evaluated, giving: $1-\tilde{K}\left({\int_{\nu = u}^1}\pi(\nu)\dx \nu\right)$. Similarly, in term $(iii)$ the integral over $x$ results in: $\E[K]-\tilde{K}'\left({\int_{\nu = u}^1}\pi(\nu)\dx \nu\right)$.
    \item We split the integral over $x$ in term $(ii)$ over the range $0$ to $u$ and $u$ to $1$ to account for the discontinuities of the integrand at $u$ and $1$. We then apply partial integration to these terms.
\end{itemize}
This results in:
\begin{align*}
     \E[D] &= \underbrace{\E[B]\int_{u=0}^1[\rho\pi(u)+  1-\rho]{\int_{x=0}^{1} }^*\pi(x)\tilde{K}'\bigg({\int_{\nu = u}^x}^*\pi(\nu)\dx \nu\bigg)\exp\left(\rho \int_{\xi = x}^1 \pi(\xi)\dx \xi\right)\dx x \dx u}_{(i)}\\
    &\quad +
    \underbrace{\begin{aligned}[t]
        \frac{1}{\lambda \E[K]}\int_{u=0}^1&[\rho\pi(u)+  1-\rho]\\
        &\cdot -\tilde{K}'\bigg({\int_{\nu = u}^1}\pi(\nu)\dx \nu\bigg) \cdot\bigg\{\exp\left(\rho \int_{\xi = u}^1 \pi(\xi)\dx \xi\right)-\exp\left(\rho\right)\bigg\}\dx u\\
    \end{aligned}}_{(ii)}\\
        &\quad +
    \underbrace{\begin{aligned}[t]
        \frac{1}{\lambda \E[K]}\int_{u=0}^1&[\rho\pi(u)+  1-\rho]\\
        &\cdot \tilde{K}'\bigg({\int_{\nu = u}^1}\pi(\nu)\dx \nu\bigg) \cdot\bigg\{\exp\left(\rho \int_{\xi = u}^1 \pi(\xi)\dx \xi\right)-1\bigg\}\dx u\\
    \end{aligned}}_{(ii)}\\
    &\quad \underbrace{- \E[B]\int_{u=0}^1[\rho\pi(u)+  1-\rho]{\int_{x=0}^{1}}^*\pi(x)\tilde{K}'\bigg({\int_{\nu = u}^x}^*\pi(\nu)\dx \nu\bigg)\exp\left(\rho \int_{\xi = x}^1 \pi(\xi)\dx \xi\right)\dx x \dx u}_{(ii)}
    \\
    &\quad +
    \underbrace{\begin{aligned}[t]
        \frac{1}{\lambda \E[K]}\int_{u=0}^1[\rho\pi(u)&+  1-\rho]\cdot\bigg[\E[K]-\tilde{K}'\left({\int_{\nu = u}^1}\pi(\nu)\dx \nu\right)\bigg]\\
         &\cdot\bigg\{\exp\left(\rho \int_{\xi = u}^1 \pi(\xi)\dx \xi + \rho \right)\\
         &\quad\quad-\rho\int_{\omega = u}^1\pi(\omega)\dx \omega\exp\left(\rho\int_{\xi = u}^1 \pi(\xi)\dx \xi\right) - 1\bigg\}\dx u
    \end{aligned}}_{(iii)}\\
     &\quad + \underbrace{\frac{\alpha}{2(1-\rho)}}_{(iv)} + \underbrace{\frac{\alpha}{(1-\rho)}\int_{u=0}^1[\rho\pi(u)+  1-\rho]\cdot\bigg[1-\tilde{K}\left({\int_{\nu = u}^1}\pi(\nu)\dx \nu\right)\bigg]\dx u}_{(v)}\\
     &\quad + \underbrace{\rho\frac{\E[B^2]}{2\E[B]}}_{(vi)} + \underbrace{\frac{\rho^2\E[B^2]}{2(1-\rho)\E[B]}}_{(vii)} + \underbrace{\frac{\rho\E[B^2]}{(1-\rho)\E[B]}\int_{u=0}^1\pi(u)\cdot\bigg[1-\tilde{K}\left({\int_{\nu = u}^1}\pi(\nu)\dx \nu\right)\bigg]\dx u}_{(vii)}\\
    &\quad - \underbrace{\frac{\rho\E[B^2]}{\E[B]}\int_{u=0}^1\pi(u)\exp\left(\rho\int_{\xi =u}^1\pi(\xi)\dx \xi\right)\cdot\bigg[1-\tilde{K}\left({\int_{\nu = u}^1}\pi(\nu)\dx \nu\right)\bigg]\dx u}_{(ix)}\\
    &\quad + \E[D_K].
\end{align*}
We now proceed by cancelling $(i)$ against the third term of $(ii)$, and cancelling the exponentials in the first and second term of $(ii)$. We also apply the substitution $\int_{\nu = u}^1\pi(\nu)\dx \nu = \omega$ in terms $(vii)$ and $(ix)$. We then obtain:
\begin{align*}
     \E[D] &= \underbrace{ \frac{1}{\lambda \E[K]}\cdot\big(\exp(\rho)-1\big)\cdot\int_{u=0}^1[\rho\pi(u)+  1-\rho]\cdot \tilde{K}'\left({\int_{\nu = u}^1}\pi(\nu)\dx \nu\right)\dx u}_{(ii)}\\
     &\quad + \underbrace{\begin{aligned}[t]
        \frac{1}{\lambda \E[K]}\int_{u=0}^1[\rho\pi(u)&+  1-\rho]\cdot\bigg[\E[K]-\tilde{K}'\left({\int_{\nu = u}^1}\pi(\nu)\dx \nu\right)\bigg]\\
         &\cdot\bigg\{\exp\left(\rho \int_{\xi = u}^1 \pi(\xi)\dx \xi + \rho \right)\\
         &\quad\quad-\rho\int_{\omega = u}^1\pi(\omega)\dx \omega\exp\left(\rho\int_{\xi = u}^1 \pi(\xi)\dx \xi\right) - 1\bigg\}\dx u
    \end{aligned}}_{(iii)}\\
     &\quad + \underbrace{\frac{\alpha}{2(1-\rho)}}_{(iv)} + \underbrace{\frac{\alpha}{(1-\rho)}\int_{u=0}^1[\rho\pi(u)+  1-\rho]\cdot\bigg[1-\tilde{K}\left({\int_{\nu = u}^1}\pi(\nu)\dx \nu\right)\bigg]\dx u}_{(v)}\\
     &\quad + \underbrace{\rho\frac{\E[B^2]}{2\E[B]}}_{(vi)} + \underbrace{\frac{\rho^2\E[B^2]}{2(1-\rho)\E[B]}}_{(vii)} + \underbrace{\frac{\rho\E[B^2]}{(1-\rho)\E[B]}\int_{\omega=0}^1\big[1-\tilde{K}(\omega)\big]\dx \omega}_{(vii)}\\
    &\quad - \underbrace{\frac{\rho\E[B^2]}{\E[B]}\int_{\omega=0}^1\exp\left(\rho\omega\right)\cdot\bigg[1-\tilde{K}\left(\omega\right)\bigg]\dx \omega}_{(ix)}+ \E[D_K].
\end{align*}
The proof is finished by using that $\E[K/(K+1)] = \int_{\omega = 0}^1 [1-\tilde{K}(\omega)]\dx \omega$, and using that $\E[D_K]$ is given by \eqref{eq:DKdef}.
\end{proof}

\subsection{Proof of Proposition \ref{prop:del_pre}}
\label{app:ProofDel2}
The proof of Proposition \ref{prop:del_pre} is mainly based upon an extension of Lemma \ref{lemma:generatedwaiting}. Because the time to delivery might include more than one cycle of the server, we need to adapt the definition of the branching process construction of Section \ref{sec:fpe}. Specifically, we now trim the branching process after the delivery of a tagged customer. For reference, see Figure \ref{fig:deliverBranching} (taken from \citep{Engels}).\\
In turn, this calls for adapted definitions both $S(\cdot, \cdot)$ and $T(\cdot,\cdot)$. We now let $S_D(x,y)$ be the extra time to delivery of a customer at location $y$ generated by a service at location $x$. Similarly, $T_D(x,y)$ denotes the extra delivery time generated by the travelling of the server from location $x$ to the moment of delivery of the customer at location $y$. An attentive reader might remark that the location $y$ (of the tagged customer), in both $S_D(x,y)$ and $T_D(x,y)$, merely serves as an indicator for when this customer is served, either at the end of the same cycle as the service at location $x$ or the cycle thereafter. To further clarify, when a tagged customer is positioned clockwise between the customer at location $x$ and the depot ($x\leq y$), then the customer will be delivered at the end of the cycle in which the customer at $x$ is also served. Otherwise, $x>y$, the tagged customer will be delivered at the end of the cycle after the cycle in which the customer at $x$ is served. The former of these cases is simple, as the extra time to delivery is then equivalent to $S(x,1)$ and $T(x,1)$ (the extra waiting time of an auxiliary customer at the depot).\\
Finding the first moment of these variables can be done by a similar approach to that in Lemma \ref{lemma:generatedwaiting}.

\begin{figure}[!htbp]
    \centering
    \begin{subfigure}{0.3\textwidth}
                 \begin{tikzpicture}
        \draw[postaction = {decorate, decoration = {markings, mark = at position 0.9 with {\node[draw, red, circle, fill, inner sep = 1mm] (server){};}}}]
        [postaction = {decorate, decoration = {markings, mark = at position 0.75 with {\node[draw, black!50, fill, inner sep = 1mm] (depot){};}}}]
        [postaction = {decorate, decoration = {markings, mark = at position 0.65 with {\node[regular polygon, regular polygon sides = 5,draw, black,  fill, inner sep = 0.7mm] {};}}}]
        [postaction = {decorate, decoration = {markings, mark = at position 0.42 with {\node[regular polygon, regular polygon sides = 5,draw, black,  fill, inner sep = 0.7mm]{};}}}]
        [postaction = {decorate, decoration = {markings, mark = at position 0.35 with {\node[regular polygon, regular polygon sides = 5,draw, black,  fill, inner sep = 0.7mm]{};}}}]
        [postaction = {decorate, decoration = {markings, mark = at position 0.9 with {\node[regular polygon, regular polygon sides = 5,draw,  fill=orange!50, inner sep = 0.7mm]{};}}}]
        [postaction = {decorate, decoration = {markings, mark = at position 0.13 with {\node[regular polygon, regular polygon sides = 5,draw, fill = blue!50,  fill, inner sep = 0.7mm] {};}}}]
        [postaction = {decorate, decoration = {markings, mark = at position 0.55 with {\node[regular polygon, regular polygon sides = 5,draw, fill = blue!50,  fill, inner sep = 0.7mm] {};}}}]
        [postaction = {decorate, decoration = {markings, mark = at position 0.65 with {\node[regular polygon, regular polygon sides = 5,draw, fill = blue!50,  fill, inner sep = 0.7mm] {};}}}]
        [postaction = {decorate, decoration = {markings, mark = at position 0.3 with {\node[regular polygon, regular polygon sides = 5,draw, fill=green!50, inner sep = 0.7mm] {};}}}]
        [postaction = {decorate, decoration = {markings,  mark = at position 0.5 with {\arrowreversed[line width = 0.7mm]{stealth}}}}]
        [postaction = {decorate, decoration = {markings,  mark = at position 1 with {\arrowreversed[line width = 0.7mm]{stealth}}}}]
        (0,0) circle (2);
        \node [anchor = north west, yshift = 0cm] at (server.east) {Server};
        \node [anchor = north] at (depot.south) {Depot};
        \end{tikzpicture}
    \end{subfigure}
        \begin{subfigure}{0.3\textwidth}
         \begin{tikzpicture}
        \draw[postaction = {decorate, decoration = {markings, mark = at position 0.65 with {\node[draw, red, circle, fill, inner sep = 1mm] (server){};}}}]
        [postaction = {decorate, decoration = {markings, mark = at position 0.75 with {\node[draw, black!50, fill, inner sep = 1mm] (depot){};}}}]
        [postaction = {decorate, decoration = {markings, mark = at position 0.65 with {\node[regular polygon, regular polygon sides = 5,draw, black,  fill, inner sep = 0.7mm] {};}}}]
        [postaction = {decorate, decoration = {markings, mark = at position 0.42 with {\node[regular polygon, regular polygon sides = 5,draw, black,  fill, inner sep = 0.7mm]{};}}}]
        [postaction = {decorate, decoration = {markings, mark = at position 0.35 with {\node[regular polygon, regular polygon sides = 5,draw, black,  fill, inner sep = 0.7mm]{};}}}]
        [postaction = {decorate, decoration = {markings, mark = at position 0.13 with {\node[regular polygon, regular polygon sides = 5,draw, fill = blue!50,  fill, inner sep = 0.7mm] {};}}}]
        [postaction = {decorate, decoration = {markings, mark = at position 0.55 with {\node[regular polygon, regular polygon sides = 5,draw, fill = blue!50,  fill, inner sep = 0.7mm] {};}}}]
        [postaction = {decorate, decoration = {markings, mark = at position 0.65 with {\node[regular polygon, regular polygon sides = 5,draw, fill = blue!50,  fill, inner sep = 0.7mm] {};}}}]
        [postaction = {decorate, decoration = {markings, mark = at position 0.85 with {\node[regular polygon, regular polygon sides = 5,draw, fill = purple!50,  fill, inner sep = 0.7mm] {};}}}]
        [postaction = {decorate, decoration = {markings, mark = at position 0.7 with {\node[regular polygon, regular polygon sides = 5,draw, fill = purple!50,  fill, inner sep = 0.7mm] {};}}}]
        [postaction = {decorate, decoration = {markings, mark = at position 0.2 with {\node[regular polygon, regular polygon sides = 5,draw, fill = purple!50,  fill, inner sep = 0.7mm] {};}}}]        
        [postaction = {decorate, decoration = {markings, mark = at position 0.3 with {\node[regular polygon, regular polygon sides = 5,draw, fill=green!50, inner sep = 0.7mm] {};}}}]
        [postaction = {decorate, decoration = {markings,  mark = at position 0.5 with {\arrowreversed[line width = 0.7mm]{stealth}}}}]
        [postaction = {decorate, decoration = {markings,  mark = at position 1 with {\arrowreversed[line width = 0.7mm]{stealth}}}}]
        (0,0) circle (2);
        \node [anchor = north east, yshift = -0.0cm] at (server.west) {Server};
        \node [anchor = north] at (depot.south) {Depot};
        \end{tikzpicture}
    \end{subfigure}
    \hspace{0.04\textwidth}
        \begin{subfigure}{0.3\textwidth}
            \begin{tikzpicture}
            \useasboundingbox (-3,-3.5) rectangle (3,0);
            \node[regular polygon, regular polygon sides = 5, black, draw, fill = orange!50, inner sep = 2mm] (G1) at (0,0){};
            \node[regular polygon, regular polygon sides = 5, black, draw, fill = blue!50, inner sep = 2mm] (G2_1) at (-1.5,-1.5){};
            \node[regular polygon, regular polygon sides = 5, black, draw, fill = blue!50, inner sep = 2mm] (G2_2) at (1.5,-1.5){};
            -1.5){};
            \node[regular polygon, regular polygon sides = 5, black, draw, fill = blue!50, inner sep = 2mm] (G2_3) at (0,-1.5){};
            \node[regular polygon, regular polygon sides = 5, black, draw, fill = purple!50, inner sep = 2mm] (G3_1) at (-2,-3){};
            \node[regular polygon, regular polygon sides = 5, black, draw, fill = purple!50, inner sep = 2mm] (G3_2) at (-1,-3){};
            \path[draw] (G1) -- (G2_1);
            \path[draw] (G1) -- (G2_2);
            \path[draw] (G1) -- (G2_3);
            \path[draw] (G2_1) -- (G3_1);
            \path[draw] (G2_1) -- (G3_2);
            \end{tikzpicture}
    \end{subfigure}
    \caption{Illustration of the time to delivery of a tagged customer (green) that is generated by a service (of the orange customer) and the corresponding (trimmed) branching process. During the service of the orange customer, blue customers arrive, of which all add to the time to the delivery of the green customer. During the service of the first blue customer, the red customers arrive, of which only the last one will not be served before the delivery of the tagged customer.}
    \label{fig:deliverBranching}
\end{figure}
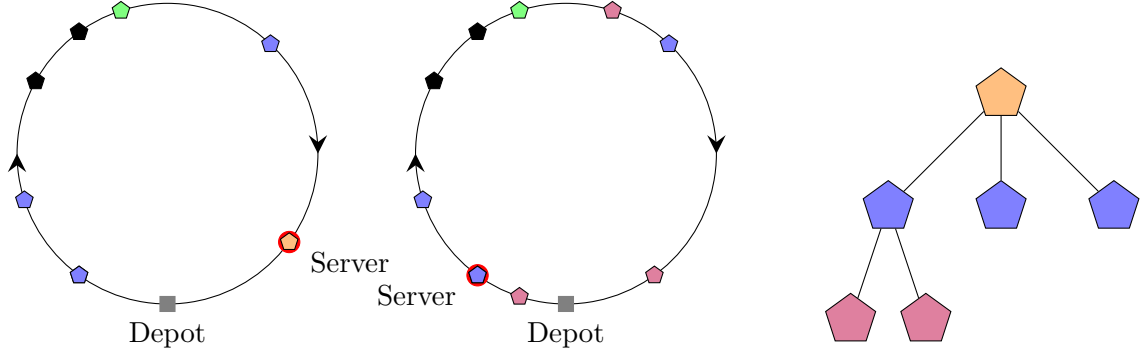

\begin{applemma}
    The extra time to delivery of a customer at location $y$ generated by a service at location $x$ satisfies:
    \begin{align}
    \label{eq:generateddeliver}
        \E[S_D(x,y)] =\begin{dcases}
            \E[B]\exp\left( \rho{\int_{\xi = x}^1}^*\pi(\xi)\dx \xi\right) &\text{if: } x\leq y \\
            \begin{aligned}[c]\E[B]\bigg[&\exp\left(\rho + \rho{\int_{\xi = x}^1}^*\pi(\xi)\dx \xi\right) \\
            &- \rho{\int_{\omega = x}^1}^*\pi(\omega)\dx \omega \exp\left( \rho{\int_{\xi = x}^1}^*\pi(\xi)\dx \xi\right) \bigg]
            \end{aligned}&\text{if: } x> y. 
        \end{dcases} 
    \end{align}
    The average extra deliver time of a customer at location $y$ generated by the travelling of the server, starting at location $x$, is given by:
    \begin{align}
        \E[T_D(x,y)] =\begin{dcases}
            \alpha\int_{u=x}^1 \exp\left(\rho\int_{\xi = u}^1 \pi(\xi)\dx \xi \right)\dx u  &\text{if: } x\leq y\\
            \begin{aligned}[c]
               \alpha&\int_{u=x}^1 \bigg\{ \exp\left(\rho + \rho\int_{\xi = u}^1 \pi(\xi)\dx \xi \right)\\
               &\hspace{1.5cm}-\rho\int_{\omega = u}^1\pi(\omega)\dx \omega \exp\left(\rho\int_{\xi = u}^1 \pi(\xi)\dx \xi \right)\bigg\}\dx u \\
              & + \alpha\int_{u=0}^1 \exp\left(\rho\int_{\xi = u}^1\pi(\xi)\dx \xi\right)\dx u
            \end{aligned}
              &\text{if: } x> y. 
        \end{dcases} 
    \end{align}
\end{applemma}
\begin{proof}
    The case $x \leq y$ follows immediately from the fact that then $S_D(x,y) = S(x,1)$, as the server delivers the tagged customer in the same cycle as the customer at $y$. For the case $x>y$, we condition over the arrivals that occur during the service of the customer at location $x$. Each customer arriving during this service and clockwise between the customer at $x$ and the depot, say location $z \in [x,1]$, generates an extra time of $S_D(z,y)$. When a customer instead arrives at location $z \in [0,x]$, they generate an extra time $S(z,1)$, as the delivery instant of the customer at location $y$ is at the end of the cycle in which the customer at location $z$ is also served. Therefore:
    \begin{align*}
        \E[S_D(x,y)] &= \E[B] + \rho{\int_{z =x}^1} \pi(z)\E[S_D(z,y)]\dx z+ \rho{\int_{z = 0}^x} \pi(z)\E[S(z,1)]\dx z\\
        &= \E[B] + \rho{\int_{z =x}^1} \pi(z)\E[S_D(z,y)]\dx z + \rho{\int_{z =0}^x} \pi(z)\E[B]\exp\left(\rho{\int_{\xi = z}^1}\pi(\xi)\dx \xi\right)\dx z\\
        &= \E[B]\left(1 + \exp(\rho) - \exp\left(\rho \int_{\xi = x}^1\pi(\xi)\dx \xi\right) \right) + \rho{\int_{z = x}^1} \pi(z)\E[S_D(z,y)]\dx z.
    \end{align*}
    It now suffices to prove that \eqref{eq:generateddeliver} solves this integral equation, with the additional condition that $\E[S_D(1,y)] = \E[B]\exp(\rho)$. This condition follows from the fact that $S_D(1,y)$ equals $S(0,1)$. We substitute the proposed expression into the right-hand side and use the substitution $\zeta = \int_{\xi = z}^1\pi(\xi)\dx \xi$:
    \begin{align*}
        \E[S_D(x,y)] &= \E[B]\left(1 + \exp(\rho) - \exp\left(\rho \int_{\xi = x}^1\pi(\xi)\dx \xi\right) \right) \\
        &\quad + \E[B]{\int_{z =x}^1} \rho\pi(z)\bigg\{\exp\left(\rho +\rho\int_{\xi = z}^1\pi(\xi)\dx \xi\right) - \rho\int_{\omega = z}^1\pi(\omega)\dx \omega\exp\left(\rho\int_{\xi = z}^1\pi(\xi)\dx \xi \right)\bigg\}\dx z\\
        &=\E[B]\left(1 + \exp(\rho) - \exp\left(\rho \int_{\xi = x}^1\pi(\xi)\dx \xi\right) \right) \\
        &\quad + \E[B]{\int_{\zeta = 0}^{\int_{\xi = x}^1\pi(\xi)\dx \xi}} \rho\Big\{\exp(\rho + \rho \zeta) - \zeta \exp(\rho \zeta)\Big\}\dx \zeta.
    \end{align*}
    Applying partial integration to the last term, shows:
    \begin{align*}
       \E[S_D(x,y)] &= \E[B]\left(1 + \exp(\rho) - \exp\left(\rho \int_{\xi = x}^1\pi(\xi)\dx \xi\right) \right)\\
       &\quad + \E[B]\left\{\exp\left(\rho + \rho \int_{\xi = x}^1\pi(\xi)\dx \xi\right) - \rho\int_{\omega = x}^1\pi(\omega)\dx \omega\exp\left(\rho\int_{\xi = x}^1\pi(\xi)\dx \xi \right) - \exp(\rho)\right\}\\
       &\quad +  \E[B]{\int_{\zeta = 0}^{\int_{\xi = x}^1\pi(\xi)\dx \xi}} \rho\exp(\rho \zeta)\dx \zeta\\
       &= \E[B]\left\{\exp\left(\rho + \rho \int_{\xi = x}^1\pi(\xi)\dx \xi\right) - \rho\int_{\omega = x}^1\pi(\omega)\dx \omega\exp\left(\rho\int_{\xi = x}^1\pi(\xi)\dx \xi \right)\right\},
    \end{align*}
    where the last step follows from the integration of the last integral. Note that this expression is equal to the proposed expression in \eqref{eq:generateddeliver}, and thus \eqref{eq:generateddeliver} solves the integral equation.\\
    For $T_D(x,y)$ we again see that $T_D(x,y) = T(x,1)$ when $x < y$. For $x>y$ we use the same idea as we did for $T(x,y)$. First, we remark that the this extra time consists of two parts: (i) time generated by customers arriving during the travelling from $x$ to 1 and (ii) time generated by customers arriving during the travelling of the server from $0$ to $1$ (an entire cycle), this again equals $T(0,1)$. For part (i) we distinguish between customers being served in the same cycle as the tagged customer and the cycle before. Using that on average $\lambda \E[K]\alpha \dx u$ customers arrive during the travelling over a distance $\dx u$, we now see that: 
    \begin{align*}
        \E[T_D(x,y)] &= \E[T(0,1)] + \alpha d(x,1) + \lambda\E[K]\alpha \int_{u = x}^1\int_{z = u}^1 \pi(z)\E[S_D(z,y)]\dx z\dx u\\
        &\quad + \lambda\E[K]\alpha \int_{u = x}^1\int_{z = 0}^u\pi(z)\E[S(z,1)]\dx z\dx u.
    \end{align*}
    Substituting in the previously derived expression for $S_D$ (after interchanging the third and fourth term) and using the substitution $\zeta = \int_{\xi = z}^1\pi(\xi)\dx \xi$ again now gives:
    \begin{align*}
        \E[T_D(x,y)] &= \alpha\int_{u = 0}^1 \exp\left(\rho \int_{\xi = u}^1 \pi(\xi)\dx \xi\right) \dx u  + \alpha d(x,1) \\
        &\quad+ \alpha \int_{u = x}^1\int_{z = 0}^u \rho \pi(z) \exp\left(\rho \int_{\xi = z}^1\pi(\xi)\dx\xi\right)\dx z \dx u\\
        &\quad +  \alpha \int_{u = x}^1\int_{z = u}^1 \rho \pi(z) \bigg\{\exp\left(\rho + \rho \int_{\xi = z}^1\pi(\xi)\dx\xi\right) \\
        &\hspace{3.7cm}- \rho\int_{\omega = z}^1\pi(\omega)\dx \omega \exp\left(\rho \int_{\xi=z}^1\pi(\xi)\dx\xi\right)\bigg\}\dx z \dx u \\
        &= \alpha\int_{u = 0}^1 \exp\left(\rho \int_{\xi = u}^1 \pi(\xi)\dx \xi\right) \dx u  + \alpha d(x,1) \\
        &\quad +\alpha \int_{u = x}^1\int_{z = \int_{\xi=u}^1\pi(\xi)\dx \xi}^1 \rho \exp\left(\rho \zeta\right)\dx \zeta \dx u\\
        &\quad +  \alpha \int_{u = x}^1\int_{\zeta = 0}^{\int_{\xi=u}^1\pi(\xi)\dx \xi} \rho \bigg\{\exp\left(\rho + \rho \zeta\right) - \rho\zeta \exp\left(\rho\zeta \right)\bigg\}\dx \zeta \dx u.
    \end{align*}
    By partial integration, we obtain:
    \begin{align*}
        \E[T_D(x,y)] &=\alpha\int_{u = 0}^1 \exp\left(\rho \int_{\xi = u}^1 \pi(\xi)\dx \xi\right) \dx u  + \alpha d(x,1) \\
        &\quad + \alpha  \int_{u = x}^1 \left\{\exp(\rho) - \exp\left(\rho \int_{\xi=u}^1\pi(\xi)\dx \xi\right)\right\}\dx u\\
        &\quad + \alpha \int_{u=x}^1 \left\{\exp\left(\rho + \rho \int_{\xi=u}^1\pi(\xi)\dx \xi\right) - \int_{\omega=u}^1\pi(\omega)\dx \omega\exp\left(\rho \int_{\xi=u}^1\pi(\xi)\dx \xi\right) -  \exp(\rho)\right\}\dx u\\
        &\quad +\alpha  \int_{u=x}^1\int_{\zeta = 0}^{\int_{\xi=u}^1\pi(\xi)\dx\xi} \rho \exp\left(\rho\zeta \right)\dx \zeta \dx u\\
        &=\alpha\int_{u = 0}^1 \exp\left(\rho \int_{\xi = u}^1 \pi(\xi)\dx \xi\right) \dx u  + \alpha d(x,1) - \alpha \int_{u = x}^1\exp\left(\rho \int_{\xi=u}^1\pi(\xi)\dx \xi\right)\dx u \\
        &\quad + \alpha \int_{u=x}^1 \left\{\exp\left(\rho + \rho \int_{\xi=u}^1\pi(\xi)\dx \xi\right) - \int_{\omega=u}^1\pi(\omega)\dx \omega\exp\left(\rho \int_{\xi=u}^1\pi(\xi)\dx \xi\right) \right\}\dx u\\
        &\quad + \alpha \int_{u=x}^1\left\{\exp\left(\rho \int_{\xi=u}^1\pi(\xi)\dx \xi\right) - 1 \right\}\dx u.
    \end{align*}
The proof now follows from cancelling the last integral against the second and third terms.
\end{proof}

\begin{remark}
    In line with Remark \ref{rem:SR}, the extra time to delivery generated by a residual service, $S_D^R(x,y)$, satisfies \eqref{eq:generateddeliver} with $\E[B]$ replaced by $\E[B^2]/(2\E[B])$.
\end{remark}

With this, the proof of Proposition \ref{prop:del_pre} is rather straightforward. We first consider a case distinction, related to the moment of the delivery, and then decompose the time to delivery.

\begin{proof}[Proof of Proposition \ref{prop:del_pre}]
Recall that the server is within $[u, u + \dx u]$ with probability $[\rho\pi(u)+1-\rho]\dx u$. Further recall that the furthest, with respect to the server at location $u$, customer in a batch of size $k$ is located within $[x, x+\dx x]$ with probability $\sum_{k=1}^\infty p_k k \pi(x)({\int_{\nu = u}^x}^*\pi(\nu)\dx\nu)^{k-1}$ (one of the $k$ customers must arrive at location $x$, while others must arrive in between $u$ and $x$). Therefore, the expected time to delivery satisfies:
    \begin{align*}
         \E[D] = \int_{u=0}^1[\rho\pi(u)+  1-\rho]\int_{x=0}^1 \sum_{k=1}^\infty &p_k k \pi(x)\bigg({\int_{\nu = u}^x}^* \pi(\nu)\dx \nu\bigg)^{k-1} \\
         &\cdot \E[D\vert S = u, X^B = x, K = k]\dx x \dx u.
    \end{align*}
Here, a crucial case distinction needs to be considered. The batch of customers is delivered at the end of the cycle in which it arrives, only if the furthest customer is located in-between the server and the depot, i.e. $u\leq x$. Otherwise, the moment of delivery occurs in the cycle thereafter.\\
The former case, $u\leq x$, is simpler. As the expected time to delivery now involves less than a cycle of work from the server, the extended definitions of $S(\cdot,\cdot), T(\cdot, \cdot)$ are thus not yet needed. The decomposition of the time to delivery now is given by: 
\begin{enumerate}[label = (\roman*)]
    \item the time to delivery generated by the travelling of the server from location $u$ to $1$, $\E[T(u,1)]$;
    \item the time to delivery generated by the possible residual service time at the arrival instant, with probability $\rho\pi(u)/(\rho\pi(u)+1-\rho)$ the server is working at location $u$, in which case an extra $\E[S^R(u,1)]$ is generated;
    \item the extra time generated by customers arriving in the current batch. The furthest customer generates an extra time $\E[S(x,1)]$. The other customers generate an extra time $\E[S(z,1)]$, when the customer arrives at location $z$. Conditional on the location of the furthest customer, the other customers arrive within $[z, z+ \dx z]$ ($u\leq z < x)$ with probability: $\pi(z)\dx z/({\int_{\nu = u}^x}^*\pi(\nu)\dx\nu)$;
    \item the extra time generated by customers already present at the time of arrival. In expectation $f(z,u)\dx z$ customers are within $[z, z+ \dx z]$, each of which generates a time  $\E[S(z,1)]$.
\end{enumerate}
Therefore, for $u\leq x$, we have:
\begin{align*}
     \E[D\vert S = u, &X^B = x, K = k] \\
         &= \E[T(u,1)] + \frac{\rho\pi(u)}{\rho\pi(u)+1-\rho} \E[S^R(u,1)] + \E[S(x,1)]
        \\
         &\quad + \frac{(k-1)}{{\int_{\nu = u}^x}^* \pi(\nu)\dx \nu}\cdot \int_{z = u}^x \pi(z)\E[S(z,1)\dx z + \int_{z = u}^1 f(z,u)\E[S(z,1)\dx z\\
         &= \alpha\int_{z=u}^1 \exp\left(\rho\int_{\xi = z}^1\pi(\xi)\dx\xi\right)\dx z + \frac{\rho\pi(u)}{\rho\pi(u)+1-\rho}\frac{\E[B^2]}{2\E[B]}\exp\left(\rho\int_{\xi = u}^1\pi(\xi)\dx\xi\right) \\
         &\quad + \E[B]\exp\left(\rho\int_{\xi = x}^1\pi(\xi)\dx\xi\right) + \frac{(k-1)}{{\int_{\nu = u}^x}^* \pi(\nu)\dx \nu}\cdot \int_{z = u}^x \pi(z)\E[B]\exp\left(\rho\int_{\xi = z}^1\pi(\xi)\dx\xi\right)\dx z\\
         &\quad +\int_{z = u}^1 f(z,u)\E[S(z,1)]\dx z.
\end{align*}
The proof of the first case is now finished by evaluating the integral of the fourth term.\\
The composition of $D$ under the second case, $u>x$, is as follows:
\begin{enumerate}[label = (\roman*)]
    \item the time to delivery generated by the travelling of the server from location $u$ to the depot, plus an entire cycle, $\E[T_D(u,x)]$;
    \item the time to delivery generated by the possible residual service time at the arrival instant. Again, the probability that the server is working at location $u$ is given by $\rho\pi(u)/(\rho\pi(u)+1-\rho)$. The extra time generated, then, equals  $\E[S^{R}_D(u,1)]$;
    \item the extra time generated by customers arriving in the same batch. The furthest customer, still, generates an extra time $\E[S(x,1)]$. Customers arriving at location $z$, with $0<z<x$, generate an extra time $\E[S(z,1)]$. Customers arriving at $z$, with $u<z<1$ generate an extra time of $\E[S_D(z,x)]$;
    \item the extra time generated by customers already present at the time of arrival. Here a customer at location $z$ generates an extra time $\E[S_D(z,x)]$, which equals $\E[S(z,1)]$ when $0<z<x$.
\end{enumerate}
Hence, we obtain for $u>x$:
\begin{align}
\label{eq:Proofc1Dcase2}
    \E[&D\vert S = u, X^B = x, K = k]\\
         &= \E[T_D(u,x)] + \frac{\rho \pi(u)}{\rho \pi(u) + 1-\rho}\E[S^{R}_D(u,x)] + \E[S(x,1)]\nonumber\\
         &\quad  + \frac{(k-1)}{{\int_{\nu = u}^x}^* \pi(\nu)\dx \nu}\cdot \int_{z = 0}^x \pi(z)\E[S(z,1)]\dx z +  \frac{(k-1)}{{\int_{\nu = u}^x}^* \pi(\nu)\dx \nu}\cdot \int_{z = u}^1 \pi(z)\E[S_D(z,x)]\dx z \nonumber\\
         &\quad + \int_{z = 0}^u f(z,u)\E[S(z,1)]\dx z+ \int_{z = u}^1 f(z,u)\E[S_D(z,x)]\dx z \nonumber\\
         &= \alpha \int_{z=u}^1 \bigg\{\exp\left(\rho + \rho\int_{\xi = z}^1\pi(\xi)\dx \xi\right) - \rho \int_{\omega = z}^1 \pi(\omega)\dx \omega\exp\left(\rho\int_{\xi = z}^1\pi(\xi)\dx \xi\right)\bigg\}\dx z\nonumber\\
         &\quad + \alpha \int_{z=0}^1\exp\left(\rho\int_{\xi = z}^1\pi(\xi)\dx \xi\right)\dx z\nonumber\\
         &\quad +\frac{\rho \pi(u)}{\rho \pi(u) + 1-\rho}\frac{\E[B^2]}{2\E[B]}\cdot \bigg\{\exp\left(\rho + \rho\int_{\xi = u}^1\pi(\xi)\dx \xi\right) \nonumber\\
         &\hspace{5.3cm}- \rho \int_{\omega = u}^1 \pi(\omega)\dx \omega\exp\left(\rho\int_{\xi = u}^1\pi(\xi)\dx \xi\right)\bigg\}\nonumber\\
         &\quad + \E[B]\exp\left(\rho\int_{\xi = x}^1\pi(\xi)\dx \xi\right) + \frac{k-1}{{\int_{\nu = u}^x}^*\pi(\nu)\dx \nu}\int_{z=0}^x \pi(z)\E[B]\exp\left(\rho\int_{\xi = z}^1\pi(\xi)\dx \xi\right)\dx z\nonumber\\
         &\quad + \frac{k-1}{{\int_{\nu = u}^x}^*\pi(\nu)\dx \nu}\int_{z=u}^1 \pi(z)\E[B]\bigg\{\exp\left(\rho + \rho\int_{\xi = z}^1\pi(\xi)\dx \xi\right) \nonumber\\
         &\hspace{5.3cm}- \rho \int_{\omega = z}^1 \pi(\omega)\dx \omega\exp\left(\rho\int_{\xi = z}^1\pi(\xi)\dx \xi\right)\bigg\}\dx z \nonumber\\
         &\quad + \int_{z = 0}^u \E[B]f(z,u)\exp\bigg(\rho\int_{\xi = z}^1 \pi(\xi)\dx \xi\bigg)\dx z\nonumber\\
         &\quad + \int_{z = u}^1 \E[B]f(z,u)\bigg\{\exp\bigg(\rho + \rho\int_{\xi = z}^1 \pi(\xi)\dx \xi \bigg) \nonumber \\
         &\hspace{3cm}-\rho\int_{\omega=z}^1 \pi(\omega)\dx \omega \exp\left(\rho\int_{\xi = z}^1 \pi(\xi)\dx \xi\right)\bigg\}\dx z\nonumber.
\end{align}
Now focus on the terms containing $(k-1)$, and use the substitution $\zeta = \int_{\xi =z}^1\pi(\xi)\dx \xi$ and partial integration:
\begin{align*}
        &\frac{k-1}{{\int_{\nu = u}^x}^*\pi(\nu)\dx \nu}\int_{z=0}^x \pi(z)\E[B]\exp\left(\rho\int_{\xi = z}^1\pi(\xi)\dx \xi\right)\dx z\\
         &\quad + \frac{k-1}{{\int_{\nu = u}^x}^*\pi(\nu)\dx \nu}\int_{z=u}^1 \pi(z)\E[B]\bigg\{\exp\left(\rho + \rho\int_{\xi = z}^1\pi(\xi)\dx \xi\right) - \rho \int_{\omega = z}^1 \pi(\omega)\dx \omega\exp\left(\rho\int_{\xi = z}^1\pi(\xi)\dx \xi\right)\bigg\}\dx z \\
    &=\frac{k-1}{{\int_{\nu = u}^x}^*\pi(\nu)\dx \nu}\int_{\zeta=\int_{\xi =x}^1\pi(\xi)\dx \xi}^1 \E[B]\exp\left(\rho\zeta \right)\dx \zeta\\
         &\quad + \frac{k-1}{{\int_{\nu = u}^x}^*\pi(\nu)\dx \nu}\int_{\zeta=0}^{\int_{\xi =u}^1\pi(\xi)\dx \xi} \E[B]\bigg\{\exp\left(\rho + \rho\zeta\right) - \rho  \zeta \exp\left(\rho\zeta\right)\bigg\}\dx \zeta \\
    &=\frac{k-1}{{\int_{\nu = u}^x}^*\pi(\nu)\dx \nu}\frac{1}{\lambda \E[K]}\left\{\exp(\rho) - \exp\left(\rho \int_{\xi =x}^1\pi(\xi)\dx \xi\right)\right\}\\
    &\quad + \frac{k-1}{{\int_{\nu = u}^x}^*\pi(\nu)\dx \nu}\frac{1}{\lambda \E[K]}\left\{\exp\left(\rho + \rho\int_{\xi = u}^1\pi(\xi)\dx \xi\right) - \rho \int_{\omega =u}^1\pi(\omega)\dx \omega \exp\left(\rho\int_{\xi = u}^1\pi(\xi)\dx \xi\right)-\exp(\rho)\right\}\\
    &\quad + \frac{k-1}{{\int_{\nu = u}^x}^*\pi(\nu)\dx \nu}\frac{1}{\lambda \E[K]}\int_{\zeta=0}^{\int_{\xi =u}^1\pi(\xi)\dx \xi} \rho\exp(\rho \zeta)\dx \zeta\\
    &= \frac{k-1}{{\int_{\nu = u}^x}^*\pi(\nu)\dx \nu}\frac{1}{\lambda \E[K]}\left\{- \exp\left(\rho \int_{\xi =x}^1\pi(\xi)\dx \xi\right)\right\} \\
    &\quad + \frac{k-1}{{\int_{\nu = u}^x}^*\pi(\nu)\dx \nu}\frac{1}{\lambda \E[K]}\left\{\exp\left(\rho + \rho\int_{\xi = u}^1\pi(\xi)\dx \xi\right) - \rho \int_{\omega =u}^1\pi(\omega)\dx \omega \exp\left(\rho\int_{\xi = u}^1\pi(\xi)\dx \xi\right)\right\}\\
    &\quad + \frac{k-1}{{\int_{\nu = u}^x}^*\pi(\nu)\dx \nu}\frac{1}{\lambda \E[K]}\left\{\exp\left(\rho\int_{\xi = u}^1\pi(\xi)\dx \xi\right) -1\right\}.
\end{align*}
Reordering the terms in this expression and substituting it into \eqref{eq:Proofc1Dcase2} finalizes the proof.
\end{proof}

\section{Numerical algorithm for solving the integral equation}
\label{app:numerical}
We now consider a numerical approach to the algorithm discussed in Section \ref{sec:successive}. We propose a numerical variant of this algorithm, using left Riemann sums as an approximation for the integration in each iteration step. The interval $[0,1)$ is split into $N$ equally sized intervals. We then calculate $g_{n+1}$ as follows for all $i,j \in \{0,1,...,N-1\}$:
\[g_{n+1}(i/N,j/N) = \rho \pi(j/N) \sum_{m=i}^{j-1} \Big[g_n(i/N,m/N) + g_n(j/N, m/N)\Big],\]
where we use the following convention for the sum:
\begin{align*}
\sum_{m = i}^j a_m = \begin{dcases}
a_i + a_{i+1} + ... + a_j &\text{if: } i\leq j\\
a_i + a_{i+1} + ... + a_{N-1} + a_0 + a_1 + a_2 + ... + a_j  &\text{if: } i > j.
\end{dcases}
\end{align*}
The approximation of $g$ can be extended to non-grid arguments by e.g. interpolation.\\
Clearly, this routine incorporates approximations errors in each iteration.
For this algorithm to still (i) converge and (ii) give good approximations of $f^*$, it is important that these errors do not propagate. This is the case under a mild regularity condition, stating that the Riemann approximation of the integral over $\pi$ should be close to $1$.\\
First we prove (i), following similar steps to the proof of convergence in Lemma \ref{lemma:convergence}.

\begin{lemma}
\label{lemma:convergence}
Assume that the left-Riemann sum approximation of the integral over $\pi$ satisfies:
\begin{align*}
\frac{1}{N}\sum_{m=0}^{N-1} \pi(m/N) \leq \frac{1}{\rho},
\end{align*}
then the algorithm converges for any arbitrary stopping level $\delta > 0$.
\end{lemma}
\begin{proof}
    The proof of the convergence is quite similar to the convergence proof of the analytical procedure, cf. Lemma \ref{lemma:convergence}. Let $g_{n-1}, g_{n}$ be two arbitrary subsequent iterations of the algorithm, then we have that:
    \begin{equation}
    \label{eq:proof_numroutine_convergence1}
    \begin{aligned}[c]
        \big\vert g_n(i/N,& j/N) - g_{n-1}(i/N, j/N)\big\vert\\ 
        &\leq \frac{\rho}{N}\pi(j/N)\sum_{m = i}^{j-1} \big\vert g_{n-1}(i/N, m/N) - g_{n-2}(i/N, m/N)\big\vert\\\
        &\quad+  \frac{\rho}{N}\pi(j/N)\sum_{m = i}^{j-1}\big\vert g_{n-1}(j/N, m/N) - g_{n-2}(j/N, m/N)\big\vert,
    \end{aligned}
    \end{equation}
    where we can bound the summations by using the supremum norm over the differences:
    \begin{equation}
    \label{eq:proof_numroutine_convergence2}
    \begin{aligned}[c]
        \big\vert g_n(i/N, j/N) - g_{n-1}(i/N, j/N)\big\vert 
        &\leq 2\rho\pi(j/N)\supnorm{g_{n-1}-g_{n-2}}.
    \end{aligned}
    \end{equation}
    One can now use \eqref{eq:proof_numroutine_convergence2} for $g_{n-1}(i/N, m/N)-g_{n-2}(i/N, m/N)$ and $g_{n-1}(j/N, m/N)-g_{n-2}(j/N, m/N)$ in \eqref{eq:proof_numroutine_convergence1}. This gives:
    \begin{equation}
        \label{eq:proof_numroutine_convergence3}
        \begin{aligned}
            \big\vert g_n(i/N, j/N) - g_{n-1}(i/N, j/N)\big\vert \\
            &\leq {4\rho^2}\pi(j/N)\supnorm{g_{n-2}-g_{n-3}}\cdot\frac{1}{N}\sum_{m = i}^{j-1} \pi(m/N).
        \end{aligned}
    \end{equation}
    We now substitute \eqref{eq:proof_numroutine_convergence3} into \eqref{eq:proof_numroutine_convergence1}, resulting in:
    \begin{equation}
        \label{eq:proof_numroutine_convergence4}
        \begin{aligned}
            \big\vert g_n(i/N, &j/N) - g_{n-1}(i/N, j/N)\big\vert \\
            &\leq \frac{4\rho^3}{N^2}\pi(j/N)\supnorm{g_{n-3}-g_{n-4}} \\
            &\quad\quad\cdot \bigg[\sum_{m = i}^{j-1}\pi(m/N)\sum_{l=i}^{m-1} \pi(l/N) + \sum_{m = i}^{j-1}\pi(m/N)\sum_{l=j}^{m-1} \pi(l/N)\bigg].
        \end{aligned}
    \end{equation}
    This can be further simplified by interchanging the order of summation, noting that:
    \begin{align}
        \sum_{m = i}^{j-1}\pi(m/N)&\sum_{l=i}^{m-1} \pi(l/N) + \sum_{m = i}^{j-1}\pi(m/N)\sum_{l=j}^{m-1} \pi(l/N) \nonumber\\
        &= \sum_{m = i}^{j-1}\pi(m/N)\sum_{l=i}^{m} \pi(l/N) + \sum_{m = i}^{j-1}\pi(m/N)\sum_{l=j}^{m-1} \pi(l/N) - \sum_{m = i}^{j-1}\pi(m/N)^2 \nonumber\\
        &\leq \sum_{l = i}^{j-1}\pi(l/N)\sum_{m=l}^{j-1} \pi(m/N) + \sum_{m = i}^{j-1}\pi(m/N)\sum_{l=j}^{m-1} \pi(l/N) \nonumber.
    \intertext{Note that the inner sums together add up to the sum over all $\pi(m/N)$ values from $m = 0$ to $N-1$. Therefore:}
    \label{eq:doublesumeq}
         \sum_{m = i}^{j-1}\pi(m/N)&\sum_{l=i}^{m-1} \pi(l/N) + \sum_{m = i}^{j-1}\pi(m/N)\sum_{l=j}^{m-1} \pi(l/N) = \sum_{m=0}^{N-1} \pi(m/N)\cdot\sum_{l = i}^{j-1}\pi(l/N).
    \end{align}
    Turning back to \eqref{eq:proof_numroutine_convergence4}, we see that:
    \begin{align*}
        \big\vert g_n(i/N, &j/N) - g_{n-1}(i/N, j/N)\big\vert \\
        &\leq \frac{4\rho^3}{N}\pi(j/N)\sum_{m=0}^{N-1} \pi(m/N)\supnorm{g_{n-3}-g_{n-4}}\cdot\frac{1}{N}\sum_{l = i}^{j-1}\pi(l/N).
    \end{align*}
    We can repeat these manipulations until we obtain the following for $n\geq 4$.
    \begin{align}
    \label{eq:laststepconvergencenumerical}
         \big\vert g_n(i/N, &j/N) - g_{n-1}(i/N, j/N)\big\vert \nonumber\\
         &\leq 4\rho\pi(j/N)\bigg(\rho \cdot \frac{1}{N}\sum_{m=0}^{N-1} \pi(m/N)\bigg)^{n-2}\supnorm{g_{1}-g_{0}}.
    \end{align}
    Due to the regularity condition, we have $\rho/N\cdot \sum_{m=0}^{N-1} \pi(m/N) < 1$. Thus, for $n$ sufficiently large,  $\big\vert g_n(i/N, j/N) - g_{n-1}(i/N, j/N)\big\vert \leq \delta$.
\end{proof} 

\begin{remark}
The regularity condition for this Proposition might appear more strict than it in reality is. Many examples exist in which the condition is not met, however one can overcome this issue by shifting the arrival location distribution by a small amount ($<1/N$) and solve the integral equation under this altered arrival location distribution. One shifts the result by the same amount in the other direction again. Doing so, one can improve the Riemann approximation of $\pi$.
\end{remark}

\subsection{Analysis of the algorithm}
\label{sec:5errorproof}
We now turn to the resulting error in the final approximate of $f^*$. Throughout this section, we assume that the regularity condition is met. This derivation consists of several steps. First, we bound the error of the final iterate compared to the solution of the \emph{numerical} integral equation, that is the solution to:
\begin{align}
\label{eq:intergral_numerical}
    \hat{g}(i/N, j/N) = \rho\pi(j/N)\sum_{m=i}^{j-1}\Big[\hat{g}(i/N, m/N) + \hat{g}(j/N, m/N)\Big] + b(i/N,j/N). 
\end{align}
We will write the shorthand $\hat{g}^*$ for the solution to this equation. Using a generalisation of \eqref{eq:laststepconvergencenumerical}, we are able to prove that the final iterate satisfies Lemma \ref{lemma:errorproof1}.
\begin{lemma}
    \label{lemma:errorproof1}
    The final iterate of the numerical algorithm, $g_n$, with stopping criterion $\delta > 0$, satisfies:
    \begin{align}
    \label{eq:}
        \big\vert\hat{g}^*(i/N, j/N) - g_n(i/N, j/N) \big\vert \leq \frac{4\rho\pi(j/N)\delta}{1 - \rho/N\sum_{m=0}^{N-1}\pi(m/N)}.
    \end{align}
\end{lemma}
\begin{proof}
    A generalisation of \eqref{eq:laststepconvergencenumerical} reads, for any $l > n$:
    \begin{align}
    \label{eq:numproofconv_eq1}
         \big\vert g_l(i/N, &j/N) - g_{l-1}(i/N, j/N)\big\vert \\
         &\leq 4\rho \pi(j/N)\bigg(\rho \cdot \frac{1}{N}\sum_{m=0}^{N-1} \pi(m/N)\bigg)^{l-n-1}\supnorm{g_{n}-g_{n-1}}.\nonumber
    \end{align}
    Also, due to the convergence, we have:
    \begin{align*}
        \hat{g}^*(i/N, j/N) &= g_0^*(i/N, j/N) + \lim_{r\to \infty} \sum_{l=1}^r \big\{g_l(i/N,j/N)-g_{l-1}(i/N, j/N)\big\}\\
        &=g_{n}(i/N, j/N) +  \lim_{r\to \infty} \sum_{l={n+1}}^r \big\{g_l(i/N,j/N)-g_{l-1}(i/N, j/N)\big\}.
    \end{align*}
    Using the bound in \eqref{eq:numproofconv_eq1} shows that:
    \begin{align*}
        \big\vert \hat{g}^*(i/N, &j/N) - g_{n}(i/N, j/N) \big\vert \\
        &\leq  \sum_{l={n+1}}^\infty  \big\vert g_l(i/N, j/N) - g_{l-1}(i/N, j/N)\big\vert\\
        &\leq 4\rho\pi(j/N)\supnorm{ g_{n}-g_{n-1}}\sum_{l={n+1}}^\infty\bigg(\rho \cdot \frac{1}{N}\sum_{m=0}^{N-1} \pi(m/N)\bigg)^{l-n-1}.
    \end{align*}
    The proof is finished by using that $\supnorm{g_{n} - g_{n-1}} \leq \delta$ by means of the stopping criterion and evaluating the sum.
\end{proof}

The remainder of the proof consists of the analysis of the difference between the analytical solution $g^*$ and the {numerical} solution $\hat{g}^*$ to \eqref{eq:intergral_numerical}. The basis of this proof is the following insight:
\begin{align}
\label{eq:errorproof_maineq}
\vert  g^*(i/N,j/N) &- \hat{g}^*(i/N,j/N) \vert \nonumber\\
&\leq \rho\pi(j/N)\sum_{m = i}^{j-1} \int_{u=m/N}^{(m+1)/N}\Big\vert g^*(i/N, u) - \hat{g}^*(i/N, m/N)\Big\vert \dx u \\
&\quad + \rho\pi(j/N)\sum_{m = i}^j \int_{u=m/N}^{(m+1)/N}\Big\vert g^*(j/N, u) - \hat{g}^*(j/N, m/N) \Big\vert\dx u. \nonumber
\end{align}

We can further split the errors as follows:
\begin{align}
\label{eq:splitting}
    \Big\vert g^*(i/N, u) &- \hat{g}^*(i/N, m/N)\Big\vert\\
    &\leq \underbrace{ \big\vert g^*(i/N, u) - g^*(i/N, m/N)\big\vert}_{(i)} + \underbrace{\big\vert g^*(i/N, m/N) - \hat{g}^*(i/N, m/N)\big\vert}_{(ii)},\nonumber
\end{align}
of which we have already found a bound for (ii). The term (i) describes the error in $g^*$ that is caused by using the numerical approximation of the integrals. Lemma \ref{lemma:gproperties} provides a bound for this term.
\begin{lemma}
\label{lemma:gproperties}
    Let $g^*$ be the solution to the integral equation \eqref{eq:fpe_g_general}. Then we have:
    \begin{equation}
    \label{eq:app_boundg}
         g^*(x,y) \leq \supnorm{b}+\frac{2\rho(1+\rho)\pi(y)}{1-\rho}\supnorm{b}.
    \end{equation}
    For any $u,x,y \in [0,1]$ we also have that for $x\leq u$:
    \begin{equation}
    \label{eq:app_bounddiffxg}
    \begin{aligned}[c]
            \big\vert g^*(x,y) - g^*(u,y) \big\vert 
            \leq
            \frac{2\rho\supnorm{b}\pi(y)}{1-\rho}{\int_{v=x}^u}^*\big[1+\frac{2\rho(1+\rho)}{1-\rho}\pi(v)\big] \dx v \\
            \quad +  \sup_{v \in [x,y]}\vert b(x,v)-b(u,v)\vert\cdot \bigg[1+\frac{2\rho\pi(y)}{1-\rho}\bigg].
    \end{aligned}
    \end{equation}
    For any $u,x,y \in [0,1]$ and $y\leq u$ the following holds:
    \begin{align}
    \label{eq:app_bounddiffyg}
        \big\vert g^*(x,y) - g^*(x,u)\big\vert 
        &\leq
        \frac{2\rho(2-\rho)\supnorm{b}\pi(y)}{1-\rho}\int_{\omega = y}^u \Big[1+\frac{2\rho(1+\rho)}{1-\rho}\pi(\omega)\Big]\dx \omega \nonumber\\
        &\quad + \frac{\rho(1+\rho)}{1-\rho}\pi(y)\sup_{v \in [0,1]}\vert b(u,v)-b(y,v)\vert \\
        &\quad + \frac{2\rho(1+\rho)}{1-\rho}\supnorm{b}\big\vert\pi(y)-\pi(u)\big\vert 
        + \vert b(x,y)-b(x,u)\vert.\nonumber
    \end{align}
\end{lemma}

\begin{proof}
We first remark the following:
\begin{align}
\label{eq:supnormgeq1}
\vert g^*(x,y) \vert &\leq \rho\pi(y) {\int_{u = x}^y}^* \big[g^*(x,u) + g^*(y,u)\big] \dx u + b(x,y)\\
\label{eq:supnormgeq2}
&\leq 2\rho\pi(y)\supnorm{g^*} + b(x,y).
\end{align}
We now use \eqref{eq:supnormgeq2} for $g^*(x,u), g^*(y,u)$ in \eqref{eq:supnormgeq1}, showing:
\begin{align}
\label{eq:supnormgeq3}
\vert g^*(x,y) \vert &\leq 4\rho^2\pi(y) {\int_{u = x}^y}^* \pi(u) \dx u + \supnorm{b}\bigg\{1+2\rho\pi(y) \bigg\}.
\end{align}
Substitution of \eqref{eq:supnormgeq3} into \eqref{eq:supnormgeq1} gives:
\begin{align*}
\vert g^*(x,y) \vert &\leq 4\rho^3\pi(y) \supnorm{g^*} {\int_{u = x}^y}^* \pi(u){\int_{v=x}^u}^* \pi(v)\dx v\dx u \\
&\quad + 4\rho^3\pi(y)\supnorm{g^*} {\int_{u = x}^y}^* \pi(u){\int_{v=y}^u}^* \pi(v)\dx v\dx u  \\
&\quad + \supnorm{b}\bigg\{1+2\rho\pi(y) + 4\rho^2\pi(y){\int_{u = x}^y}^*\pi(u)\dx u\bigg\}.
\end{align*}
We interchange the order of integration of the first two integrals, this shows:
\begin{align*}
\vert g^*(x,y) \vert &\leq  
4\rho^3\pi(y)\supnorm{g^*} {\int_{v = x}^y}^* \pi(v){\int_{u=v}^y}^* \pi(u)\dx u\dx v \\
&\quad + 4\rho^3\pi(y)\supnorm{g^*} {\int_{u = x}^y}^* \pi(u){\int_{v=y}^u}^* \pi(v)\dx v\dx u  \\
&\quad + \supnorm{b}\bigg\{1+2\rho\pi(y) + 4\rho^2\pi(y){\int_{u = x}^y}^*\pi(u)\dx u\bigg\}\\
&=4\rho^3\supnorm{g^*}\pi(y) {\int_{u = x}^y}^* \pi(u)\dx u  +\supnorm{b}\bigg\{1+2\rho\pi(y) + 4\rho^2\pi(y){\int_{u = x}^y}^*\pi(u)\dx u\bigg\}.
\end{align*}
We can now substitute this inequality for $g^*(x,u), g^*(y,u)$ into \eqref{eq:supnormgeq1}, and repeat this process to ultimately obtain \eqref{eq:app_boundg} (also applying this trick on the $b$-terms):
\begin{align*}
\vert g^*(x,y) \vert &\leq \lim_{n\to\infty} \bigg\{4\rho^n \supnorm{g^*}{\int_{u = x}^y}^* \pi(u)\dx u + \supnorm{b} + 2\rho\pi(y)\supnorm{b} + 4\rho\pi(y)\supnorm{b}\sum_{i=1}^n\rho^{i}\bigg\}\\
&= \supnorm{b} + 2\rho\pi(y)\supnorm{b} + \frac{4\rho^2\pi(y)\supnorm{b}}{1-\rho} = \supnorm{b}+\frac{2\rho(1+\rho)\pi(y)}{1-\rho}\supnorm{b}.
\end{align*}
For \eqref{eq:app_bounddiffxg} we use a similar approach. Starting from the integral equation, we have:
\begin{align}
\label{eq:propg_second_eq1}
    \vert g^*(x,y) - g^*(u,y)\vert&\leq \rho\pi(y)\bigg\vert {\int_{v=x}^y}^*\big[g^*(x,v) + g^*(y,v)\big]\dx v - {\int_{v=u}^y}^*\big[g^*(u,v) + g^*(y,v)\big]\dx v\bigg\vert    \nonumber\\
    &\quad + \vert b(x,y)-b(u,y)\vert    \nonumber\\
    &\leq \rho\pi(y) {\int_{v=0}^1}^*\big\vert g^*(x,v) - g^*(u,v)\big\vert\dx v \\
    &\quad + \rho\pi(y) {\int_{v=x}^u}^*\big[g^*(x,v) + g^*(y,v)\big]\dx v  + \vert b(x,y)-b(u,y)\vert. \nonumber
\end{align}
We bound the integrand of the first integral by the previously derived bound. For the first integral, we use that: $\vert g^*(x,v) - g^*(u,v)\vert \leq \sup_{v \in [x,y]}\vert g^*(x,v) - g^*(u,v)\vert$.
\begin{align*}
    \vert g^*(x,y) - g^*(u,y)\vert&\leq   \rho\pi(y)\sup_{v \in [x,y]}\vert g^*(x,v) - g^*(u,v)\vert  \\
    &\quad+2\rho\pi(y)\supnorm{b}{\int_{v=x}^u}^*\big[1+\frac{2\rho(1+\rho)}{1-\rho}\pi(v)\big] \dx v \\
    &\quad+ \sup_{v \in [x,y]}\vert b(x,v)-b(u,v)\vert.
\end{align*}
We apply this inequality to \eqref{eq:propg_second_eq1} and repeat this for the newly derived bounds. While doing so we carefully bound the deviation in $b$ by its supremum over $v \in [x,y]$:
\begin{align*}
    \vert g^*(x,y) - g^*(u,y)\vert&\leq \rho^2\pi(y)\sup_{v \in [x,y]}\vert g^*(x,v) - g^*(u,v)\vert{\int_{v=u}^y}^*\pi(v)\dx v \\
    &\quad + 2(\rho + \rho^2)\pi(y)\supnorm{b}{\int_{v=x}^u}^*\big[1+\frac{2\rho(1+\rho)}{1-\rho}\pi(v)\big] \dx v\\
    &\quad + \sup_{v \in [x,y]}\vert b(x,v)-b(u,v)\vert + 2\rho\pi(y)\sup_{v \in [x,y]}\vert b(x,v)-b(u,v)\vert,
\end{align*}
and thus for any $n$,
\begin{align*}
    \vert g^*(x,&y) - g^*(u,y)\vert\\
    &\leq \rho^n\pi(y)\sup_{v \in [x,y]}\vert g^*(x,v) - g^*(u,v)\vert{\int_{v=u}^y}^*\pi(v)\dx v \\
    &\quad + 2\sum_{i=1}^n\rho^i\pi(y)\supnorm{b}{\int_{v=x}^u}^*\big[1+\frac{2\rho(1+\rho)}{1-\rho}\pi(v)\big] \dx v\\
    &\quad + \sup_{v \in [x,y]}\vert b(x,v)-b(u,v)\vert + 2\sum_{i=1}^{n-1}\rho^i \pi(y)\sup_{v \in [x,y]}\vert b(x,v)-b(u,v)\vert\\
    &\overset{n\to\infty}{\to} \frac{2\rho\supnorm{b}\pi(y)}{1-\rho}{\int_{v=x}^u}^*\big[1+\frac{2\rho(1+\rho)}{1-\rho}\pi(v)\big] \dx v \\
    &\quad +  \sup_{v \in [x,y]}\vert b(x,v)-b(u,v)\vert\cdot \bigg[1+\frac{2\rho\pi(y)}{1-\rho}\bigg].
\end{align*}
Thus \eqref{eq:app_bounddiffxg} is proven.
\eqref{eq:app_bounddiffyg} can now be proven using the previous two results. 
    \begin{align*}
        \big\vert g^*(x,&y) - g^*(x,u)\big\vert \\
        &\leq \rho\bigg\vert \pi(y){\int_{v=x}^y}^*\big[g^*(x,v) + g^*(y,v)\big]\dx v - \pi(u){\int_{v=x}^u}^*\big[g^*(u,v) + g^*(y,v)\big]\dx v\bigg\vert    \\
        &\quad + \vert b(x,y)-b(x,u)\vert   \\
        &\leq \rho\pi(y) {\int_{v=y}^u}^*\big[g^*(x,v) + g^*(y,v)\big]\dx v  + \rho\pi(y) {\int_{v=x}^y}^*\big\vert g^*(u,v) - g^*(y,v)\big\vert\dx v   \\
        &\quad + \rho\big\vert \pi(y) - \pi(u)\big\vert{\int_{v=x}^u}^*\big[g^*(x,v) + g^*(y,v)\big]\dx v + \vert b(x,y)-b(x,u)\vert. 
    \end{align*}
    We apply both \eqref{eq:app_boundg} and \eqref{eq:app_bounddiffxg} to conclude:
    \begin{align*}
        \big\vert g^*(x,y) - g^*(x,u)\big\vert 
        &\leq
        2\rho\supnorm{b}\pi(y)\cdot\int_{\omega = y}^u \Big[1+\frac{2\rho(1+\rho)}{1-\rho}\pi(\omega)\Big]\dx \omega\\
        &\quad +\frac{2\rho\supnorm{b}\pi(y)}{1-\rho}\int_{\omega = y}^u \Big[1+\frac{2\rho(1+\rho)}{1-\rho}\pi(\omega)\Big]\dx \omega \\
        &\quad + \rho\pi(y)\sup_{v \in [0,1]}\vert b(u,v)-b(y,v)\vert\cdot \bigg[1+\frac{2\rho}{1-\rho}\bigg] \\
        &\quad + 2\rho\big\vert\pi(y)-\pi(u)\big\vert \supnorm{b}\bigg[1+\frac{2\rho}{1-\rho}\bigg]+  \vert b(x,y)-b(x,u)\vert\\
        &\leq \frac{2\rho(2-\rho)\supnorm{b}\pi(y)}{1-\rho}\int_{\omega = y}^u \Big[1+\frac{2\rho(1+\rho)}{1-\rho}\pi(\omega)\Big]\dx \omega\\
        &\quad + \frac{\rho(1+\rho)}{1-\rho}\pi(y)\sup_{v \in [0,1]}\vert b(u,v)-b(y,v)\vert \\
        &\quad + \frac{2\rho(1+\rho)}{1-\rho}\supnorm{b}\big\vert\pi(y)-\pi(u)\big\vert 
        + \vert b(x,y)-b(x,u)\vert.
    \end{align*}
\end{proof}

With this, we have all ingredients for the derivation of the error bound. We can now prove the following:

\begin{proposition}
The final iterate $g_n$ of the numerical algorithm with stopping criterion $\delta>0$ satisfies the following error bound:
\begin{align*}
    \supnorm{g^* &- g_n} \\
    &\leq \frac{4\rho\pi(j/N)\delta}{1 - \rho/N\sum_{m=0}^{N-1}\pi(m/N)} + \frac{2\rho\pi(j/N)\epsilon \cdot \rho/N\sum_{m=0}^{N-1}\pi(m/N)}{1 - \rho/N\sum_{m=0}^{N-1}\pi(m/N)} + \pi(j/N)\epsilon
\end{align*}
where:
\begin{align*}
\epsilon &= \frac{4\rho^2(2-\rho)\supnorm{b}}{1-\rho}\frac{1}{N}\sum_{m=0}^{N-1} \pi(m/N)\int_{\omega = m/N}^{(m+1)/N} \Big[1+\frac{2\rho(1+\rho)}{1-\rho}\pi(\omega)\Big]\dx \omega\\
&\quad + \frac{2\rho^2(1+\rho)}{1-\rho}\cdot \max_{m\in \{0,1,...,N-1\}}\sup_{u\in [m/N, (m+1)/N]}\sup_v \big\vert b(m/N , v) - b(u,v)\big\vert\frac{1}{N}\sum_{m=0}^{N-1} \pi(m/N) \\
&\quad+ \frac{4\rho^2(1+\rho)}{1-\rho}\supnorm{b}\frac{1}{N}\sum_{m=0}^{N-1}\int_{u=m/N}^{(m+1)/N}\big\vert\pi(m/N)-\pi(u)\big\vert\dx u\\
&\quad + 2\rho \cdot  \max_{m\in \{0,1,...,N-1\}}\sup_{u\in [m/N, (m+1)/N]}\sup_v \big\vert b(v , m/N) - b(v,u)\big\vert.
\end{align*}
\label{prop:error}    
\end{proposition}
\begin{proof}
We apply Lemma \ref{lemma:gproperties} to \eqref{eq:splitting}:
\begin{align*}
    \Big\vert g^*(i/N, u) &- \hat{g}^*(i/N, m/N)\Big\vert \\
    &\leq \frac{2\rho(2-\rho)\supnorm{b}}{1-\rho}\pi(m/N)\int_{\omega = m/N}^u \Big[1+\frac{2\rho(1+\rho)}{1-\rho}\pi(\omega)\Big]\dx \omega\\
    &\quad + \frac{\rho(1+\rho)}{1-\rho}\pi(m/N)\sup_{v \in [0,1]}\vert b(u,v)-b(m/N,v)\vert \\
    &\quad + \frac{2\rho(1+\rho)}{1-\rho}\supnorm{b}\big\vert\pi(m/N)-\pi(u)\big\vert + \vert b(i/N,m/N)-b(i/N,u)\vert \\
    &\quad +  \Big\vert g^*(i/N, m/N) - \hat{g}^*(i/N, m/N)\Big\vert.
\end{align*}
Let $\epsilon_b^x(N) := \max_{m\in \{0,1,...,N-1\}}\sup_{u\in [m/N, (m+1)/N]}\sup_v \big\vert b(m/N , v) - b(u,v)\big\vert$ denote the maximal Riemann error of $b$ in the $x$-coordinate. Define $\epsilon_b^y(N)$ analogously. Moreover, the first integral on the right-hand side can be bounded by the same integral over the range from $m/N$ to $(m+1)/N$. Then \eqref{eq:errorproof_maineq} can be bounded as:
\begin{align*}
    &\big\vert g^*(i/N, j/N)-\hat{g}^*(i/N, j/N)\big\vert \\
    &\leq \frac{4\rho^2(2-\rho)\supnorm{b}}{1-\rho}\pi(j/N)\frac{1}{N}\sum_{m=i}^{j-1} \pi(m/N)\int_{\omega = m/N}^{(m+1)/N} \Big[1+\frac{2\rho(1+\rho)}{1-\rho}\pi(\omega)\Big]\dx \omega\\
    &\quad + \frac{2\rho^2(1+\rho)}{1-\rho}\pi(j/N)\epsilon_b^x(N) \frac{1}{N}\sum_{m=i}^{j-1} \pi(m/N) \\
    &\quad+ \frac{4\rho^2(1+\rho)}{1-\rho}\pi(j/N)\supnorm{b}\frac{1}{N}\sum_{m=i}^{j-1}\int_{u=m/N}^{(m+1)/N}\big\vert\pi(m/N)-\pi(u)\big\vert\dx u
    + 2\rho \pi(j/N) \epsilon_b^y(N)\\
    &\quad+ \rho\pi(j/N)\frac{1}{N}\sum_{m=i}^{j-1} \big\vert g^*(i/N, m/N)-\hat{g}^*(i/N, m/N)\big\vert \\
    &\quad +\rho\pi(j/N)\frac{1}{N}\sum_{m=i}^{j-1} \big\vert g^*(j/N, m/N)-\hat{g}^*(j/N, m/N)\big\vert.
    \end{align*}
Let $\epsilon$ be defined as:
\begin{align*}
    \epsilon &:= \frac{4\rho^2(2-\rho)\supnorm{b}}{1-\rho}\frac{1}{N}\sum_{m=0}^{N-1} \pi(m/N)\int_{\omega = m/N}^{(m+1)/N} \Big[1+\frac{2\rho(1+\rho)}{1-\rho}\pi(\omega)\Big]\dx \omega\\
    &\quad + \frac{2\rho^2(1+\rho)}{1-\rho}\epsilon_b^x(N) \frac{1}{N}\sum_{m=0}^{N-1} \pi(m/N) \\
    &\quad+ \frac{4\rho^2(1+\rho)}{1-\rho}\supnorm{b}\frac{1}{N}\sum_{m=0}^{N-1}\int_{u=m/N}^{(m+1)/N}\big\vert\pi(m/N)-\pi(u)\big\vert\dx u
    + 2\rho \epsilon_b^y(N),
\end{align*}
such that:
\begin{align*}
    \big\vert g^*(i/N, &j/N)-\hat{g}^*(i/N, j/N)\big\vert\\
    &\leq \pi(j/N)\epsilon + \frac{\rho}{N}\pi(j/N)\sum_{m=i}^{j-1} \big\vert g^*(i/N, m/N) - \hat{g}^*(i/N, m/N)\big\vert\\
    &\quad +\frac{\rho}{N}\pi(j/N)\sum_{m=i}^{j-1}\big\vert g^*(j/N, m/N) - \hat{g}^*(j/N, m/N)\big\vert.
\end{align*}
We use this inequality in itself, for the arguments $(i/N, m/N)$ and $(j/N, m/N)$ and further bound the difference between $g^*$ and $\hat{g}^*$ by the supremum norm:
\begin{align*}
    \big\vert g^*(i/N, j/N)-\hat{g}^*(i/N, j/N)\big\vert &\leq \pi(j/N)\epsilon\Big\{1+\frac{2\rho}{N}\sum_{m=i}^{j-1} \pi(m/N)\Big\} \\
    &\quad + \frac{2\rho}{N} \pi(j/N)\supnorm{g^*-\hat{g}^*} \sum_{m=i}^{j-1}\pi(m/N).
\end{align*}
Now we substitute this inequality into the previous inequality:
\begin{align*}
    \big\vert g^*(i/N, &j/N)-\hat{g}^*(i/N, j/N)\big\vert\\
    &\leq \pi(j/N)\epsilon\Big\{1+\frac{2\rho}{N}\sum_{m=i}^{j-1} \pi(m/N)\Big\} \\
    &\quad+ \frac{2\epsilon\rho^2}{N^2}\pi(j/N)\sum_{m=i}^{j-1}\bigg\{ \pi(m/N)\cdot\Big[\sum_{l=i}^{m-1}\pi(l/N)+\sum_{l=j}^{m-1}\pi(l/N)\Big] \bigg\}\\
    &\quad + \frac{2\rho^2}{N^2} \pi(j/N)\supnorm{g^*-\hat{g}^*}\sum_{m=i}^{j-1}\bigg\{ \pi(m/N)\cdot\Big[\sum_{l=i}^{m-1}\pi(l/N)+\sum_{l=j}^{m-1}\pi(l/N)\Big] \bigg\}.
\end{align*}
Applying the identity in \eqref{eq:doublesumeq} shows that:
\begin{align*}
    \big\vert g^*(i/N, j/N)&-\hat{g}^*(i/N, j/N)\big\vert\\
    &\leq \pi(j/N)\epsilon\Big\{1+\frac{2\rho}{N}\sum_{m=i}^{j-1} \pi(m/N)\Big\} \\
    &\quad+ \frac{2\epsilon\rho^2}{N^2}\pi(j/N)\sum_{m=i}^{N-1}\pi(m/N)\cdot \sum_{l=i}^{j-1}\pi(l/N)\\
    &\quad + \frac{2\rho^2}{N^2} \pi(j/N)\supnorm{g^*-\hat{g}^*}\sum_{m=i}^{N-1}\pi(m/N)\cdot \sum_{l=i}^{j-1}\pi(l/N).
\end{align*}
Repeated application of this ultimately shows the following for any $n > 0$:
\begin{align*}
     \big\vert g^*(i/N, j/N)&-\hat{g}^*(i/N, j/N)\big\vert \\
    &\leq \pi(j/N)\epsilon\bigg\{1 + 2\sum_{l=1}^n\Big[\frac{\rho}{N}\sum_{m=1}^{N-1} \pi(m/N)\Big]^l\bigg\} \\
    &\quad+ 2\pi(j/N)\supnorm{g^* - \hat{g}^*}\cdot\Big[\frac{\rho}{N}\sum_{l=1}^{N-1}\pi(l/N)\Big]^n.
\end{align*}
We can now take $n\to\infty$ due to the regularity condition, combined with Lemma \ref{lemma:errorproof1}, this results in the final statement.
\end{proof}

\subsection{Error analysis for mean batch sojourn time}
\label{sec:5sojournerror}

In Section \ref{sec:5errorproof} we have seen that the numerical algorithm converges and incorporates an error that can be improved by increasing $N$. In this section, we derive how this final error in $g$ propagates in the mean batch sojourn time. We prove the following.

\begin{proposition}
    Let $\hat{\E}[S^B]$ be the approximated expected batch sojourn time, achieved by substituting an approximation $\hat{f}_K$ into \eqref{eq:batchsojourn}. Assume that:
    \begin{align*}
        \big\vert\hat{f}_K(x,y) - f_K(x,y) \big\vert \leq \frac{\pi(x)\pi(y)}{\rho\pi(y) + 1-\rho}\zeta,
    \end{align*}
    for some $\zeta > 0$, then:
    \begin{align}
        \bigg\vert \E[S^B] - \hat{\E}[S^B]\bigg\vert \leq \frac{\E[B]\zeta}{\rho}\big[\exp(\rho) -1\big].
    \end{align}
\end{proposition}
\begin{proof}
    The only part that is approximated is the last integral in \eqref{eq:batchsojourn}. We apply the assumed bound and evaluate the integrals to see:
    \begin{align*}
    \bigg\vert \E[S^B] - \hat{\E}[S^B]\bigg\vert &\leq 
    \int_{u=0}^1 \bigg\{\pi(u)\int_{x=0}^1 \pi(x)\tilde{K}'\bigg({\int_{v = u}^x}^* \pi(v)\dx v \bigg)\\
    &\hspace{2cm}{\int_{y=u}^x}^* \E[B]\zeta \pi(y)\exp\left(\rho{\int_{\xi = y}^x}^* \pi(\xi)\dx \xi\right) \dx y\dx x\bigg\} \dx u\\
    &\leq \E[B]\zeta\int_{u=0}^1 \pi(u)\int_{x=0}^1 \pi(x)\tilde{K}'\bigg({\int_{v = u}^x}^* \pi(v)\dx v\bigg) \\
    &\hspace{3.5cm}\cdot\frac{1}{\rho}\bigg[\exp\left(\rho{\int_{\xi = u}^x}^* \pi(\xi)\dx \xi\right)- 1\bigg]\dx x \dx u.\\
    \intertext{We use the substitution $\omega = {\int_{v=u}^x}^* \pi(v)\dx v$, which removes the dependence on $u$ in the most inner integral. Afterwards, we bound the exponential function by $\exp(\rho)$:}
     \bigg\vert \E[S^B] - \hat{\E}[S^B]\bigg\vert &\leq \frac{\E[B]\zeta}{\rho}\int_{u=0}^1 \pi(u)\int_{\omega=0}^1 \tilde{K}'(\omega)\bigg[\exp\left(\rho\omega\right)- 1\bigg]\dx \omega \dx u\\
    &= \frac{\E[B]\zeta}{\rho}\int_{\omega=0}^1 \tilde{K}'(\omega)\bigg[\exp\left(\rho\omega\right)- 1\bigg]\dx \omega\leq \frac{\E[B]\zeta}{\rho}\big[\exp(\rho) -1\big]. \qedhere
    \end{align*}     
\end{proof}

\section{Miscellaneous results}
\label{app:Misc}
In this appendix, we provide a miscellaneous result that is not directly applicable to the problem at hand, yet it reveals an interesting feature of the integral equation.
\begin{lemma}\label{lemma:unique}
        Any function $f$ satisfying the integral equation \eqref{eq:integraleq} has the following property:
        \begin{align*}
            \int_{u=0}^1\int_{x=0}^1 [&\rho\pi(u) + 1-\rho]f(x,u)\dx x \dx u \\
            &=  \frac{\lambda\mathbb{E}[K]}{2(1-\rho)}\Big(\alpha + \lambda\mathbb{E}[K]\mathbb{E}[B^2] + \mathbb{E}[B]\cdot\frac{\mathbb{E}[K(K-1)]}{\mathbb{E}[K]}\Big).
        \end{align*}
    \end{lemma}
    \begin{proof}
    We first derive some properties of the fixed terms of the integral equation. Let $X_1,X_2$ denote two independent random variables with density $\pi(\cdot)$ and $U_1,U_2$ two independent uniform random variables on $[0,1]$. Additionally, we write $x\in (a,b]$ when $d(a,x) + d(x,b) = d(a,b)$.
    \begin{align*}
        \int_{y=0}^1 \int_{x=0}^1 &\pi(x) {\int_{u=x}^y}^* [\rho\pi(u)+1-\rho]\dx u \dx x \dx y \\
        &= \int_{y=0}^1 \int_{x=0}^1 \pi(x)\cdot\Big(\rho\mathbb{P}(X_1\in (x,y]) + (1-\rho)\mathbb{P}(U_1\in (x,y])\Big) \dx x \dx y\\
        &=\rho\mathbb{P}(X_1\in (X_2,U_2]) + (1-\rho)\mathbb{P}(U_1\in (X_2,U_2]).
    \end{align*}
    Note, first of all, $\mathbb{P}(X_1\in (X_2,U_2]) + \mathbb{P}(X_2\in (X_1,U_2]) = 1$, implying that $\mathbb{P}(X_1\in (X_2,U_2]) = 1/2$ due to the symmetry. A similar relationship holds for $\mathbb{P}(U_1\in (X_2,U_2])$. Additionally:
    \begin{align*}
        \int_{y=0}^1 \pi(y)\int_{x=0}^1 &\pi(x) {\int_{u=x}^y}^* [\rho\pi(u)+1-\rho]\dx u \dx x \dx y \\
        &=\rho\mathbb{P}(X_1\in (X_2,X_3]) + (1-\rho)\mathbb{P}(U_1\in (X_2,X_1]).
    \end{align*}
    The same symmetry arguments again show that both probabilities equal $1/2$.
    Combining these properties we have:
        \begin{align*}
             \int_{x=0}^1&\int_{y=0}^1 [\rho\pi(y) + 1-\rho]f(x,y)\dx y \dx x \\
             &=  \rho\int_{x=0}^1\int_{y=0}^1\ {\int_{u=x}^y}^* [\rho\pi(u) + 1-\rho]\big(\pi(x)f(y,u) + \pi(y)f(x,u)\big)\dx u \dx y \dx x\\
             &\quad +\frac{\lambda\mathbb{E}[K]}{2}\Big(\alpha +\lambda\mathbb{E}[K]\mathbb{E}[B^2] + \mathbb{E}[B]\frac{\mathbb{E}[K(K-1)]}{\mathbb{E}[K]}\Big).
        \end{align*} 
    Now interchanging the integrals, we find:
    \begin{align*}
         \int_{x=0}^1&\int_{y=0}^1 [\rho\pi(y) + 1-\rho]f(x,y)\dx y \dx x \\
        &=  \rho\int_{y=0}^1 \int_{u=0}^1 [\rho\pi(u) + 1-\rho]f(y,u) {\int_{x=u}^y}^*\pi(y)\dx x \dx u \dx y\\
        &\quad+\rho\int_{x=0}^1 \int_{u=0}^1 [\rho\pi(u) + 1-\rho]f(x,u) {\int_{y=u}^x}^*\pi(y)\dx y \dx u \dx x\\
        &\quad +\frac{\lambda\mathbb{E}[K]}{2}\Big(\alpha +\lambda\mathbb{E}[K]\mathbb{E}[B^2] + \mathbb{E}[B]\frac{\mathbb{E}[K(K-1)]}{\mathbb{E}[K]}\Big).
    \end{align*}
    Remark that the first two terms of the right-hand side add up to $\rho$ times the left-hand side, as the most inner integrals add up to an integral from $0$ to $1$, hence:
    \begin{align*}
        (1-\rho)\int_{x=0}^1&\int_{y=0}^1 [\rho\pi(y) + 1-\rho]f(x,y)\dx y \dx x \\
        &= \frac{\lambda\mathbb{E}[K]}{2}\Big(\alpha +\lambda\mathbb{E}[K]\mathbb{E}[B^2] + \mathbb{E}[B]\frac{\mathbb{E}[K(K-1)]}{\mathbb{E}[K]}\Big).
    \end{align*}
    Taking $(1-\rho)$ to the right-hand side now concludes the proof.
    \end{proof}

\end{document}